\documentclass[aos]{imsart}

\usepackage[authoryear]{natbib}
\usepackage{enumitem}
\usepackage{graphicx}
\usepackage{subcaption}
\usepackage{verbatim}
\usepackage{xr}
\RequirePackage[colorlinks,citecolor=blue,urlcolor=blue,breaklinks=true]{hyperref}
\usepackage{breakcites}

\startlocaldefs
\usepackage{pig}
\newcommand{\mitn}{\footnotemark[6]}
\newcommand{\nyun}{\footnotemark[7]}
\endlocaldefs

\begin{document}
\begin{frontmatter}

\title{Optimality and Sub-optimality of PCA I: Spiked Random Matrix Models}
\runtitle{Optimality and Sub-optimality of PCA}

\author{\fnms{Amelia} \snm{Perry}\corref{}\ead[label=e1]{}\thanksref{fa}\thanksref{t1}\protect\mitn}
\and
\author{\fnms{Alexander S.} \snm{Wein}\corref{}\ead[label=e2]{awein@mit.edu}\thanksref{fa}\thanksref{t2}\protect\mitn}
\and
\author{\fnms{Afonso S.} \snm{Bandeira}\ead[label=e3]{bandeira@cims.nyu.edu}\thanksref{t3}\protect\nyun}
\and
\author{\fnms{Ankur} \snm{Moitra}\ead[label=e4]{moitra@mit.edu}\thanksref{t4}\protect\mitn}
\runauthor{A.\ Perry, A.\ S.\ Wein, A.\ S.\ Bandeira \and A.\ Moitra}

\affiliation{Massachusetts Institute of Technology\protect\mitn}
\affiliation{Courant Institute of Mathematical Sciences, New York University\protect\nyun}

\thankstext{fa}{The first two authors contributed equally.}
\thankstext{t1}{This work was supported in part by NSF CAREER Award CCF-1453261 and a grant from the MIT NEC Corporation.}
\thankstext{t2}{This research was conducted with Government support under and awarded by DoD, Air Force Office of Scientific Research, National Defense Science and Engineering Graduate (NDSEG) Fellowship, 32 CFR 168a.}
\thankstext{t3}{A.S.B.\ was supported by NSF Grants DMS-1317308, DMS-1712730, and DMS-1719545.
 Part of this work was done while with the MIT Department of Mathematics.}
\thankstext{t4}{This work was supported in part by NSF CAREER Award CCF-1453261,
NSF Large CCF-1565235, a David and Lucile Packard Fellowship and an Alfred P. Sloan Fellowship.}

\address{Amelia Perry \\ Department of Mathematics \\ Massachusetts Institute of Technology \\ 77 Massachusetts Avenue \\ Cambridge, MA 02139, USA}
\address{Alexander S.\ Wein \\ Department of Mathematics \\ Massachusetts Institute of Technology \\ 77 Massachusetts Avenue \\ Cambridge, MA 02139, USA \\ \printead{e2}}
\address{Afonso S.\ Bandeira \\ Department of Mathematics, and \\ Center for Data Science \\ Courant Institute of \\ \quad Mathematical Sciences \\ New York University \\ 251 Mercer Street \\ New York, NY 10012, USA \\ \printead{e3}}
\address{Ankur Moitra \\ Department of Mathematics, and \\ Computer Science and Artificial \\ \quad Intelligence Laboratory \\ Massachusetts Institute of Technology \\ 77 Massachusetts Avenue \\ Cambridge, MA 02139, USA \\ \printead{e4}}

\begin{abstract}

A central problem of random matrix theory is to understand the eigenvalues of spiked random matrix models,
introduced by Johnstone,
in which a prominent eigenvector (or ``spike'') is planted into a random matrix. These distributions form natural statistical models for principal component analysis (PCA) problems throughout the sciences. Baik, Ben Arous and P\'ech\'e showed that the spiked Wishart ensemble exhibits a sharp phase transition asymptotically: when the spike strength is above a critical threshold, it is possible to detect the presence of a spike based on the top eigenvalue, and below the threshold the top eigenvalue provides no information. Such results form the basis of our understanding of when PCA can detect a low-rank signal in the presence of noise.
However, under structural assumptions on the spike, not all information is necessarily contained in the spectrum.
We study the statistical limits of tests for the presence of a spike, including non-spectral tests. Our results leverage Le Cam's notion of contiguity, and include:

\noindent i) For the Gaussian Wigner ensemble, we show that PCA achieves the optimal detection threshold for certain natural priors for the spike.

\noindent ii) For any non-Gaussian Wigner ensemble, PCA is sub-optimal for detection. However, an efficient variant of PCA achieves the optimal threshold (for natural priors) by pre-transforming the matrix entries.

\noindent iii) For the Gaussian Wishart ensemble, the PCA threshold is optimal for positive spikes (for natural priors) but this is not always the case for negative spikes.

\noindent

\end{abstract}



\end{frontmatter}

\section{Introduction}
\label{sec:intro}

\nocite{bmnn} % manipulate citation order

One of the most common ways to analyze a collection of data is to extract top eigenvectors of a sample covariance matrix that represent directions of largest variance, often referred to as principal component analysis (PCA). Starting from the work of Karl Pearson, this technique has been a mainstay in statistics and throughout the sciences for more than a century. For instance, genome-wide association studies construct a correlation matrix of expression levels, whereby PCA is able to identify collections of genes that work together. PCA is also used in economics to extract macroeconomic trends and to predict yields and volatility \citep{macro1,forni,stock-watson,macro2}, and in network science to find well-connected communities \citep{mcsherry}. More broadly, it underlies much of exploratory data analysis, dimensionality reduction, and visualization.

Classical random matrix theory provides a suite of tools to characterize the behavior of the eigenvalues of various random matrix models in high-dimensional settings. Nevertheless, most of these works can be thought of as focusing on a pure noise model \citep{AGZ-book,BS-book,Tao-book} where there is not necessarily any low-rank structure to extract. A direction initiated by \citet{J-spk} has brought this powerful theory closer to statistical questions by introducing \emph{spiked models} that are of the form ``signal + noise.'' Such models have yielded fundamental new insights on the behaviors of several methods such as principal component analysis (PCA) \citep{JL-sparse-pca,Paul,Nadler}, sparse PCA \citep{AW-sparse-pca,VL-sparse,BR-opt,Ma-sparse,SSM-sparse,CMW-sparse,Birnbaum-sparse,DM-sparse-pca,KNV-sparse-pca}, and synchronization algorithms
\citep{sin11,boumal2014cramer,angular-tightness,boumal}. More precisely, given a true signal in the form of an $n$-dimensional unit vector $x$ called the \emph{spike}, we can define two natural spiked random matrix ensembles as follows:
\begin{itemize}[leftmargin=*]
\item Spiked (Gaussian) Wishart: observe the sample covariance $Y = \frac{1}{N} X X^\top$, where $X$ is an $n \times N$ matrix with columns drawn \iid from $\cN(0,I_n + \beta x x^\top)$, in the high-dimensional setting where the sample count $N$ and dimension $n$ scale proportionally as $n/N \to \gamma$. We allow $\beta \in [-1,\infty)$.
\item Spiked Wigner: observe $Y = \lambda x x^\top + \frac{1}{\sqrt{n}} W$, where $W$ is an $n \times n$ random symmetric matrix with entries drawn \iid (up to symmetry) from a fixed distribution of mean $0$ and variance $1$.
\end{itemize}
\noindent We adopt a Bayesian viewpoint, taking the spike $x$ to be drawn from an arbitrary but known prior. This enables our approach to address structural assumptions on the spike, such as sparsity or an entrywise constraint to $\{\pm 1/\sqrt{n}\}$, to model variants of sparse PCA or community detection \citep{dam}.

% justify the models
The Wishart model describes the sample covariance of high-dimensional data. The Gaussian Wigner distribution arises from the Wishart as a particular small-$\gamma$ limit \citep{jo-testing-spiked}. The spiked Wigner model also describes various inference problems where pairwise measurements are observed between $n$ entities; this captures, for instance, Gaussian variants of community detection \citep{dam} and $\mathbb{Z}/2$ synchronization \citep{sdp-phase}.

We will refer to the parameter $\beta$ or $\lambda$ as the signal-to-noise ratio (SNR).

In each of the above models, we study the following statistical questions:

\begin{itemize}[leftmargin=*]
\item \emph{Detection}: For what values of the SNR is it possible to consistently test (with probability $1-o(1)$ as $n \to \infty$) between a random matrix drawn from the spiked distribution and one from the unspiked distribution?

\item \emph{Recovery}: For what values of the SNR can an estimator $\hat x$ achieve correlation with the true spike $x$ that is bounded above zero as $n \to \infty$?
\end{itemize}

\noindent We primarily study the detection problem, which has previously been explored in various statistical models \citep{Donoho-Jin,CJL,SN-size,Ingster-detection,ACCD,ACCP,aCBL,BI-detect,SN-max,av-dense,av-sparse}.

The spiked random matrix models above all enjoy a sharp characterization of the performance of PCA through random matrix theory.
In the complex Wishart case, the seminal work of \citet*{bbp} showed that when $\beta > \sqrt{\gamma}$ an isolated eigenvalue emerges from the Marchenko--Pastur-distributed bulk. Later \citet*{baik-silverstein} established this result in the real Wishart case. In the Wigner case, the top eigenvalue separates from the semicircular bulk when $\lambda > 1$ \citep{peche,fp,cdf,wig-spk}. Each result establishes a sharp {\em spectral threshold} at which PCA (top eigenvalue) is able to solve the detection problem for the respective spiked random matrix model. Moreover, it is known that above this threshold, the top eigenvector correlates nontrivially with $x$, while the correlation concentrates about zero below the threshold. Despite detailed research on the spectral properties of spiked random matrix models, much less is known about the more general statistical question: can any hypothesis test consistently detect the presence of a spike below the threshold where PCA succeeds? Our main goal in this paper is to address this question in each of the models above, and as we will see, the answer varies considerably across them. Our results shed new light on how much of the accessible information about $x$ is {\em not} captured by the top eigenvalue, or even by the full spectrum.

Several recent works have examined this question. \citet{sphericity} study the spiked Wishart model where $x$ is an arbitrary unknown unit vector (which, by rotational symmetry, is equivalent to drawing $x$ from the uniform prior on the unit sphere). They identify the optimal hypothesis testing power (between spiked and unspiked) and in particular show that there is no test to consistently detect the presence of a spike below the spectral threshold. Even more recent work \citep{onatski-multispike,all-eigs,rare-weak} elaborates on this point in other spiked models. In the Gaussian Wigner model it has been established by \citet{mrz-sphere} and \citet{jo-testing-spiked} that detection is impossible below the spectral threshold, and the former used techniques similar to those of the present paper, which are not fundamentally limited to spherically symmetric models; indeed, these techniques were applied to sparse PCA by \citet{bmvvx}.

In another line of work, several papers have studied recovery in structured spiked random matrix models through approximate message passing \citep{amp-cs,bm,jm}, Guerra interpolation \citep{guerra}, and other tools originating from statistical physics. These results span sparse PCA \citep{sparse-pca-amp,phase-sparse-pca}, non-negative PCA \citep{nonneg-pca}, cone-constrained PCA \citep{cone-constrained}, and general structured PCA \citep{RF-amp,LKZ-amp,dam,mi,mi-proof,lelarge-limits-lowrank}. Methods based on approximate message passing typically exhibit the same threshold as PCA, but above the threshold they obtain better (and often optimal) estimates of the spike. In many cases, the above techniques give the asymptotic minimum mean square error (MMSE) and, in particular, identify the threshold for nontrivial recovery. However, they do not typically address the detection problem (although we expect the detection and recovery thresholds to match), and they tend to be restricted to \iid priors.

We develop a number of general-purpose tools for proving both upper and lower bounds on detection. We defer the precise statement of our results in each model to their respective sections, but for now we highlight some of our main results:

\begin{itemize}[leftmargin=*]

\item In the Gaussian Wigner model, detection is impossible below the spectral threshold ($\lambda=1$) for priors such as the spherical prior\footnote{$x$ is uniform on the unit sphere in $\mathbb{R}^n$} (Corollary~\ref{cor:sphere-prior}), the Rademacher prior\footnote{$x$ is \iid uniform on $\{\pm 1/\sqrt{n}\}$} (Corollary~\ref{cor:pmone-wigner}), and any sufficiently subgaussian prior (Theorem~\ref{thm:subg-iid}). We also study sparse Rademacher priors\footnote{$x$ is \iid where each entry is 0 with probability $1-\rho$ and otherwise uniform on $\{\pm \frac{1}{\sqrt{\rho n}}\}$}, where we see that the spectral threshold is sometimes optimal and sometimes sub-optimal depending on the sparsity level (Section~\ref{sec:sparse-rad}).

\item In the Wigner model with non-Gaussian noise, the spectral threshold is never optimal (subject to mild conditions): there is an entrywise pre-transformation on the observed matrix that exploits the non-Gaussianity of the noise and strictly improves the performance of PCA (Theorem~\ref{thm:nong-upper}). This method was first described by \citet{LKZ-amp} and we give a rigorous analysis. Moreover we provide a lower bound (Theorem~\ref{thm:nongauss-lower}) which often matches this upper bound.

\item In the Wishart model, the PCA threshold is optimal for the spherical prior, both for positive and negative $\beta$. For the Rademacher prior, PCA is optimal for all positive $\beta$; however, in the less-studied case of negative $\beta$, an inefficient algorithm succeeds below the spectral threshold when $\gamma$ is sufficiently large. This exposes a new statistical phase transition that seems to be previously unexplored. For the sparse Rademacher prior, PCA can be sub-optimal in both the positive and negative $\beta$ regimes, but it is always optimal for sufficiently large positive $\beta$.

\end{itemize}

% emphasize threshold-optimal
We emphasize that when we say PCA is optimal, we refer only to the threshold for consistent detection. In essentially all cases we consider (except the spherical prior), the top eigenvector has sub-optimal estimation error above the threshold; optimal error is often given by an approximate message passing algorithm such as that of \citet{dam}. Furthermore, PCA does not achieve optimal hypothesis testing power below the threshold, and in fact no method based on a finite number of top eigenvalues can be optimal in this sense \citep{sphericity,onatski-multispike,jo-testing-spiked,all-eigs}.

All our lower bounds follow a similar pattern and are based on the notion of \emph{contiguity} introduced by \citet{lecam}. On a technical level, we show that a particular second moment is bounded which (as is standard in contiguity arguments) implies that the spiked distribution cannot be consistently distinguished (with $o(1)$ error as $n \to \infty$) from the corresponding unspiked distribution. We develop general tools for controlling the second moment based on subgaussianity and large deviations theory that apply across a range of models and a range of different priors on $x$. 

While bounds on the second moment do not \emph{a priori} imply anything about the recovery problem, it follows from results of \citet{bmvvx} that many of our non-detection results imply the corresponding non-recovery results. The value of the second moment also yields bounds on hypothesis testing power (see Proposition~\ref{prop:hyptest}).

\footnotetext{Recall that Type I error refers to the probability of reporting a spike when none exists (false positives), while Type II error is the probability of reporting no spike when one does exist (false negatives).}

% computational gaps
Our work fits into an emerging theme in statistics: we indicate several scenarios when PCA is sub-optimal but the only known tests that beat it are computationally inefficient. Such computational vs.\ statistical gaps have received considerable recent attention (e.g.\ \citet{Ber-Rig,Ma-Wu}), often in connection with sparsity. We provide evidence for a new such gap in the negatively-spiked Wishart model with the Rademacher prior, offering an example where sparsity is not present.

\subsection*{Outline}
In Section~\ref{sec:contig} we give preliminaries on contiguity and the second moment method. In Section~\ref{sec:gwig} we study the spiked Gaussian Wigner model, in Section~\ref{sec:nong-wig} we study the spiked non-Gaussian Wigner model, and in Section~\ref{sec:wish} we study the spiked Wishart model. Some proofs are deferred to appendices in \ref{supp} (included in this document).

% CONTIG
\section{Contiguity and the second moment method}
\label{sec:contig}

Contiguity and related ideas will play a crucial role in this paper. First introduced by \citet{lecam}, contiguity is a central concept in the asymptotic theory of statistical experiments, and has found many applications throughout probability and statistics. Our work builds on a history of using contiguity and related tools such as the \emph{small subgraph conditioning method} to establish fundamental results about random graphs (e.g.\ \citet{rw-ham,janson,1-fact}; see \citet{wor-survey} for a survey) and impossibility results for detecting community structure in the sparse stochastic block model \citep{mns, bmnn}. Contiguity is formally defined as follows:

% defn of contiguity
\begin{definition}[\citet{lecam}]
Let distributions $P_n$, $Q_n$ be defined on the measurable space $(\Omega_n,\F_n)$. We say that the sequence $Q_n$ is \defn{contiguous} to $P_n$, and write $Q_n \contig P_n$, if for any sequence $A_n$ of events,
$$\lim_{n \to \infty} P_n(A_n) = 0 \;\implies\; \lim_{n \to \infty} Q_n(A_n) = 0.$$
\end{definition}

\noindent Contiguity readily implies that the distributions $P_n$ and $Q_n$ cannot be consistently distinguished (given a single sample) in the following sense:
\begin{observation}
\label{obs:contig}
If $Q_n \contig P_n$ then there is no hypothesis test of the alternative $Q_n$ against the null $P_n$
with $\prob{\text{type I error}} + \prob{\text{type II error}} = o(1)$.
\end{observation}

\noindent Note that $Q_n \contig P_n$ and $P_n \contig Q_n$ are not equivalent, but either of them implies non-distinguishability. Also, showing that two (sequences of) distributions are contiguous does not rule out the existence of a test that distinguishes between them with constant error probability (better than random guessing). In fact, such tests do exist for the spiked Wigner and Wishart models, for instance by thresholding the trace of the matrix; optimal tests are discussed by \citet{sphericity} and \citet{jo-testing-spiked}.

% second moment
Our goal in this paper is to show thresholds below which spiked and unspiked random matrix models are contiguous. We will do this through computing a particular second moment, related to the $\chi^2$-divergence as $1+\chi^2(Q_n || P_n)$, through a classical form of the second moment method:
\begin{lemma}
\label{lem:sec}
Let $\{P_n\}$ and $\{Q_n\}$ be two sequences of distributions on $(\Omega_n,\F_n)$. If the second moment
$$\Ex_{P_n} \left[\left( \dd[Q_n]{P_n} \right)^2\right]$$
exists and remains bounded as $n \to \infty$, then $Q_n \contig P_n$.
\end{lemma}

All of the contiguity results in this paper will follow through Lemma~\ref{lem:sec} and its conditional variant below. The roles of $P_n$ and $Q_n$ are not symmetric, and we will always take $P_n$ to be the unspiked distribution and take $Q_n$ to be the spiked distribution, as the second moment is more tractable to compute in this direction. We include the proof of Lemma~\ref{lem:sec} here for completeness:

\begin{proof}
Let $\{A_n\}$ be a sequence of events. Using Cauchy--Schwarz,
\begin{align*}
Q_n(A_n) 
= \int_{A_n} \dd[Q_n]{P_n} \,\dee P_n \le \sqrt{\int_{A_n} \left( \dd[Q_n]{P_n} \right)^2 \,\dee P_n} \;\cdot\; \sqrt{\int_{A_n} \,\dee P_n}. %\\
\end{align*}
The first factor on the right-hand side is bounded; so if $P_n(A_n) \to 0$ then also $Q_n(A_n) \to 0$.
\end{proof}

% conditioning away from bad events
There will be times when the above second moment is unbounded but we are still able to prove contiguity using a modified second moment that conditions away from rare `bad' events that would otherwise dominate the second moment. This idea has appeared previously \citep{av-dense,av-sparse,bmnn,bmvvx}.
\begin{lemma}
\label{lem:cond}
Let $\omega_n$ be an event that occurs with probability $1-o(1)$ under $Q_n$.
Let $\tilde Q_n$ be the conditional distribution of $Q_n$ given $\omega_n$. If the modified second moment $\EE_{P_n}\left[(\ddflat[\tilde Q_n]{P_n} )^2 \right]$ remains bounded as $n \to \infty$, then $Q_n \contig P_n$.
\end{lemma}
\begin{proof}
%By Lemma~\ref{lem:sec} we have $\tilde Q_n \contig P_n$. This implies $Q_n \contig P_n$ because $Q_n \contig \tilde Q_n$.
By Lemma~\ref{lem:sec} we have $\tilde Q_n \contig P_n$. As $Q_n \contig \tilde Q_n$ we have $Q_n \contig P_n$.
\end{proof}

Moreover, given a value of the second moment, we are able to obtain bounds on the tradeoff between type I and type II error in hypothesis testing, which are valid non-asymptotically:
\begin{proposition}\label{prop:hyptest}
Consider a hypothesis test of a simple alternative $Q$ against a simple null $P$. Let $\alpha$ be the probability of type I error, and $\beta$ the probability of type II error. Regardless of the test, we must have
$$\frac{(1-\beta)^2}{\alpha} + \frac{\beta^2}{(1-\alpha)} \le \Ex_P \left(\dd[Q]{P}\right)^2,$$
assuming the right-hand side is defined and finite. Furthermore, this bound is tight: for any $\alpha,\beta \in (0,1)$ there exist $P,Q$, and a test for which equality holds.
\end{proposition}
%\noindent This tradeoff is illustrated in Figures~\ref{fig:hyp1} and~\ref{fig:hyp2} in the introduction.

\begin{proof}
Let $A$ be the event that the test selects the alternative $Q$, and let $\bar A$ be its complement.
\begin{align*}
\Ex_P\left(\dd[Q]{P}\right)^2 &= \int \dd[Q]{P} \,\dee Q = \int_A \dd[Q]{P} \,\dee Q + \int_{\bar A} \dd[Q]{P} \,\dee Q \\
&\ge \frac{\left(\int_A \dee Q\right)^2}{\int_A (\dee P/ \dee Q) \,\dee Q} + \frac{\left(\int_{\bar A} \dee Q\right)^2}{\int_{\bar A} (\dee P/\dee Q) \,\dee Q} = \frac{(1-\beta)^2}{\alpha} + \frac{\beta^2}{(1-\alpha)	}
\end{align*}
where the inequality follows from Cauchy--Schwarz. The following example shows tightness: let $P = \mathrm{Bernoulli}(\alpha)$ and let $Q = \mathrm{Bernoulli}(1-\beta)$. On input $0$, the test chooses $P$, and on input $1$, it chooses $Q$.
\end{proof}

% Non-recovery
Although contiguity is a statement about non-detection rather than non-recovery, our results also have implications for non-recovery. In general, the detection problem and recovery problem can have different thresholds, but such counterexamples are often unnatural. For a wide class of problems with additive Gaussian noise, the results of \citet{bmvvx} imply that if the second moment from above is bounded then nontrivial recovery is impossible. This result applies to the Gaussian Wigner model and the positively-spiked ($\beta > 0$) Wishart model\footnote{For the Wishart case, consider the asymmetric $n \times N$ matrix of samples, which can be equivalently written as $\sqrt{\beta} x u^\top + W$ where $u \sim \cN(0,I_N)$ and $W$ is \iid $\cN(0,1)$.}, and so our non-detection results immediately imply non-recovery results in those settings.

% GAUSSIAN WIGNER
\section{Gaussian Wigner models}
\label{sec:gwig}

\subsection{Main results}\label{sec:gwig-intro}

We define the spiked Gaussian Wigner model:
\begin{definition}
\label{def:prior}
A \emph{spike prior} is a family of distributions $\cX = \{\cX_n\}$, where $\cX_n$ is a distribution over $\RR^n$. We require our priors to be normalized so that $x^{(n)}$ drawn from $\cX_n$ has $\|x^{(n)}\| \to 1$ (in probability) as $n \to \infty$.
\end{definition}

\begin{definition}
For $\lambda \ge 0$ and a spike prior $\cX$, we define the spiked Gaussian Wigner model $\GWig(\lambda,\cX)$ as follows. We first draw a spike $x \in \RR^n$ from the prior $\cX_n$. Then we reveal
$$Y = \lambda xx^\top + \frac{1}{\sqrt n} W$$
where $W$ is drawn from the $n \times n$ $\GOE$ (Gaussian orthogonal ensemble), i.e.\ $W$ is a random symmetric matrix with off-diagonal entries $\mathcal{N}(0,1)$, diagonal entries $\mathcal{N}(0,2)$, and all entries independent (except for symmetry $W_{ij} = W_{ji}$). We denote the unspiked model ($\lambda = 0$) by $\GWig(0)$.
\end{definition}

% state spectral results
It is well known that this model admits the following spectral behavior.
\begin{theorem}[\citet{fp,nong-eigv1}]
Let $Y$ be drawn from $\GWig(\lambda,\cX)$ with any spike prior $\cX$ supported on unit vectors ($\|x\|=1$).
\begin{itemize}
\item If $\lambda \le 1$, the top eigenvalue of $Y$ converges almost surely to $2$ as $n \to \infty$, and the top (unit-norm) eigenvector $v$ has trivial correlation with the spike: $\langle v,x \rangle^2 \to 0$ almost surely.
\item If $\lambda > 1$, the top eigenvalue converges almost surely to $\lambda + 1/\lambda > 2$, and $v$ estimates the spike nontrivially: $\langle v,x \rangle^2 \to 1 - 1/\lambda^2$ almost surely.
\end{itemize}
\end{theorem}
\noindent It follows that if $\|x\| \to 1$ in probability then the above convergence holds in probability (instead of almost surely). Thus PCA solves the detection and recovery problems precisely when $\lambda > 1$. In the critical case $\lambda = 1$ or near-critical case $\lambda \to 1$, there is also a test to consistently distinguish the spiked and unspiked models based on their spectra \citep{jo-testing-spiked}; see Appendix~\ref{app:critical} for details. Our goal is now to investigate whether detection is possible when $\lambda < 1$.

As a starting point, we compute the second moment of Lemma~\ref{lem:sec}:
\begin{proposition}
\label{prop:2nd-comp}
Let $\lambda \ge 0$ and let $\cX$ be a spike prior. Let $Q_n = \GWig_n(\lambda,\cX)$ and $P_n = \GWig_n(0)$. Let $x$ and $x'$ be independently drawn from $\cX_n$. Then
$$\Ex_{P_n}\left(\dd[Q_n]{P_n}\right)^2 = \Ex_{x,x'} \exp\left(\frac{n \lambda^2}{2} \langle x,x' \rangle^2\right).$$
\end{proposition}
\noindent We defer the proof of this proposition until Section~\ref{sec:2nd-comp}. 
For specific choices of the prior $\cX$, our goal will be to show that if $\lambda$ is below some critical $\lambda^*_\cX$, this second moment is bounded as $n \to \infty$ (implying that detection is impossible). We will specifically consider the following types of priors.
\begin{definition}
Let $\cXs$ denote the spherical prior: $x$ is a uniformly random unit vector in $\RR^n$.
\end{definition}
\noindent By spherical symmetry, the spherical prior is equivalent to asking for a test that works for \emph{any} unit-norm spike (i.e.\ no prior). Without loss of generality, any test for the spherical prior depends only on the spectrum.
\begin{definition}
If $\pi$ is a distribution on $\RR$ with $\EE[\pi] = 0$ and $\mathrm{Var}[\pi] = 1$, let $\mathrm{iid}(\pi/\sqrt{n})$ denote the spike prior that samples each coordinate of $x$ independently from $\pi/\sqrt{n}$.
\end{definition}

% subgaussian method
We will give two general techniques for showing contiguity for various priors. We call the first method the \emph{subgaussian method}, and it is presented in Section~\ref{sec:subg-method}. The idea is that if the correlation $\langle x,x' \rangle$ between two independent draws from the prior is sufficiently subgaussian, this implies strong tail bounds on $\langle x,x' \rangle$ which can be integrated to show that the second moment is bounded. For instance, this gives results in the case of an \iid prior where the entrywise distribution $\pi$ is subgaussian.

% conditioning method
In Section~\ref{sec:cond-method} we present our second method, the \emph{conditioning method}, which uses the conditional second moment method and can improve upon the subgaussian method is some cases. It only applies to finitely-supported \iid priors and is based on a result from \citet{bmnn}.

% specific results
For certain natural priors, we are able to show contiguity for all $\lambda < 1$, matching the spectral threshold. In particular, this holds for the spherical prior $\cXs$ (Corollary~\ref{cor:sphere-prior}), the \iid Gaussian prior $\IID(\cN(0,1/n))$ (Corollary~\ref{cor:gauss-prior}), the \iid Rademacher prior $\IID(\pm 1/\sqrt{n})$ (Corollary~\ref{cor:pmone-wigner}), and more generally for $\IID(\pi/\sqrt{n})$ where $\pi$ is strictly subgaussian (Theorem~\ref{thm:subg-iid}).

Not all priors are as well behaved as those above. In Section~\ref{sec:sparse-rad} we discuss the sparse Rademacher prior, where we see that the PCA threshold is not always optimal.

In Section~\ref{sec:compare} we show that (in some sense) similar priors have the same detection threshold (Proposition~\ref{prop:compare}). One corollary (Corollary~\ref{cor:contig-eigs}) is that regardless of the prior, no test based only on the eigenvalues can succeed below the $\lambda = 1$ threshold.
% this is basically just a corollary of spherical prior

Our results often yield the limit value of the second moment and therefore imply asymptotic bounds on hypothesis testing via Proposition~\ref{prop:hyptest}; see Appendix~\ref{app:hyptest} for details.

\subsection{Second moment computation}
\label{sec:2nd-comp}

We begin by computing the second moment $\EE_{P_n}[(\ddflat[Q_n]{P_n})^2]$ where $Q_n = \GWig_n(\lambda,\cX)$ and $P_n = \GWig_n(0)$.
First we simplify the likelihood ratio:
\begin{align*}
\dd[Q_n]{P_n} &= \frac{\Ex_{x \sim \cX_n} \exp(-\frac{n}{4} \langle Y-\lambda x x^\top, Y-\lambda x x^\top \rangle)}{\exp(-\frac{n}{4} \langle Y,Y \rangle)} \\
&= \Ex_{x \sim \cX_n} \exp\left( \frac{\lambda n}{2} \langle Y, x x^\top \rangle - \frac{n \lambda^2}{4} \langle x x^\top, x x^\top \rangle \right).
\end{align*}
Now passing to the second moment:
\begin{align*}
\Ex_{P_n}\left(\dd[Q_n]{P_n}\right)^2
&= \Ex_{x,x' \sim \cX_n}\, \Ex_{Y \sim P_n} \exp\left( \frac{\lambda n}{2} \langle Y, x x^\top + x' x'^\top \rangle \right. \\
&\qquad \left. - \frac{n \lambda^2}{4} \left( \langle x x^\top, x x^\top \rangle + \langle x' x'^\top, x' x'^\top \rangle \right) \right),
\intertext{where $x$ and $x'$ are drawn independently from $\cX_n$. We now simplify the Gaussian moment-generating function over the randomness of $Y$, and cancel terms, to arrive at the expression}
&= \Ex_{x,x'} \exp\left(\frac{n \lambda^2}{2} \langle x,x' \rangle^2\right),
\end{align*}
which proves Proposition~\ref{prop:2nd-comp}.

\subsection{The subgaussian method}
\label{sec:subg-method}

In this section we give a general method for controlling the second moment $\EE_{x,x'} \exp\left(\frac{n \lambda^2}{2} \langle x,x' \rangle^2\right)$.
We will need the concept of a subgaussian random variable.
\begin{definition}\label{def:subg}
A $\mathbb{R}^n$-valued random variable $X$ is $\sigma^2$-\emph{subgaussian} if $\EE[X] = 0$ and, for all $v \in \RR^n$,
$\EE \exp(\langle v,X \rangle) \le \exp(\sigma^2 \|v\|^2 / 2 )$.
\end{definition}

\noindent The most general form of the subgaussian method is the following.

\begin{proposition}
\label{prop:subg}
Let $\cX$ be any spike prior. Let $P_n = \GWig_n(0)$ and $Q_n = \GWig_n(\lambda,\cX)$. With $x$ and $x'$ drawn independently from $\cX_n$, suppose $\langle x,x' \rangle$ is $(\sigma^2/n)$-subgaussian for some constant $\sigma$. If $\lambda < 1/\sigma$ then $\EE_{x,x'} \exp\left(\frac{n\lambda^2}{2}\langle x,x' \rangle^2\right)$ is bounded and so $Q_n \contig P_n$.
\end{proposition}
\begin{proof}
Using the well-known subgaussian tail bound $\prob{|\langle x,x' \rangle| \ge t} \le 2 \exp\left(-n t^2/2\sigma^2\right)$, we have
\begin{align*}
\Ex_{x,x'} \exp\left(\frac{n\lambda^2}{2} \langle x,x' \rangle^2\right) &= \int_0^\infty \problr{\exp\left(\frac{n \lambda^2}{2} \langle x, x' \rangle^2\right) \ge u} \, \dee u \\
&= \int_0^\infty \problr{|\langle x, x' \rangle| \ge \sqrt{\frac{2 \log u}{n \lambda^2}}} \, \dee u \\
&\le \int_0^\infty 2 u^{-1/\sigma^2 \lambda^2} \, \dee u
\end{align*}
which is finite (uniformly in $n$) provided $\lambda < 1/\sigma$.
\end{proof}

We next show that it is sufficient for the prior itself to be (multivariate) subgaussian.

\begin{proposition}
\label{prop:subg-total}
Let $P_n = \GWig_n(0)$ and $Q_n = \GWig_n(\lambda,\cX)$. Suppose $\cX_n$ is $(\sigma^2/n)$-subgaussian. If $\lambda < 1/\sigma$ then $Q_n \contig P_n$.
\end{proposition}
\begin{proof}
Let $\delta > 0$. We use the conditional second moment method (Lemma~\ref{lem:cond}), taking $\tilde \cX_n$ to be the conditional distribution of $\cX_n$ given the $(1-o(1))$-probability event $\|x\| \le 1+\delta$. With $\tilde Q_n = \GWig_n(\lambda,\tilde \cX)$, the conditional second moment $\EE_{P_n}(\ddflat[\tilde Q_n]{P_n})^2$ is (by Proposition~\ref{prop:2nd-comp})
$$\Ex_{x,x' \sim \tilde\cX} \exp\left(\frac{n\lambda^2}{2} \langle x,x' \rangle^2\right) \le (1+o(1)) \Ex_{x \sim \cX,\, x' \sim \tilde\cX} \exp\left(\frac{n\lambda^2}{2} \langle x,x' \rangle^2\right).$$
With $x \sim \cX$ and $x' \sim \tilde \cX$, we have that $\langle x,x' \rangle$ is $(\sigma^2(1+\delta)^2/n)$-subgaussian because for any $v \in \mathbb{R}$,
$$\Ex_{x \sim \cX,\, x' \sim \tilde \cX} \exp(v \langle x,x' \rangle)
\le \Ex_{x' \sim \tilde \cX} \exp(\sigma^2 v^2 \|x'\|^2/2n) \le \exp(\sigma^2 v^2 (1+\delta)^2/2n).$$
Choosing $\delta$ small enough so that $\lambda < 1/(\sigma (1+\delta))$, the result now follows from Proposition~\ref{prop:subg}.
\end{proof}

Specializing to \iid priors, it is sufficient for the distribution of each entry to be subgaussian. In this case we can also compute the limit value of the (conditional) second moment.

\begin{theorem}[subgaussian method for \iid priors]
\label{thm:subg-iid}
Let $\pi$ be a mean-zero unit-variance distribution on $\RR$ and let $\cX = \mathrm{iid}(\pi/\sqrt{n})$. Let $P_n = \GWig_n(0)$, $Q_n = \GWig_n(\lambda,\cX)$, and $\tilde Q_n$ as in the proof of Proposition~\ref{prop:subg-total}. Suppose $\pi$ is $\sigma^2$-subgaussian. If $\lambda < \frac{1}{\sigma}$ then 
$\lim_{n\to\infty} \EE_{P_n}(\ddflat[\tilde Q_n]{P_n})^2 = (1-\lambda^2)^{-1/2} < \infty$ and so $Q_n \contig P_n$.
\end{theorem}

\begin{proof}
Since $\pi$ is $\sigma^2$-subgaussian, it follows easily from the definition that $\cX_n$ is $(\sigma^2/n)$-subgaussian and so contiguity follows from Proposition~\ref{prop:subg-total}. To compute the limit value, by the central limit theorem we have that for $x,x' \sim \cX$, $\sqrt{n} \langle x,x' \rangle$ converges in distribution to $\cN(0,1)$. The same holds for $x,x' \sim \tilde \cX$. By the continuous mapping theorem applied to $g(z) = \exp\left(\lambda^2 z^2/2\right)$, we also get convergence in distribution
$\exp\left(n \lambda^2 \langle x, x' \rangle^2 / 2 \right) \xrightarrow{d} \exp\left( \lambda^2 \chi_1^2 / 2 \right)$.
The convergence in expectation $\mathbb{E}_{x,x' \sim \tilde\cX}\exp\left(n \lambda^2 \langle x, x' \rangle^2/2 \right) \rightarrow \mathbb{E}\exp\left( \lambda^2 \chi_1^2 / 2 \right) = (1-\lambda^2)^{-1/2}$
follows since the sequence $\exp\left(n \lambda^2 \langle x, x' \rangle^2 / 2 \right)$ is uniformly integrable; this is clear from the final step of the proof of Proposition~\ref{prop:subg} (which has no dependence on $n$).
\end{proof}

Since $\mathrm{Var}[\pi] = 1$, $\pi$ cannot be $\sigma^2$-subgaussian with $\sigma < 1$. If $\pi$ is $1$-subgaussian (``strictly subgaussian'') then Theorem~\ref{thm:subg-iid} gives a tight result, matching the spectral threshold. For instance, the standard Gaussian distribution is 1-subgaussian, so we have the following.

\begin{corollary}
\label{cor:gauss-prior}
If $\lambda < 1$ then $\GWig(\lambda,\IID(\cN(0,1/n))) \contig \GWig(0)$.
\end{corollary}

\noindent Note that the \iid Gaussian prior is very similar to the spherical prior; in Section~\ref{sec:compare} we show how to transfer the proof to the spherical prior.

\subsection{Application: the Rademacher prior}

If $\pi$ is a Rademacher random variable (uniform on $\{\pm 1\}$) then $\IID(\pi/\sqrt{n})$ is the \emph{Rademacher prior}, which we abbreviate as $\IID(\pm 1/\sqrt{n})$. This case of the Gaussian Wigner model has been studied by \citet{dam} and \citet{sdp-phase} as a Gaussian model for community detection and $\mathbb{Z}/2$ synchronization. The former proves that the spectral threshold $\lambda = 1$ is precisely the threshold above which nontrivial recovery of the signal is possible. We further show contiguity below this $\lambda = 1$ threshold (which, recall, is not implied by non-recovery).

\begin{corollary}\label{cor:pmone-wigner}
If $\lambda < 1$ then $\GWig(\lambda,\IID(\pm 1/\sqrt{n})) \contig \GWig(0)$.
\end{corollary}
\begin{proof}
The Rademacher distribution is 1-subgaussian by Hoeffding's lemma, so the proof follows from Theorem~\ref{thm:subg-iid}.
\end{proof}

Perhaps it is surprising that the spectral threshold is optimal for the Rademacher prior because it suggests that there is no way to exploit the $\pm 1$ structure. However, PCA is only optimal in terms of the threshold and not in terms of error in recovering the spike once $\lambda > 1$. An efficient estimator that asymptotically minimizes the mean squared error is the approximate message passing algorithm of \citet{dam}.

\subsection{Comparison of similar priors}
\label{sec:compare}

We show that two similar priors have the same contiguity threshold, in the following sense.
\begin{proposition}
\label{prop:compare}
Let $\lambda^* \ge 0$. Let $\cX$ and $\mathcal{Y}$ be spike priors. Suppose that $x \sim \cX_n$ and $y \sim \mathcal{Y}_n$ can be coupled such that $y = \alpha x$ where $\alpha = \alpha_n$ is a random variable with $\alpha_n \to 1$ in probability as $n\to\infty$. Suppose that for each $\lambda < \lambda^*$, the second moment $\EE_{x,x' \sim \cX} \exp\left(\frac{n\lambda^2}{2}\langle x,x' \rangle^2\right)$ remains bounded as $n \to \infty$. Then for any $\lambda < \lambda^*$, $\GWig(\lambda,\mathcal{Y}) \contig \GWig(0)$.
\end{proposition}

\begin{proof}
Let $\lambda < \lambda^*$ and $\delta > 0$. Let $\tilde{\mathcal{Y}}$ be the conditional distribution of $\mathcal{Y}$ given the $(1-o(1))$-probability event $\alpha \le 1+\delta$. Letting $\tilde Q_n = \GWig(\lambda,\tilde{\mathcal{Y}})$ and $P_n = \GWig(0)$, we have
\begin{align*}
\dd[\tilde Q_n]{P_n} &= \Ex_{y,y' \sim \tilde{\mathcal{Y}}} \exp\left(\frac{n\lambda^2}{2}\langle y,y' \rangle^2\right) \\
&= (1+o(1))\Ex_{x,x' \sim \cX} \one_{\alpha \le 1+\delta}\, \one_{\alpha' \le 1+\delta} \exp\left(\frac{n\lambda^2}{2} (\alpha \alpha')^2 \langle x,x' \rangle^2\right) \\
&\le (1+o(1))\Ex_{x,x' \sim \cX} \exp\left(\frac{n\lambda^2}{2} (1+\delta)^4 \langle x,x' \rangle^2\right)
\end{align*}
which is bounded provided we choose $\delta$ small enough so that $\lambda(1+\delta)^2 < \lambda^*$. The result now follows from the conditional second moment method (Lemma~\ref{lem:cond}).
\end{proof}

We can now show that the spectral threshold is optimal for the spherical prior (uniform on the unit sphere) by comparison to the \iid Gaussian prior; this result was obtained previously by \citet{mrz-sphere,jo-testing-spiked}.
\begin{corollary}\label{cor:sphere-prior}
If $\lambda < 1$ then $\GWig(\lambda,\cXs) \contig \GWig(0)$.
\end{corollary}
\begin{proof}
We have shown that for any $\lambda < 1$, the second moment is bounded for a conditioned version of the \iid Gaussian prior (conditioning on $\|x\| \le 1 + \delta$); see Corollary~\ref{cor:gauss-prior}. This conditioned Gaussian prior can be coupled to the spherical prior as required by Proposition~\ref{prop:compare}, due to Gaussian spherical symmetry. The result follows from Proposition~\ref{prop:compare}.
\end{proof}
\noindent A more direct proof for the spherical prior is possible using known properties of the confluent hypergeometric function; see Appendix~\ref{app:confluent}.

Another corollary is that \emph{any} prior $\cX$ (with $\|x\| \to 1$ in probability) and for any $\lambda < 1$, contiguity holds on the level of spectra; this implies that no test depending only on the eigenvalues can succeed below the $\lambda = 1$ threshold, even though other tests can in some cases (e.g.\ the sparse Rademacher prior of Section~\ref{sec:sparse-rad}).
\begin{corollary}\label{cor:contig-eigs}
Let $\cX$ be any spike prior (with $\|x\| \to 1$ in probability). Let $Q_n$ be the joint distribution of eigenvalues of $\GWig_n(\lambda,\cX)$ and let $P_n$ be the joint distribution of eigenvalues of $\GWig_n(0)$. If $\lambda < 1$ then $Q_n \contig P_n$.
\end{corollary}
\begin{proof}
Due to Gaussian spherical symmetry, the distribution of eigenvalues of the spiked matrix depends only on the norm of the spike and not its direction; thus without loss of generality, $\cX$ is a mixture of spherical priors, over a norm distribution converging in probability to 1. The result now follows from Proposition~\ref{prop:compare} and Corollary~\ref{cor:sphere-prior}.
\end{proof}

\subsection{The conditioning method}
\label{sec:cond-method}

In this section, we give an alternative to the subgaussian method that can give tighter results in some cases. Here we give an overview, with the full details deferred to Appendix~\ref{app:cond-method}. Throughout this section we require the prior to be $\cX = \IID(\pi/\sqrt{n})$ where $\pi$ has finite support.

The main idea is that the second moment takes a particular form involving a multinomial random variable; it turns out that this exact form has been studied by \citet{bmnn} in the context of contiguity in the stochastic block model. Following their work, we apply the conditional second moment method (Lemma~\ref{lem:cond}), conditioning on a high-probability `good' event where the empirical distribution of $x$ is close to $\pi/\sqrt{n}$. Proposition~5 in \citet{bmnn} provides an exact condition (involving an optimization problem over matrices) for boundedness of the conditional second moment. This method improves upon the subgaussian method in some cases (see e.g.\ Section~\ref{sec:sparse-rad}).

Let $\Delta_{s^2}(\pi)$ denote the set of nonnegative vectors $\alpha \in \mathbb{R}^{s^2}$ with row- and column-sums prescribed by $\pi$, i.e.\ treating $\alpha$ as an $s \times s$ matrix, we have (for all $i$) that row $i$ and column $i$ of $\alpha$ each sum to $\pi_i$. Let $D(u,v)$ denote the KL divergence between two vectors: $D(u,v) = \sum_i u_i \log(u_i/v_i)$.

\begin{theorem}[conditioning method]
\label{thm:cond-method}
Let $\mathcal{X} = \mathrm{iid}(\pi)$ where $\pi$ has mean zero, unit variance, and finite support $\Sigma \subseteq \mathbb{R}$ with $|\Sigma| = s$. Let $Q_n = \GWig_n(\lambda,\cX)$ and $P_n = \GWig_n(0)$. Define the $s \times s$ matrix $\beta_{ab} = ab$ for $a,b \in \Sigma$. Identify $\pi$ with the vector of probabilities $\pi \in \RR^\Sigma$, and define $\bar\alpha = \pi \pi^\top$. Let
$$\overline{\lambda}_\cX = \left[\sup_{\alpha \in \Delta_{s^2}(\pi)} \frac{\langle \alpha,\beta \rangle^2}{2D(\alpha,\bar\alpha)}\right]^{-1/2}.$$
If $\lambda < \overline{\lambda}_\cX$ then $Q_n \contig P_n$.
\end{theorem}

\noindent In Appendix~\ref{app:cond-method}, we give the full proof and also compute that the limit value of the conditional second moment is $(1-\lambda^2)^{-1/2}$ (the same as in Theorem~\ref{thm:subg-iid}). We also explain the intuition behind the matrix optimization problem.

\subsection{Application: the sparse Rademacher prior}
\label{sec:sparse-rad}

Now consider the case where $\pi = \sqrt{1/\rho}\,\mathcal{R}(\rho)$ where $\mathcal{R}(\rho)$ is the sparse Rademacher distribution with sparsity $\rho \in (0,1]$: $\mathcal{R}(\rho)$ is 0 with probability $1-\rho$, and otherwise uniform on $\{\pm 1\}$. Here we give a summary of our results, with full details deferred to Appendix~\ref{app:sparse-rad}.

We know from Corollary~\ref{cor:pmone-wigner} that when $\rho = 1$, detection is impossible below the spectral threshold. However, for sufficiently small $\rho$ (roughly 0.054), an exhaustive search procedure is known to perform detection for some range of $\lambda$ values below the spectral threshold \citep{bmvvx}. Towards a matching lower bound, we would like to find $\rho^*$ as small as possible such that PCA is optimal for all $\rho \ge \rho^*$.

Using the subgaussian method (Theorem~\ref{thm:subg-iid}) it follows that PCA is optimal for all $\rho \ge 1/3$. The conditioning method (Theorem~\ref{thm:cond-method}) improves this constant substantially, to roughly $0.184$. Using a more sophisticated method that conditions on an event depending jointly on the signal and noise, \citet{noise-cond} improve the constant further, to roughly $0.138$. Similar (but quantitatively weaker) results have been obtained by \citet{bmvvx}.

Based on heuristics from statistical physics, \citet{LKZ-amp} predicted that the exact $\rho$ value at which PCA becomes sub-optimal is given by the replica-symmetric (RS) formula, which yields $\rho_\mathrm{RS} \approx 0.09$. It was later proven rigorously that $\rho_\mathrm{RS}$ is the exact threshold for nontrivial \emph{recovery} below $\lambda = 1$, and that if $\rho < \rho^*$ then detection below $\lambda=1$ is possible (by thresholding the free energy) \citep{mi,mi-proof,lelarge-limits-lowrank}. It remains open to show that detection is impossible below $\lambda = 1$ for all $\rho \ge \rho_\mathrm{RS}$. \citet{LKZ-amp} also conjecture a computational gap: when $\lambda < 1$, no polynomial-time algorithm can perform detection or recovery (regardless of $\rho$).

% NON-GAUSSIAN WIGNER
\section{Non-Gaussian Wigner models}
\label{sec:nong-wig}

\subsection{Main results}\label{sec:mainngw} We first define the spiked non-Gaussian Wigner model.

\begin{definition}\label{def:spiked-nongaussian-wigner}
In the {\em general spiked Wigner model} $\Wig(\lambda,\cP,\cP_d,\cX)$, one observes a matrix
$$ Y = \lambda x x^\top + \frac{1}{\sqrt{n}}W, $$
with the spike $x$ drawn from a spike prior $\cX$, and the entries of noise matrix $W$ drawn independently up to symmetry, with the off-diagonal entries drawn from a distribution $\cP$ and the diagonal entries drawn from a second distribution $\cP_d$. For the sake of normalization, we assume that $\cP$ has mean $0$ and variance $1$.
\end{definition}

\noindent Recall that the prior $\cX$ is required to obey the normalization $\|x\| \to 1$ in probability (see Definition~\ref{def:prior}).

The spectral behavior of this model is well understood\footnote{Many of the results cited here assume $\|x\| = 1$ and show almost-sure convergence of various quantities. Since we assume only $\|x\| \to 1$ in probability, the same convergence is true only in probability (which is enough for our purposes).} (see e.g.\ \citet{fp,cdf,wig-spk,nong-eigv1}). In fact it exhibits \emph{universality} (see e.g.\ \citet{TV-univ}): regardless of the choice of the noise distributions $\cP,\cP_d$ (with sufficiently many finite moments), many properties of the spectrum behave the same as if $\cP$ were a standard Gaussian distribution. In particular, for $\lambda \le 1$, the spectrum bulk has a semicircular distribution and the maximum eigenvalue converges almost surely to $2$. For $\lambda > 1$, an isolated eigenvalue emerges from the bulk with value converging to $\lambda + 1/\lambda$, and (under suitable assumptions) the top eigenvector has squared correlation $1 - 1/\lambda^2$ with the truth. 

In stark contrast we will show that from a statistical standpoint, universality breaks down entirely: the detection problem becomes easier when the noise is non-Gaussian. Let $\cX$ be a spike prior, and suppose that through the second moment method, we can establish contiguity between the \emph{Gaussian} spiked and unspiked models whenever $\lambda$ lies below some critical value
$$ \lambda^*_\cX \defeq \sup \left\{ \lambda \;\Big|\; \EE_{x,x' \sim \cX} \exp\left( \frac{n \lambda^2}{2} \langle x,x' \rangle^2\right) \text{ is bounded as $n \to \infty$} \right\}.$$
The detection threshold for the non-Gaussian Wigner model depends on $\lambda_\cX^*$ as well as a parameter $F_\cP$ (defined below) that depends on the noise distribution $\cP$.

\begin{theorem-nolabel}[informal; see Theorems~\ref{thm:nongauss-lower} and~\ref{thm:nong-upper}]
Under suitable conditions (see Assumptions~\ref{as:nong-lower} and~\ref{as:nong-upper}), the spiked model is contiguous to the unspiked model for all $\lambda < \lambda^*_\cX/\sqrt{F_\cP}$; but when $\lambda > 1/\sqrt{F_\cP}$, there exists an entrywise transformation $f$ such that the spiked and unspiked models can be consistently distinguished via the top eigenvalue of $f(\sqrt{n} Y)$.
\end{theorem-nolabel}
\noindent Recall that if we take the spike prior to be e.g.\ spherical or Rademacher, we have $\lambda^*_\cX = 1$, implying that our upper and lower bounds match, and thus our pre-transformed PCA procedure achieves the optimal threshold for \emph{any} noise distribution (subject to regularity assumptions). For reasons discussed later (see Appendix~\ref{app:nong-discrete}), we require $\cP$ to be a continuous distribution with a density function $p(w)$. The parameter $F_\cP$, which quantifies its difficulty, is the Fisher information of $\cP$ under translation:
$$ F_\cP = \Ex_{w \sim \cP}\left[\left(\frac{p'(w)}{p(w)}\right)^2\right] = \int_{-\infty}^\infty\frac{p'(w)^2}{p(w)} \,\dee w. $$
Gaussian noise enjoys an extremal value of this Fisher information, qualifying it as the unique hardest noise distribution (among a large class):

\begin{proposition}[\citet{pitman1979} p.\ 37]
\label{prop:F}
Let $\cP$ be a real distribution with a $C^1$, non-vanishing density function $p(w)$. Suppose $\mathrm{Var}[\cP] = 1$. Then $F_\cP \ge 1$, with equality if and only if $\cP$ is a standard Gaussian.
\end{proposition}
\noindent This is effectively a form of the Cram\'er--Rao inequality, and can be exploited for a proof of the central limit theorem \citep{brown1982proof,barron1986entropy}.

Our upper bound proceeds by a pre-transformed PCA procedure. Define $f(w) = -p'(w)/p(w)$, where $p$ is the probability density function of the noise $\cP$. Given the observed matrix $Y$, we apply $f$ entrywise to $\sqrt{n} Y$, and examine the largest eigenvalue. This entrywise transformation approximately yields another spiked Wigner model, but with improved signal-to-noise ratio. One can derive the transformation $-p'(w)/p(w)$ by using calculus of variations to optimize the signal-to-noise ratio of this new spiked Wigner model. This phenomenon is illustrated in Figures~\ref{fig:spectrum} and~\ref{fig:transform}:

\begin{figure}[!ht]
    \centering
    \begin{minipage}{0.48\textwidth}
    \hspace*{-0.15in}\includegraphics[width=\linewidth]{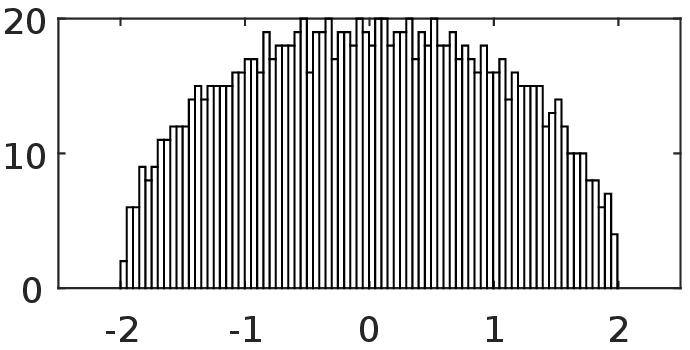}\\
    \hspace*{-0.15in}\includegraphics[width=\linewidth]{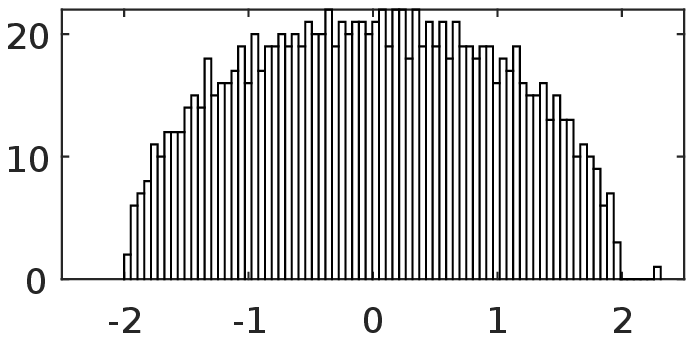}
    \captionof{figure}{Spectrum of a spiked Wigner matrix ($\lambda = 0.9$, $n=1200$) with bimodal noise, before (above) and after (below) the entrywise transformation. An isolated eigenvalue is evident only in the latter.
}
    \label{fig:spectrum}
    \end{minipage}\hfill%
    \begin{minipage}{0.48\textwidth}
    \vspace{.21in}\hspace{-0.2in}
    \includegraphics[width=\linewidth]{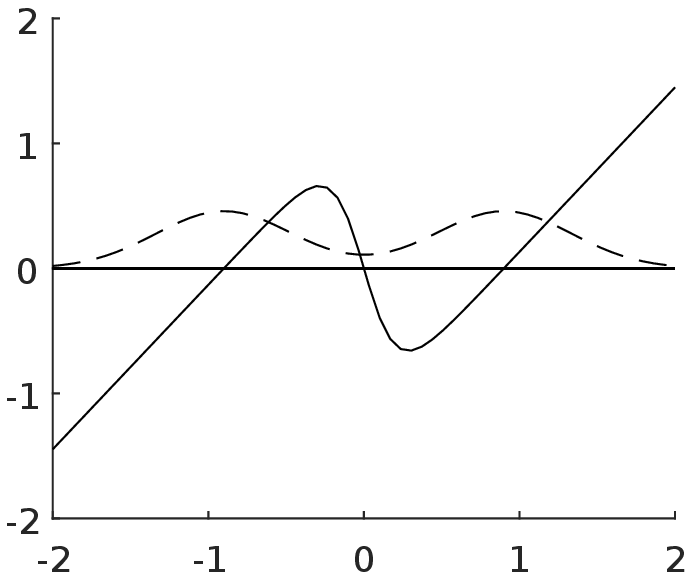}
    \captionof{figure}{The noise density $p$ (dashed) and entrywise transformation $-p'/p$ (solid). The bimodal noise is a convolution of Rademacher and Gaussian random variables.}
    \label{fig:transform}
    \end{minipage}%
\end{figure}

% intuition: \pm 1 noise
To intuitively understand why non-Gaussian noise makes the detection problem easier, consider the extreme case where the noise distributions $\cP$, $\cP_d$ are uniform on $\{\pm 1\}$, with mean $0$ and variance $1$. Since the noise contribution $\frac{1}{\sqrt n} W$ is entrywise exactly $\pm 1/\sqrt n$, it is very easy to detect and identify the small signal perturbation $\lambda xx^\top$, which is entrywise $O(1/n)$. If there is no spike, all the entries will be $\pm 1/\sqrt{n}$ (exactly). If there is a spike, each entry will be $\pm 1/\sqrt{n}$ plus a much smaller offset. One can therefore subtract off the noise and recover the signal exactly. In fact, if we let the noise be a smoothed version of $\{\pm 1\}$ (so that the derivative $p'$ exists), the entrywise transformation $-p'(w)/p(w)$ is precisely implementing this noise-subtraction procedure. This justifies the restriction to continuous noise distributions because any distribution with a point mass admits a similar trivial recovery procedure and we will not have contiguity for \emph{any} $\lambda > 0$; see Appendix~\ref{app:nong-discrete} for details.

% channel universality
The above results on non-Gaussian noise parallel a \emph{channel universality} phenomenon for mutual information, due to \citet{mi} (shown for finitely-supported \iid priors). The pre-transformed PCA procedure we use for our upper bound was previously suggested by \citet{LKZ-amp} based on linearizing an approximate message passing algorithm, but to our knowledge, no rigorous results have been previously established about its performance in general. Other entrywise pre-transformations have been shown to improve spectral approaches to various structured PCA problems \citep{DM-sparse-pca,kv}.

\subsection{Lower bound}
\label{sec:nong-symm}

In this section, we state our main statistical lower bound that establishes contiguity in the non-Gaussian Wigner setting. Given a noise distribution $\cP$, define the \emph{translation function}
$$ \tau(a,b) = \log \Ex_{\cP}\left[ \dd[T_a \cP]{\cP} \dd[T_b \cP]{\cP} \right] = \log \Ex_{z \sim \cP}\left[ \frac{p(z-a)}{p(z)} \frac{p(z-b)}{p(z)} \right], $$
where $T_a \cP$ denotes the translation of distribution $\cP$ by $a$. For instance, the translation function of standard Gaussian noise is $\tau(a,b) = ab$.

\begin{assumption}\label{as:nong-lower}
\begin{enumerate}[label=(\roman*)]
\item The prior $\cX$ satisfies (as usual) $\|x\| \to 1$ in probability, and furthermore $\cX$ is $(\sigma^2/n)$-subgaussian for some constant $\sigma^2$ (see Definition~\ref{def:subg}).
\item The prior $\cX$ satisfies high-probability norm bounds: for $q = 2,4,6,8$, there exists a constant $\alpha_q$ for which, with probability $1-o(1)$ over $x \sim \cX$, we have $\|x\|_q \leq \alpha_q n^{\frac1q - \frac12}$.
\item We assume the distributions $\cP,\cP_d$ have non-vanishing density functions $p(w),p_d(w)$, and translation functions $\tau,\tau_d$ that are $C^4$ in a neighborhood of $(0,0)$.
\end{enumerate}
\end{assumption}

\noindent Our main lower bound result is the following.

\begin{theorem}\label{thm:nongauss-lower}
Under Assumption~\ref{as:nong-lower}, $\Wig(\lambda,\cP,\cP_d,\cX)$ is contiguous to $\Wig(0,\cP,\cP_d)$ for all $\lambda < \lambda^*_\cX/\sqrt{F_\cP}$.
\end{theorem}
\noindent We defer the proof to Appendix~\ref{app:nong-lower}. In Appendix~\ref{app:nong-lower} we also show that the assumptions on $\cX$ are satisfied for the spherical prior and for reasonable \iid priors; see Propositions~\ref{prop:nong-lower-spherical} and \ref{prop:nong-lower-iid} below. The assumptions on $\cP,\cP_d$ are satisfied by any mixture of Gaussians of positive variance, for example.

\begin{proposition}\label{prop:nong-lower-spherical}
Conditions (i) and (ii) in Assumption~\ref{as:nong-lower} are satisfied for the spherical prior $\cXs$.
\end{proposition}

\begin{proposition}\label{prop:nong-lower-iid}
Consider an \iid prior $\cX = \IID(\pi/\sqrt{n})$ where $\pi$ is zero-mean, unit-variance,
%, has $\EE[\pi^{16}] < \infty$,
and subgaussian with some constant $\sigma^2$. Then conditions (i) and (ii) in Assumption~\ref{as:nong-lower} are satisfied.
\end{proposition}

% NON-GAUSSIAN WIGNER (UPPER)
\subsection{Pre-transformed PCA}
\label{sec:nong-wig-upper}

In this section we analyze a pre-transformed PCA procedure for the non-Gaussian spiked Wigner model. We need the following regularity assumptions.

% assumptions
\begin{assumption}
\label{as:nong-upper}
Of the prior $\cX$ we require (as usual) $\|x\| \to 1$ in probability, and we also assume that with probability $1-o(1)$, all entries of $x$ are small: $|x_i| \le n^{-1/2 + \alpha}$ for some fixed $\alpha < 1/8$. Of the noise $\cP$, we assume the following:
\begin{enumerate}[label=(\roman*),resume]
\item $\cP$ has a non-vanishing $C^3$ density function $p(w) > 0$,
\item Letting $f(w) = -p'(w)/p(w)$, we have that $f$ and its first two derivatives are polynomially-bounded: there exists $C > 0$ and an even integer $m \ge 2$ such that $|f^{(\ell)}(w)| \le C + w^m$ for all $0 \le \ell \le 2$.
\item With $m$ as in (ii), $\cP$ has finite moments up to $5m$:
$\EE|\cP|^k < \infty$ for all $1 \le k \le 5m$. 
\end{enumerate}
\end{assumption}

% main theorem
The main theorem of this section is the following.
\begin{theorem}\label{thm:nong-upper}
Let $\lambda \ge 0$ and let $\cX,\cP$ satisfy Assumption~\ref{as:nong-upper}. Let $\hat Y = \sqrt{n}\, Y$ where $Y$ is drawn from $\Wig(\lambda,\cP,\cP_d,\cX)$. Let $f(\hat Y)$ denote entrywise application of the function $f(w) = -p'(w)/p(w)$ to $\hat Y$, except we define the diagonal entries of $f(\hat Y)$ to be zero.
\begin{itemize}[leftmargin=*]
\item If $\lambda \le 1/\sqrt{F_\cP}$ then $\frac{1}{\sqrt n}\lambda_{\max}(f(\hat Y)) \to 2\sqrt{F_\cP}$ as $n \to \infty$.
\item If $\lambda > 1/\sqrt{F_\cP}$ then $\frac{1}{\sqrt n}\lambda_{\max}(f(\hat Y)) \to \lambda F_\cP + \frac{1}{\lambda} > 2 \sqrt{F_\cP}$ as $n \to \infty$ and furthermore the top (unit-norm) eigenvector $v$ of $f(\hat Y)$ correlates with the spike:
$\langle v,x \rangle^2 \ge (\lambda - 1/\sqrt{F_\cP})^2/\lambda^2 - o(1)$ with probability $1-o(1)$.
\end{itemize}
Convergence is in probability. Here $\lambda_{\max}(\cdot)$ denotes the maximum eigenvalue.
\end{theorem}

\noindent The proof is deferred to Appendix~\ref{app:nong-upper}, but the main idea is that the entrywise transformation $f$ approximately produces another spiked (non-Gaussian) Wigner matrix with a different signal-to-noise ratio $\lambda$, and we can choose $f$ to optimize this.

% diagonals are not important
We have set the diagonal entries to zero for convenience, but this is not essential: so long as we define the diagonals of $f(\hat Y)$ so that the largest (in absolute value) diagonal entry is $o(\sqrt{n})$, the diagonal entries can only change the spectral norm of $f(\hat Y)$ by $o(\sqrt{n})$ and so the result still holds.

% WISHART

\section{Spiked Wishart models}
\label{sec:wish}

\subsection{Main results}

% define Wishart model
We first formally define the spiked Wishart model:
\begin{definition}\label{def:spiked-wishart}
Let $\gamma > 0$ and $\beta \in [-1,\infty)$. Let $\cX = \{\cX_n\}$ be a spike prior. The spiked (Gaussian) Wishart model $\Wish(\gamma,\beta,\cX)$ on $n \times n$ matrices is defined as follows: we first draw a hidden spike $x \sim \cX_n$, and then reveal $Y = \frac{1}{N} X X^\top$, where $X$ is an $n \times N$ matrix whose columns are sampled independently from $\cN(0,I+\beta x x^\top)$; the parameters $N$ and $n$ scale proportionally with $n/N \to \gamma$ as $n \to \infty$. If $\beta < 0$ and $|\beta| \cdot \|x\|^2 > 1$ (so that the covariance matrix is not positive semidefinite), output a failure event $\perp$.
\end{definition}

\noindent Recall that spike priors are required to satisfy $\|x\| \to 1$ in probability (Definition~\ref{def:prior}). Our contiguity results will apply even to the case when the sample matrix $X$ is revealed.

% spectral behavior
The spiked Wishart model admits the following spectral behavior. In this high-dimensional setting, the spectrum bulk of $Y$ converges
%in the unspiked case % also true for spiked case
to the Marchenko--Pastur distribution with shape parameter $\gamma$. By results of \citet{bbp} and \citet{baik-silverstein}, it is known that the top eigenvalue consistently distinguishes the spiked and unspiked models when $\beta > \sqrt{\gamma}$. In fact, matching lower bounds are known in the absence of a prior (equivalently, for the spherical prior) due to \citet{sphericity}: for $0 \leq \beta < \sqrt{\gamma}$, no hypothesis test distinguishes this spiked model from the unspiked model with $o(1)$ error. In the case of $-1 \leq \beta < 0$, a corresponding PCA threshold exists: the minimum eigenvalue exits the bulk when $\beta < -\sqrt{\gamma}$ \citep{baik-silverstein}, but we are not aware of lower bounds in the literature. The case of $\beta < -1$ is of course invalid, as the covariance matrix must be positive semidefinite. As in the Wigner model, consistent detection is possible in the critical case $|\beta| = \sqrt{\gamma}$, at least when $\beta > 0$; see Appendix~\ref{app:critical}.

% crude lower bound
Our goal in this section will be to give lower and upper bounds on the statistical threshold for $\gamma$ (as a function of $\beta$) for various priors on the spike. We begin with a crude lower bound that allows us to transfer any lower bound for the Gaussian Wigner model into a lower bound for the Wishart model. Recall that $\lambda^*_\cX$ denotes the threshold for boundedness of the Gaussian \emph{Wigner} second moment:
\begin{equation}
\label{eq:lambdastar}
\lambda^*_{\cX} \defeq \sup \left\{ \lambda \;\Big|\; \Ex_{x,x' \sim \cX} \exp\left( \frac{n \lambda^2}{2} \langle x,x' \rangle^2\right) \text{ is bounded as $n \to \infty$} \right\}.
\end{equation}
\begin{proposition}
\label{prop:wigner-wishart-bound}
Let $\cX$ be a spike prior. If $\beta^2 < 1-e^{-\gamma (\lambda^*_\cX)^2}$ then $\Wish(\gamma,\beta,\cX)$ is contiguous to $\Wish(\gamma)$.
\end{proposition}
\noindent The proof can be found in Section~\ref{sec:wish-small-dev}. A consequence of the above is that if $\lambda^*_\cX = 1$, so that the spectral method is optimal in the Wigner setting, it follows that the ratio between the above Wishart lower bound ($1-e^{-\gamma (\lambda^*_\cX)^2}$) and the spectral upper bound ($\gamma$) tends to $1$ as $\gamma \to 0$. This reflects the fact that the Wigner model is a particular $\gamma \to 0$ limit of the Wishart model \citep{jo-testing-spiked}. For $\beta > 0$, we will later give an even stronger implication from Wigner to Wishart lower bounds (Corollary~\ref{cor:wig-pos-wish}).

% noise conditioning
Although Proposition~\ref{prop:wigner-wishart-bound} is a strong bound for small $\gamma$, it is rather weak for large $\gamma$ (and in particular does not cover the case $\beta \ge 1$). In Section~\ref{sec:wish-main-lower} we will remedy this by giving a much tighter lower bound (Theorem~\ref{thm:wish-nc}) which depends on the rate function of the large deviations of the prior. The proof involves an application of the conditional second moment method whereby we condition away from certain `bad' events depending on interactions between the signal and noise (similarly to \citet{noise-cond}). One consequence (Corollary~\ref{cor:wig-pos-wish}) of our lower bound roughly states that if detection is impossible below the spectral threshold ($\lambda = 1$) in the \emph{Wigner} model, then it is also impossible below the spectral threshold ($|\beta| = \sqrt{\gamma}$) in the Wishart model for all positive $\beta$. (This is not true for negative $\beta$.)

% upper bound
We complement our lower bounds with the following upper bound.
\begin{theorem}\label{thm:wishart-mle}
Let $\beta \in (-1,\infty)$. Let $\cX_n$ be a spike prior supported on at most $c^n$ points, for some fixed $c > 0$. If
$$ 2\gamma \log c < \beta - \log(1+\beta)$$
then there is a (inefficient) test that consistently distinguishes between the spiked Wishart model $\Wish(\gamma,\beta,\cX)$ and the unspiked model $\Wish(\gamma)$.
\end{theorem}
\noindent The test that gives this upper bound is based on the maximum likelihood estimator (MLE), computed by exhaustive search over all possible spikes. The proof, which can be found in Appendix~\ref{app:wishart-mle}, is a simple application of the Chernoff bound and the union bound. For some priors (such as \iid sparse Rademacher) we can get the most mileage out of this theorem by first conditioning on a $(1-o(1))$-probability event (e.g.\ $x$ has a typical number of nonzeros) in order to decrease the value of $c$.

% beta = -1
We will typically not consider the boundary case $\beta = -1$. Note, however, that if $\beta = -1$ and the prior is finitely-supported (for each $n$), with $\|x\| = 1$ almost surely, then detection is possible for any $\gamma$: in the spiked model, the spike is orthogonal to all of the samples; but in the unspiked model, with probability 1 there will not exist a vector in the support of the prior that is orthogonal to all of the samples.

% specific priors
We now summarize the implications of our lower and upper bounds for some specific priors.
\begin{itemize}
\item {\bf Spherical}: For the spherical prior ($x$ is drawn uniformly from the unit sphere), it was known previously that the PCA threshold $|\beta| = \sqrt{\gamma}$ is optimal for all positive $\beta$ \citep{sphericity}. We show that the PCA threshold is also optimal for all $-1 < \beta < 0$.
\item {\bf Rademacher}: For the Rademacher prior $\IID(\pm 1/\sqrt{n})$, we show that the PCA threshold is optimal for all $\beta > 0$. However, when $\beta$ is negative and sufficiently close to $-1$, the MLE of Theorem~\ref{thm:wishart-mle} succeeds below the PCA threshold.
\item {\bf Sparse Rademacher} (defined in Section~\ref{sec:sparse-rad}): If the sparsity $\rho$ is sufficiently small, the MLE beats PCA in both the positive- and negative-$\beta$ regimes. However, for any fixed $\rho$, if $\beta$ is sufficiently large (and positive) then the PCA threshold is optimal.
\end{itemize}

\noindent See Appendix~\ref{app:wish-priors} for details on the above results, including how they follow from our general upper and lower bounds (Theorems~\ref{thm:wishart-mle} and \ref{thm:wish-nc}). Figure~\ref{fig:wishart} depicts our upper and lower bounds for the Rademacher and sparse Rademacher priors.

% hypothesis testing
As in the Wigner model, our methods often yield the limit value of the (conditional) second moment and thus imply asymptotic bounds on hypothesis testing power via Proposition~\ref{prop:hyptest}; see Appendix~\ref{app:hyptest} for details.

\begin{figure}[!ht]
    \centering
    \begin{minipage}{0.48\textwidth}
    \includegraphics[width=\linewidth]{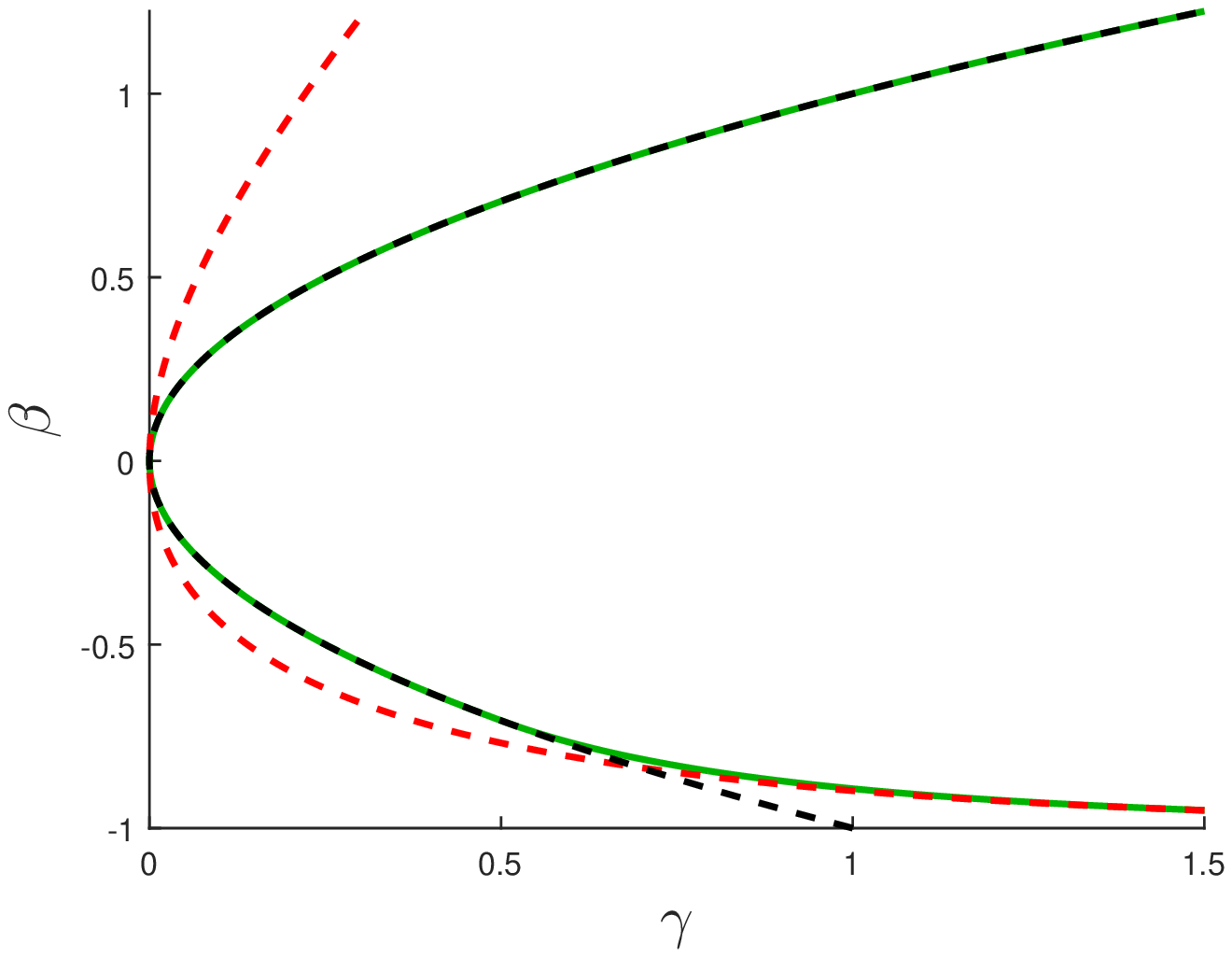}
    \end{minipage}\hfill%
    \begin{minipage}{0.48\textwidth}
    \includegraphics[width=\linewidth]{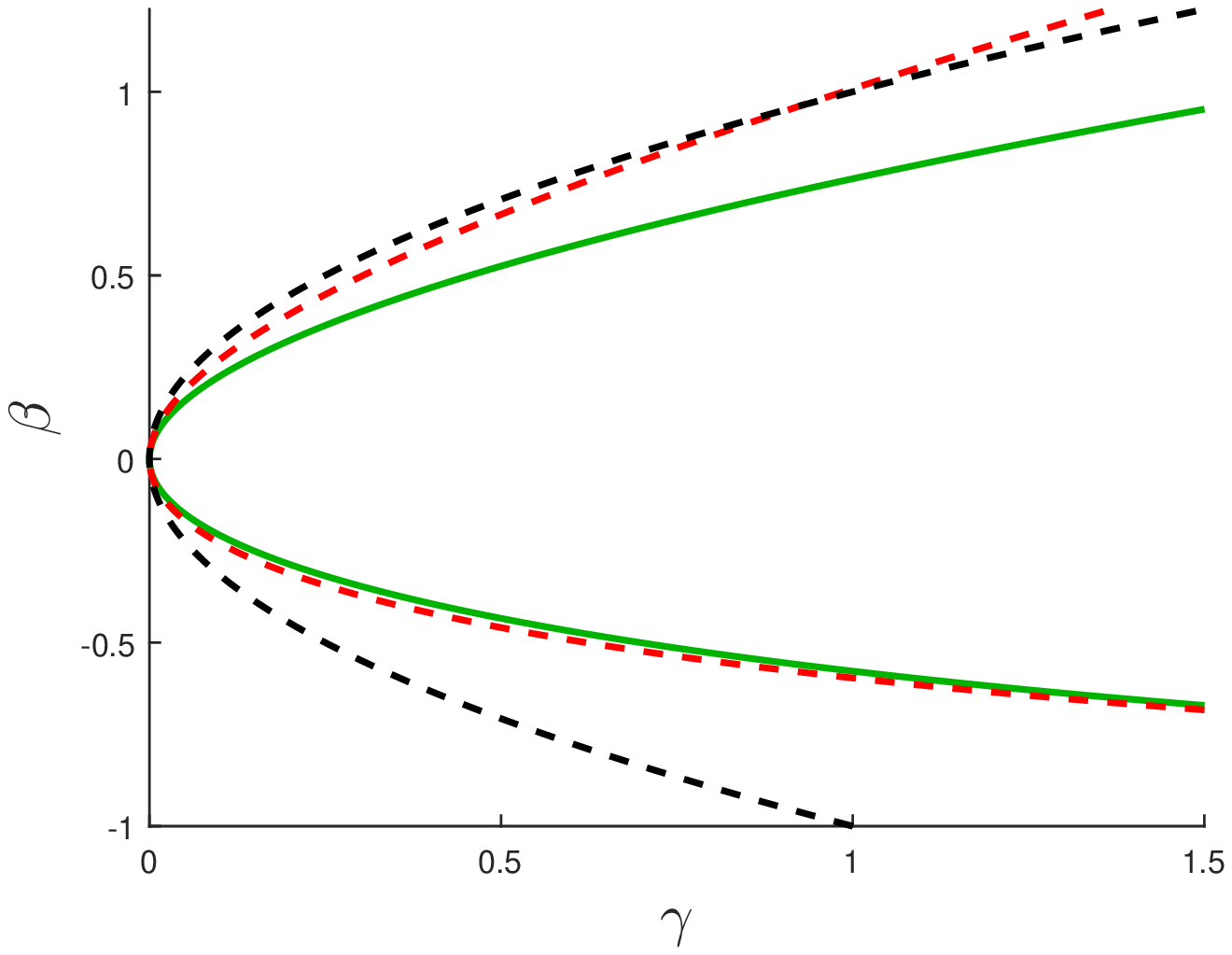}
    \end{minipage}
    \caption{Upper and lower bounds for the spiked Wishart model with Rademacher prior (left panel) and sparse Rademacher prior with $\rho = 0.03$ (right panel). PCA succeeds to the left of the dashed black curve $\beta^2 = \gamma$. To the right of the solid green curve, detection is impossible (by Theorem~\ref{thm:wish-nc}; see Appendix~\ref{app:wish-priors} for details). To the left of the dashed red curve, detection is possible via the inefficient MLE algorithm of Theorem~\ref{thm:wishart-mle}. (The red curve is not a tight analysis of the MLE and is sometimes weaker than the PCA bound.) For the Rademacher prior, the lower bound matches PCA for all $\beta > -0.7$, but the MLE succeeds below the PCA threshold for all $\beta < -0.84$. For the sparse Rademacher prior with any $\rho$, the lower bound matches PCA for sufficiently large positive $\beta$ (not shown); see Proposition~\ref{prop:wish-large-beta}.}
    \label{fig:wishart}
\end{figure}

\subsection{Rate functions}

Our main lower bound will depend on the prior through tail probabilities of the correlation $\langle x,x' \rangle$ of two spikes $x,x'$ drawn independently from the prior $\cX$. These tail probabilities are encapsulated by the rate function $f_\cX:[0,1) \to [0,\infty)$ of the large deviations of $\cX$, which is intuitively defined by $\prob{|\langle x,x' \rangle| \geq t} \approx \exp(-n f_\cX(t))$. Formally we define $f_\cX$ as follows.

\begin{definition}\label{def:rate-function}
Let $\cX = \{\cX_n\}$ be a spike prior. For $x,x'$ drawn independently from $\cX_n$ and $t \in [0,1)$, let
$$ f_{n,\cX}(t) = -\frac1n \log \prob{|\langle x,x' \rangle| \geq t}. $$
Suppose we have $f_{n,\cX}(t) \geq b_{n,\cX}(t)$ for some sequence of functions $b_{n,\cX}$ that converges uniformly on $[0,1)$ to $f_\cX$ as $n \to \infty$. Then we call such $f_{\cX}$ the \emph{rate function} of the prior $\cX$.
\end{definition}

\noindent Without loss of generality, $f_\cX(0) = 0$ and $f_\cX(t)$ is non-decreasing. Note that a tail bound of the form $\prob{|\langle x,x' \rangle| \geq t} \le \mathrm{poly}(n) \exp(-n f_\cX(t))$ is sufficient to establish that $f_\cX$ is a rate function.

We now state the rate functions for some priors of interest. It is proven by \citet{noise-cond} that these indeed satisfy the definition of rate function.

\begin{proposition}[\citet{noise-cond}]
\label{prop:rate-functions}
We have the following rate functions for the spherical, Rademacher, and sparse Rademacher priors.
\begin{itemize}
\item {\bf Spherical}: $f_\mathrm{sph}(t) = -\frac{1}{2} \log(1-t^2)$.
\item {\bf Rademacher}: $f_\mathrm{Rad}(t) = \log 2 - H\left(\frac{1+t}{2}\right)$.
\item {\bf Sparse Rademacher}\footnote{This is for a variant of the sparse Rademacher prior where the sparsity is exactly $\rho n$. See Appendix~\ref{app:wish-priors} for details on how this extends to our variant.} with sparsity $\rho$: $$f_\rho(t) = \min_{\zeta \in [\max(\rho t,1-2\rho),\rho]} G_\rho(\zeta) + \zeta f_\mathrm{Rad}\left(\frac{\rho t}{\zeta}\right)$$
where
$$G_\rho(\zeta) = -H(\{\zeta,\rho-\zeta,\rho-\zeta,1-2\rho+\zeta\}) + 2H(\rho).$$
Here $H(p) = -p \log p - (1-p)\log(1-p)$ is the binary entropy, and $H(\{p_i\}) = -\sum_i p_i \log p_i$.
\end{itemize}
\end{proposition}

\noindent Note that rate functions for general \iid priors can be easily derived from large deviations theory (Cram\'er's theorem) since $\langle x,x' \rangle$ is the sum of $n$ \iid random variables; this is how the Rademacher rate function is derived. However, to obtain stronger results in some cases, one may use a variant of the prior that conditions on typical outcomes (similarly to our conditioning method for the Wigner model (Section~\ref{sec:cond-method}) or Appendix~A of \citet{bmnn}); this is how the sparse Rademacher rate function is derived.

We will need the following strengthening of the notion of rate function.
\begin{definition}
\label{def:local-chernoff}
We say that a rate function $f_\cX$ for a prior $\cX$ admits a \emph{local Chernoff bound} if there exists $T > 0$ and $C > 0$ such that for any $n$,
$$\prob{|\langle x,x' \rangle| \ge t} \le C \exp(-n f_\cX(t)) \quad \forall t \in [0,T]$$
where $x$ and $x'$ are drawn independently from $\cX_n$.
\end{definition}
The Rademacher and sparse Rademacher rate functions in Proposition~\ref{prop:rate-functions} each admit a local Chernoff bound; see \citet{noise-cond}.

\subsection{Main lower bound result}
\label{sec:wish-main-lower}

We are now ready to state our main lower bound result. Recall that $\lambda^*_\cX$ denotes the Wigner threshold (\ref{eq:lambdastar}).

\begin{theorem}
\label{thm:wish-nc}
Let $\cX$ be a spike prior with rate function $f_\cX$. Let $\beta > -1$ and $\gamma^* > 0$. Suppose that either
\begin{enumerate}[label=(\roman*)]
\item $\beta^2/\gamma^* \le (\lambda^*_\cX)^2$, or
\item $f_\cX$ admits a local Chernoff bound (Definition~\ref{def:local-chernoff}).
\end{enumerate}
If
\begin{equation}
\label{eq:nc-result}
\gamma^* f_\cX(t) \ge (1+\beta) \frac{t(w-t)}{1-t^2} + \frac{1}{2}\log\left(\frac{1-w^2}{1-t^2}\right) \quad \forall t \in (0,1)
\end{equation}
where
$$w = \sqrt{A^2+1}-A \;\text{ with }\; A = \frac{1-t^2}{2t(\beta+1)},$$
then $\Wish(\gamma,\beta,\cX)$ is contiguous to $\Wish(\gamma)$ for all $\gamma > \gamma^*$.
\end{theorem}

% monotonicity
\noindent We expect condition (ii) to hold for all reasonable priors; condition (i) yields a weaker result in some cases but is sometimes more convenient. Some basic properties of (\ref{eq:nc-result}) are discussed in Appendix~\ref{app:prop-F}. In Appendix~\ref{app:monotonicity} we establish the following monotonicity:% A notable property of (\ref{eq:nc-result}) is the following form of monotonicity; see Appendix~\ref{app:monotonicity}.
\begin{proposition}
\label{prop:eq-monotone}
Let $\cX$ be a spike prior. Fix $\lambda > 0$ and $\bar\beta \in (-1,\infty)\setminus\{0\}$. If (\ref{eq:nc-result}) holds for $\bar\beta$ and $\gamma^* = \bar\beta^2/\lambda^2$ then it also holds for any $\beta > \bar\beta$ and $\gamma^* = \beta^2/\lambda^2$.
\end{proposition}
\noindent In particular, if $\lambda = 1$ (so that $\gamma^* = \beta^2$, corresponding to the spectral threshold) we have that if Theorem~\ref{thm:wish-nc} shows that the PCA threshold is optimal for some $\bar\beta \in (-1,\infty)\setminus\{0\}$, then the PCA threshold is also optimal for all $\beta > \bar\beta$.

% connection to wigner
The following connection to the Wigner model is also proved in Appendix~\ref{app:monotonicity}, corresponding to the $\beta \to 0$ limit of the monotonicity property above:
\begin{corollary}
\label{cor:wig-pos-wish}
Suppose $\langle x,x' \rangle$ is $(\sigma^2/n)$-subgaussian (Definition~\ref{def:subg}), where $x$ and $x'$ are drawn independently from $\cX_n$. Then for any $\beta > 0$ and any $\gamma > \beta^2 \sigma^2$ we have $\Wish(\gamma,\beta,\cX) \contig \Wish(\gamma)$.
\end{corollary}
\noindent Recall that the subgaussian condition above implies a \emph{Wigner} lower bound for all $\lambda < 1/\sigma$ (Proposition~\ref{prop:subg}). This means whenever Proposition~\ref{prop:subg} implies that the PCA threshold is optimal for the Wigner model, we also have that the PCA threshold is optimal for the Wishart model for any positive $\beta$. Conversely, if Theorem~\ref{thm:wish-nc} shows that PCA is optimal for all $\beta > 0$ then it is also optimal for the Wigner model (see Proposition~\ref{prop:pos-wish-wig}). In light of the above monotonicity (Proposition~\ref{prop:eq-monotone}), these results makes sense because the Wigner model corresponds to the $\gamma \to 0$ limit of the Wishart model \citep{jo-testing-spiked}.

% PCA tight for large beta
We also show (in Appendix~\ref{app:monotonicity}) that for a wide range of priors, the PCA threshold becomes optimal for sufficiently large $\beta$:
\begin{proposition}\label{prop:wish-large-beta}
Suppose $\cX = \IID(\pi/\sqrt{n})$ where $\pi$ is a mean-zero unit-variance distribution for which $\pi \pi'$ (product of two independent copies of $\pi$) has a moment-generating function $M(\theta) \defeq \EE\exp(\theta \pi \pi')$ which is finite on an open interval containing zero. Then there exists $\bar \beta$ such that for any $\beta \ge \bar\beta$ and any $\gamma > \beta^2$ we have $\Wish(\gamma,\beta,\cX) \contig \Wish(\gamma)$.
\end{proposition}

% mention comparison
A final property of Theorem~\ref{thm:wish-nc} is that it gives similar thresholds for similar priors in the sense of Proposition~\ref{prop:compare} for the Wigner model; see Proposition~\ref{prop:wish-compare} for details.

\subsection{Lower bound proof summary}

% proof sketch
The full proof of Theorem~\ref{thm:wish-nc} will be completed in the next section, but we now describe the proof outline and give some preliminary results. We approach contiguity for the spiked Wishart model through the second moment method outlined in Section~\ref{sec:contig}. Note that detection can only become easier when given the original sample matrix $X$ (instead of $\frac{1}{N} XX^\top$), so we establish the stronger statement that the spiked distribution on $X$ is contiguous to the unspiked distribution. We first simplify the second moment in high generality.
\begin{proposition}
\label{prop:wishart-2mom}
For any $|\beta| < 1$ there exists $\delta > 0$ such that the following holds. Let $\cX$ be a spike prior supported on vectors $x$ with $1-\delta \le \|x\| \le 1+\delta$. In distribution $Q_n$, let a hidden spike $x$ be drawn from $\cX_n$, and let $N$ independent samples $y_i$, $1 \leq i \leq N$, be revealed from the normal distribution $\cN(0, I_{n \times n} + \beta x x^\top)$. In distribution $P_n$, let $N$ independent samples $y_i$, $1 \leq i \leq N$, be revealed from $\cN(0,I_{n \times n})$. Then we have
$$ \Ex_{P_n}\left[ \left(\dd[Q_n]{P_n} \right)^2 \right] = \Ex_{x,x' \sim \cX}\left[ \left( 1 - \beta^2 \langle x,x' \rangle^2 \right)^{-N/2} \right].$$
\end{proposition}
\noindent This result has appeared in higher generality \citep{CMW-cov}; for completeness we give the proof in Section~\ref{sec:wish-2mom-pf}. The condition $1-\delta \le \|x\| \le 1+\delta$ will not be an issue because we can always consider a modified prior that conditions on this $(1-o(1))$-probability event (see Lemma~\ref{lem:cond}). Note that the above second moment has the curious property of symmetry under replacing $\beta$ with $-\beta$. In contrast, the original Wishart model does not, since for instance $\beta > 1$ is allowed while $\beta < -1$ is not. As a result, the second moment method gives good results for negative $\beta$ but substantially sub-optimal results for positive $\beta$. To remedy this, we will apply the conditional second moment method (Lemma~\ref{lem:cond}), conditioning on an event that depends jointly on the signal and noise (we previously only conditioned on the signal).

% small and large deviations
The proof of Theorem~\ref{thm:wish-nc} has two parts. In Section~\ref{sec:wish-small-dev} we control the \emph{small deviations} of the second moment, i.e.\ the contribution from $\langle x,x' \rangle^2$ values at most some small $\eps > 0$. Here we use either the Wigner lower bound (i) or the local Chernoff bound (ii) (combined with (\ref{eq:nc-result})), whichever is provided. This step uses the basic second moment of Proposition~\ref{prop:wishart-2mom} without conditioning. In Section~\ref{sec:wish-nc} we complete the proof by controlling the remaining \emph{large deviations} of the conditional second moment. Here we use the condition (\ref{eq:nc-result}) on the rate function of the prior.

We remark that conditions (i) and (ii) in Theorem~\ref{thm:wish-nc} are related because using the subgaussian method of Section~\ref{sec:subg-method}, a Chernoff-type bound on $\langle x,x' \rangle$ implies a Wigner lower bound; note however that a local Chernoff bound only needs to hold near $t=0$.

\subsection{Proof of lower bound}
\label{sec:wish-lower-pf}

This section is devoted to proving Theorem~\ref{thm:wish-nc}. Along the way we will also prove Propositions~\ref{prop:wishart-2mom} and \ref{prop:wigner-wishart-bound}.

% basic second moment
\subsubsection{Second moment computation: proof of Proposition~\ref{prop:wishart-2mom}}
\label{sec:wish-2mom-pf}
We first compute:
\begin{align*}
\dd[Q_n]{P_n}&(y_1,\ldots,y_N) = \Ex_{x' \sim \cX}\left[ \prod_{i=1}^n \frac{\exp(-\frac12 y_i^\top (I + \beta x' x'^\top)^{-1} y_i)}{\sqrt{\det(I + \beta x' x'^\top)} \exp(-\frac12 y_i^\top y_i)} \right]\\
&= \Ex_{x'}\left[ \det(I + \beta x' x'^\top)^{-N/2} \prod_{i=1}^N \exp\left( -\frac12 y_i^\top ((I + \beta x' x'^\top)^{-1} - I) y_i \right) \right].
\intertext{Note that $(I + \beta x' x'^\top)^{-1}$ has eigenvalue $(1+\beta \|x'\|^2)^{-1}$ on $x'$ and eigenvalue $1$ on the orthogonal complement of $x'$. Thus $(I + \beta x' x'^\top)^{-1} - I = \frac{-\beta}{1+\beta \|x'\|^2} x' x'^\top$, and we have:}
&= \Ex_{x'} \left[ (1+\beta \|x'\|^2)^{-N/2} \prod_{i=1}^N \exp\left( \frac12 \,\frac{\beta}{1+\beta \|x'\|^2} \langle y_i, x' \rangle^2 \right) \right].
\end{align*}

\noindent Passing to the second moment, we compute:
\begin{align*}
\Ex_{P_n}&\left[ \left(\dd[Q_n]{P_n} \right)^2 \right] = \Ex_{Q_n}\left[ \dd[Q_n]{P_n} \right] \\
&= \Ex_{x,x'}\left[ (1+\beta \|x'\|^2)^{-N/2} \prod_{i=1}^N \,\Ex_{y_i \sim \N(0,I+\beta x x^\top)} \exp\left( \frac12 \,\frac{\beta}{1+\beta \|x'\|^2} \langle y_i, x' \rangle^2 \right) \right].
\intertext{Over the randomness of $y_i$, we have $\langle y_i, x' \rangle \sim \cN(0, \|x'\|^2 + \beta \langle x,x' \rangle^2)$, so that the inner expectation can be simplified using the moment-generating function (MGF) of the $\chi_1^2$ distribution:}
&= \Ex_{x,x'}\left[ (1+\beta \|x'\|^2)^{-N/2} \prod_{i=1}^N \left( 1 - \frac{\beta}{1+\beta \|x'\|^2} (\|x'\|^2 + \beta \langle x,x' \rangle^2) \right)^{-1/2} \right] \\
&= \Ex_{x,x'}\left[ \left( 1 - \beta^2 \langle x,x' \rangle^2 \right)^{-N/2} \right]
\end{align*}
as desired. Here the MGF step requires
\begin{equation}
\label{eq:eps-cond}
\frac{\beta}{1+\beta \|x'\|^2}(\|x'\|^2+\beta\langle x,x' \rangle^2) < 1.
\end{equation}
Provided that $\|x\|$ and $\|x'\|$ are sufficiently close to 1, this is true so long as either $|\beta| < 1$ (as assumed by Proposition~\ref{prop:wishart-2mom}) or $\langle x,x' \rangle^2$ is sufficiently small (as in the small deviations of the next section).

% small deviation proof
\subsubsection{Small deviations and proof of Proposition~\ref{prop:wigner-wishart-bound}}
\label{sec:wish-small-dev}

We now show how to bound the \emph{small deviations}
$$S(\eps) \defeq \Ex_{x,x' \sim \cX} (1-\beta^2\langle x,x' \rangle^2)^{-N/2} \;\one_{\langle x,x' \rangle^2 \leq \eps}$$
of the Wishart second moment in terms of the Wigner second moment. (Assume $\|x\|,\|x'\|$ are sufficiently close to 1 and $\eps > 0$ is a sufficiently small constant so that (\ref{eq:eps-cond}) holds). Letting $\hat\gamma = n/N$ so that $\hat\gamma \to \gamma$, we have
\begin{align*}
S(\eps) &= \Ex_{x,x' \sim \cX} \exp\left(\frac{-n}{2\hat\gamma} \log(1-\beta^2\langle x,x' \rangle^2) \right) \;\one_{\langle x,x' \rangle^2 \leq \eps} \\
&\leq \Ex_{x,x' \sim \cX} \exp\left( \frac{-n}{2\hat\gamma \eps^2} \log(1-\eps^2\beta^2) \langle x,x' \rangle^2 \right)
\end{align*}
using the convexity of $t \mapsto -\log(1-\beta^2 t)$. Note that this resembles the Wigner second moment and so (by definition of $\lambda_\cX^*$) it is bounded as $n \to \infty$ so long as
\begin{equation}
\label{eq:small-dev-cond}
\frac{-1}{\gamma \eps^2} \log(1-\eps^2\beta^2) < (\lambda^*_\cX)^2.
\end{equation}
Proposition~\ref{prop:wigner-wishart-bound} now follows by setting $\eps = 1 + \delta$ for small $\delta > 0$ and conditioning the prior on $\|x\|^2 \le 1+\delta$. (See Section~\ref{sec:compare} for similar arguments; note that the conditioning can only increase the Wigner second moment by a $1+o(1)$ factor.) Furthermore, using the bound $\log t \geq 1 - 1/t$ we have the following fact that will be used in the proof of Theorem~\ref{thm:wish-nc}.
\begin{lemma}
\label{lem:wish-small-dev-1}
If $\beta^2/\gamma < (\lambda^*_\cX)^2$ then there exists $\eps > 0$ such that $S(\eps)$ is bounded as $n \to \infty$.
\end{lemma}
\noindent Note that $\beta^2/\gamma < (\lambda^*_\cX)^2$ is precisely condition (i) in the statement of Theorem~\ref{thm:wish-nc}. If instead condition (ii) holds, we can control the small deviations using the following lemma, deferred to Appendix~\ref{app:wish-small-dev-chernoff}:
\begin{lemma}
\label{lem:wish-small-dev-2}
If (\ref{eq:nc-result}) holds and $f_\cX$ admits a local Chernoff bound, then there exists $\eps > 0$ such that $S(\eps)$ is bounded as $n \to \infty$.
\end{lemma}

% main lower bound
\subsection{Proof of Theorem~\ref{thm:wish-nc}}
\label{sec:wish-nc}

We now prove our main lower bound result using the conditional second moment method. Define $Q_n$ and $P_n$ as in Proposition~\ref{prop:wishart-2mom}. For a vector $x \in \RR^n$ and an $n \times n$ matrix $Y$, define the `good' event $\Omega(x,Y)$ by
$$x^\top Y x/\|x\|^2 \in [(1+\beta\|x\|^2)(1-\eta),(1+\beta\|x\|^2)(1+\eta)]$$
where $\eta = \frac{\log n}{\sqrt n}$. Note that under $Q_n$ (where $x$ is the spike and $Y$ is the Wishart matrix: $Y = \frac{1}{N} XX^\top$ where the columns of $X$ are the samples $y_i$), $x^\top Y x/\|x\|^2 \sim (1+\beta\|x\|^2)\chi_N^2/N$ and so $\Omega(x,Y)$ occurs with probability $1-o(1)$. Let $\tilde Q_n$ be the conditional distribution of $Q_n$ given $\Omega(x,Y)$.

For simplicity we now specialize to the case where $\cX$ is supported on unit vectors $\|x\| = 1$; see Appendix~\ref{app:unit-wlog} for the general case. Similarly to the proof of Proposition~\ref{prop:wishart-2mom}, we compute the conditional second moment as follows.
$$\dd[\tilde Q_n]{P_n}%(y_1,\ldots,y_N)
= (1+o(1)) \Ex_{x' \sim \cX} \left[ \one_{\Omega(x',Y)}\, (1+\beta)^{-N/2} \prod_{i=1}^N \exp\left( \frac12\, \frac{\beta}{1+\beta} \langle y_i, x' \rangle^2 \right) \right]$$
and so $\EE_{P_n} \left(\dd[\tilde Q_n]{P_n}\right)^2 = (1+o(1)) \,\EE_{x,x' \sim \cX}\, m(\langle x,x' \rangle)$ where
\begin{align*}
m(\langle x,x' \rangle) &= \!\Ex_{Y \sim P_n}\! (1+\beta)^{-N}\! \exp\!\left( \frac{N}{2}  \,\frac{\beta}{1+\beta} (x^\top Y x + x'^\top Y x') \right) \!\one_{\Omega(x,Y)} \one_{\Omega(x',Y)} \\
&= \Ex_{Y \sim P_n} (1+\beta)^{-N} \exp\left(N \beta \left(1+\frac{\Delta}{2} + \frac{\Delta'}{2}\right)\right) \one_{|\Delta| \le \eta} \one_{|\Delta'| \le \eta} \numberthis\label{eq:cond-2nd}
%&\defeq (1+o(1)) \Ex_{x,x' \sim \cX} m(\langle x,x' \rangle) \nonumber
\end{align*}
where $\Delta,\Delta'$ are defined by $x^\top Y x = (1+\beta)(1+\Delta)$ and $x'^\top Y x' = (1+\beta)(1+\Delta')$. We will see below that $m$ is indeed only a function of $\langle x,x' \rangle$.

\subsubsection{Interval \texorpdfstring{$|\alpha| \in [\eps,1-\eps]$}{alpha in [eps,1-eps]}}
\label{sec:wish-main-interval}

Let $\alpha = \langle x,x' \rangle$. Let $\eps > 0$ be a small constant (not depending on $n$), to be chosen later. First let us focus on the contribution from $|\alpha| \in [\eps,1-\eps]$, i.e.\ we want to bound
$$M_1 \defeq \Ex_\alpha \left[\one_{|\alpha| \in [\eps,1-\eps]}\, m(\alpha) \right].$$

For $Y \sim P_n$ and with $x,x'$ fixed unit vectors, the matrix
$$\left(\begin{array}{cc} N x^\top Y x & N x^\top Y x' \\ N x^\top Y x' & N x'^\top Y x' \end{array}\right)$$
follows the $2 \times 2$ Wishart distribution with $N$ degrees of freedom and shape matrix
$$\left(\begin{array}{cc} 1 & \alpha \\ \alpha & 1 \end{array}\right),\qquad \alpha = \langle x,x' \rangle.$$

By integrating over $c = x^\top Y x'$ and using the PDF of the Wishart distribution, we have
\begin{align*}
m(\alpha) = \iiint & (1+\beta)^2 \exp\Big\{N\Big[-\log(1+\beta) + \beta\left(1 + \frac{\Delta}{2} + \frac{\Delta'}{2}\right) \\
&+ \left(\frac{1}{2} - \frac{3}{N}\right) \log((1+\beta)^2(1+\Delta)(1+\Delta')-c^2) \\
&- \frac{1}{1-\alpha^2}\left((1+\beta)\left(1+\frac{\Delta}{2}+\frac{\Delta'}{2}\right)-\alpha c\right) - \frac{1}{2} \log(1-\alpha^2) \\
&+ \log(N/2) - \frac{1}{N}\log \Gamma_2(N/2)\Big]\Big\} \,\dee c\, \dee \Delta\, \dee \Delta'
\end{align*}
where the integration is over the domain $|\Delta| \le \eta$, $|\Delta'| \le \eta$, and $|c| \le {(1+\beta)\sqrt{(1+\Delta)(1+\Delta')}}$, and $\Gamma_2$ denotes the multivariate gamma function.

Using $\eta = o(1)$ and applying Stirling's approximation to $\Gamma_2$, we have for $|\alpha| \in [\eps,1-\eps]$,
\begin{align*}
m(\alpha) \le \max_{|c| \le 1+\beta} & (1+\beta)^2 \exp\Big\{N\Big[-\log(1+\beta) + \beta + \frac{1}{2} \log((1+\beta)^2-c^2) \\
&- \frac{1+\beta-\alpha c}{1-\alpha^2} - \frac{1}{2} \log(1-\alpha^2) + 1 + o(1) \Big]\Big\}
\end{align*}
where the $o(1)$ is uniform in $\alpha$. Letting $w = c/(1+\beta)$ and solving explicitly for the optimal $w$,
\begin{align*}
m(\alpha) \le m_1(\alpha) \defeq (1+\beta)^2  \exp\Big\{N\Big[&(1+\beta) \frac{\alpha(w-\alpha)}{1-\alpha^2} \\ 
&+ \frac{1}{2}\log\left(\frac{1-w^2}{1-\alpha^2}\right) + o(1) \Big]\Big\}
\end{align*}
$$\text{where} \qquad w = w(\alpha) = \pm\sqrt{A^2+1}-A \;\text{ with }\; A = \frac{1-\alpha^2}{2\alpha(\beta+1)}$$
and $\pm$ has the same sign as $\alpha$.

We now show how to bound the contribution to $M_1$ from positive $\alpha$; the proof for negative $\alpha$ is similar. We have
\begin{align*}
&\Ex_\alpha \left[\one_{\alpha \in [\eps,1-\eps]}\, m_1(\alpha) \right] = \int_0^\infty \problr{\one_{\alpha \in [\eps,1-\eps]}\, m_1(\alpha) \ge u} \dee u \\
&= \int_0^\infty \problr{\alpha \in [\eps,1-\eps] \text{ and } m_1(\alpha) \ge u} \dee u \\
&= m_1(\eps) \problr{\alpha \in [\eps,1-\eps]} + \int_{m_1(\eps)}^{m_1(1-\eps)} \problr{\alpha \in [\eps,1-\eps] \text{ and } m_1(\alpha) \ge u} \dee u.
\intertext{Since $m_1(\alpha)$ is strictly increasing on $[0,1]$ (see Appendix~\ref{app:prop-F}), we can apply the change of variables $u = m_1(t)$ to obtain}
&= m_1(\eps) \problr{\alpha \in [\eps,1-\eps]} + \hspace{-2pt}\int_\eps^{1-\eps} \hspace{-4pt} \problr{\alpha \in [\eps,1-\eps] \text{ and } \alpha \ge t} m_1(t) \,O(N)\, \dee t \\
&\le m_1(\eps) \problr{\alpha \ge \eps} + O(N) \int_\eps^{1-\eps} \problr{\alpha \ge t} m_1(t) \dee t.
\end{align*}

\noindent Plugging in the rate function to bound $\problr{\alpha \ge \eps}$ and $\prob{\alpha \ge t}$, we obtain $M_1 = o(1)$ provided that (\ref{eq:nc-result}) holds. The contribution from negative $\alpha$ yields the same condition (\ref{eq:nc-result}) due to the symmetry $w(-\alpha) = -w(\alpha)$ and $m_1(-\alpha) = m_1(\alpha)$.

\subsubsection{Interval \texorpdfstring{$|\alpha| \in [0,\eps)$}{alpha in [0,eps)}}

This case needs special consideration because both sides of (\ref{eq:nc-result}) approach 0 as $t \to 0$ and so the last step above requires $\alpha$ to be bounded away from 0. Since (up to a factor of $1+o(1)$) conditioning $Q_n$ on $\Omega(x,Y)$ only decreases the second moment (for each value of $\alpha$), we can revert back to the basic second moment: the contribution $M_2 \defeq \Ex_\alpha \left[\one_{|\alpha| \in [0,\eps)}\, m(\alpha) \right]$ is bounded by the small deviations $S(\eps^2)$ from Section~\ref{sec:wish-small-dev}. It therefore follows from either Lemma~\ref{lem:wish-small-dev-1} or Lemma~\ref{lem:wish-small-dev-2} that provided $\eps$ is small enough, $M_2$ is bounded as $n \to \infty$.

\subsubsection{Interval \texorpdfstring{$|\alpha| \in (1-\eps,1]$}{alpha in (1-eps,1]}}

This case needs special consideration because in the calculations for the $[\eps,1-\eps]$ interval, certain terms in the exponent blow up at $|\alpha|=1$ which prevents us from replacing $\Delta,\Delta'$ by an error term that is $o(1)$ uniformly in $\alpha$. To deal with this case we will bound $m(\alpha)$ by its worst-case value $m(1)$.

To see that $m(1)$ is the worst case, notice from (\ref{eq:cond-2nd}) that up to an $\exp(o(N))$ factor (which will turn out to be negligible), $m(\alpha)$ is proportional to $\prob{|\Delta| \le \eta \text{ and } |\Delta'| \le \eta}$. Since $N x^\top Y x$ and $N x'^\top Y x'$ each follow at $\chi_N^2$ distribution (with correlation that increases with $|\alpha|$), this probability is maximized when they are perfectly correlated at $|\alpha| = 1$.

We now proceed to bound $m(1)$. Let $Y \sim P_n$, and let $x,x'$ be fixed unit vectors with $|\alpha| = 1$. We have that $N x^\top Y x$ follows a $\chi_N^2$ distribution, with $N x'^\top Y x' = N x^\top Y x$. Similarly to the computation for $[\eps,1-\eps]$ we obtain
$$m(1) \le m_3 \defeq (1+\beta) \exp\left\{N\left[-\frac{1}{2} \log(1+\beta) - \frac{1}{2}(1-\beta) + \frac{1}{2} + o(1)\right]\right\}$$
and
$$M_3 \defeq \Ex_\alpha \left[\one_{|\alpha| \in (1-\eps,1]}\, m(\alpha) \right] \le \exp(o(N)) \prob{|\alpha| \ge 1-\eps}\, m_3.$$
Plugging in the rate function, $M_3$ is $o(1)$ provided that $\gamma f_\cX(1-\eps) > -\frac{1}{2}\log(1+\beta) - \frac{1}{2}(1-\beta) + \frac{1}{2}$. This follows from (\ref{eq:nc-result}) (near $t=1$) provided $\eps$ is small enough (since $f_\cX$ is an increasing function of $t$).

\section*{Acknowledgements}
The authors are indebted to Philippe Rigollet for helpful discussions and for many comments on a draft. We thank the anonymous reviewers for many helpful, detailed comments.

% cite supplement
\begin{supplement}
\sname{Supplement A}\label{supp}
\stitle{Optimality and Sub-optimality of PCA in Spiked Random Matrix Models: Supplementary Proofs}
%\slink[url]{http://ameliaperry.me/downloads/contig-supp-v2.pdf}
\sdescription{Included below as appendices. Contains proofs omitted from this paper for the sake of length.}
\end{supplement}

\appendix

\section{Behavior near criticality}\label{app:critical}

In Gaussian Wigner settings where we have established contiguity for all $\lambda<1$, it is natural to ask whether the spiked and unspiked models remain contiguous for a sequence $\lambda = 1 + \delta_n$ with $\delta_n \to 0$ (here $\delta_n$ may be positive or negative). However, this is never the case; it is possible to consistently distinguish the models in this critical case. By adding additional GOE noise, we can reduce to the case $\lambda = 1 - \eps$ for arbitrary fixed $\eps > 0$ (perhaps taking a tail of the sequence). It is known \citep{jo-testing-spiked} that (regardless of the spike prior) the hypothesis testing error (sum of type I and type II errors) in this case tends to $0$ as $\eps \to 0$; thus the minimum hypothesis testing error in the original problem cannot be bounded away from zero.

% technically Onatski result is for ||x|| = 1 but we can reduce to this case by absorbing ||x|| into delta_n

A similar result for the positively-spiked ($\beta > 0$) Wishart model follows from \citet{sphericity}: if $\gamma$ is fixed and $\beta = \sqrt{\gamma} + \delta_n$ with $\delta_n \to 0$ then it is possible to consistently distinguish the spiked and unspiked models. (We expect the analogous result to hold for $\beta < 0$ but to the best our knowledge this has not been proven.)

\section{Bounds on hypothesis testing}\label{app:hyptest}

For both the Gaussian Wigner and Wishart models, for the spherical prior (or equivalently, limited to spectral-based tests) the optimal tradeoff curve (power envelope) between type I and type II error is known exactly in the $n\to\infty$ limit \citep{sphericity,jo-testing-spiked}. For other priors, one can apply the optimal spectral-based test from above to obtain an upper bound; however, better tests (which do not depend only on the spectrum) may be possible.

In many cases we can use Proposition~\ref{prop:hyptest} to obtain lower bounds (which do not match the upper bound above). First note that Proposition~\ref{prop:hyptest} is still valid (in the $n\to\infty$ limit) in cases when we have used the \emph{conditional} second moment. (This is because if $\tilde Q_n$ is obtained from $Q_n$ by conditioning on a $(1-o(1))$-probability event, asymptotic hypothesis testing bounds for $\tilde Q_n$ against $P_n$ imply the same bounds for $Q_n$ against $P_n$.)

For the Gaussian Wigner model, Theorem~\ref{thm:subg-iid} (subgaussian method for \iid priors) and Theorem~\ref{thm:cond-method-app} (conditioning method) both give the limit value (as $n\to\infty$) of the (conditional) second moment, and in fact the value is $(1-\lambda^2)^{-1/2}$ in both of these cases. Therefore, any time we have used one of those two methods, we obtain asymptotic hypothesis testing bounds from Proposition~\ref{prop:hyptest}. This applies to, for instance, the \iid Gaussian, Rademacher, and sparse Rademacher priors. The same bounds also hold for the spherical prior (although the exact asymptotic power envelope is known in this case) because the comparison method of Proposition~\ref{prop:compare} preserves the value of the second moment.

For the Wishart model, suppose we have a prior for which we know the \emph{Wigner} second moment has limit value $(1-\lambda^2)^{-1/2}$ (as above). Furthermore, suppose we have a Wishart lower bound for this prior via Theorem~\ref{thm:wish-nc}, using the Wigner second moment to control the small deviations (i.e.\ condition (i) of Theorem~\ref{thm:wish-nc} holds). From the proof of Theorem~\ref{thm:wish-nc}, the limit value of the Wishart (conditional) second moment is determined by the small deviations of Section~\ref{sec:wish-small-dev}; the remaining large deviations contribute $o(1)$. We see from Section~\ref{sec:wish-small-dev} that the asymptotic value of the small deviations is bounded by the value of the Wigner second moment with $\lambda^2 = -\log(1-\eps^2\beta^2)/\gamma\eps^2 \to \beta^2/\gamma$ as $\eps \to 0$. Therefore the limsup of the Wishart (conditional) second moment is at most $(1-\beta^2/\gamma)^{-1/2}$, which yields hypothesis testing bounds via Proposition~\ref{prop:hyptest}.

\section{Alternative proof for spherically-spiked Wigner}\label{app:confluent}

Here we give an alternative proof of Corollary~\ref{cor:sphere-prior}. The proof deals with the second moment directly rather than comparing to the \iid Gaussian prior.

\begin{repcorollary}{cor:sphere-prior}
Consider the spherical prior $\cXs$. If $\lambda < 1$ then $\GWig(\lambda,\cXs)$ is contiguous to $\GWig(0)$.
\end{repcorollary}
\begin{proof}
By symmetry, we reduce the second moment to
$$ \Ex_{x,x'} \exp\left(\frac{n \lambda^2}{2} \langle x,x' \rangle^2\right) = \Ex_{x} \exp\left(\frac{n \lambda^2}{2} \langle x,e_1 \rangle^2\right) = \Ex_{x_1} \exp\left(\frac{n \lambda^2}{2} x_1^2\right), $$
where $e_1$ denotes the first standard basis vector. Note that the first coordinate $x_1$ of a point uniformly drawn from the unit sphere in $\RR^n$ is distributed proportionally to $(1-x_1^2)^{(n-3)/2}$, so that its square $y$ is distributed proportionally to $(1-y)^{(n-3)/2} y^{-1/2}$. Hence $y$ is distributed as $\mathrm{Beta}(\frac12,\frac{n-1}{2})$. The second moment is thus the moment generating function of $\mathrm{Beta}(\frac12,\frac{n-1}{2})$ evaluated at $n \lambda^2/2$, and as such, we have
\begin{equation} \Ex_{P_n}\left(\dd[Q_n]{P_n}\right)^2 = {}_1 F_1\left( \frac12 ; \frac{n}{2} ; \frac{\lambda^2 n}{2} \right),  \end{equation}
where ${}_1 F_1$ denotes the confluent hypergeometric function.

Suppose $\lambda < 1$. Equation~13.8.4 from [\citet{dlmf}] grants us that, as $n \to \infty$,
\begin{align*}
{}_1 F_1\left( \frac12 ; \frac{n}{2} ; \frac{\lambda^2 n}{2} \right) &= (1+o(1)) \left(\frac{n}{2}\right)^{1/4} e^{\zeta^2 n/8} \left( \lambda^2 \sqrt{\frac{\zeta}{1-\lambda^2}} U(0,\zeta \sqrt{n/2}) \right. \\
&\qquad \left. + \left( -\lambda^2 \sqrt{\frac{\zeta}{1-\lambda^2}} + \sqrt{\frac{\zeta}{1-\lambda^2}} \right) \frac{U(-1,\zeta\sqrt{n/2})}{\zeta\sqrt{n/2}} \right),
\intertext{where $\zeta = \sqrt{2(\lambda^2-1-2\log \lambda)}$ and $U$ is the parabolic cylinder function,}
&= (1+o(1)) \left(\frac{n}{2}\right)^{1/4} e^{\zeta^2 n/8} \left( \lambda^2 \sqrt{\frac{\zeta}{1-\lambda^2}} e^{-\zeta^2 n / 8} (\zeta \sqrt{n/2})^{-1/2} \right.\\
&\qquad \left. + \left( -\lambda^2 \sqrt{\frac{\zeta}{1-\lambda^2}} + \sqrt{\frac{\zeta}{1-\lambda^2}} \right) \frac{e^{-\zeta^2 n/8} (\zeta \sqrt{n/2})^{1/2}}{\zeta\sqrt{n/2}} \right),
\intertext{by Equation~12.9.1 from [\citet{dlmf}],}
&= (1+o(1)) (1-\lambda^2)^{-1/2},
\end{align*}
which is bounded as $n \to \infty$, for all $\lambda < 1$. The result follows from Lemma~\ref{lem:sec}.
\end{proof}

\section{Conditioning method for Gaussian Wigner model}
\label{app:cond-method}

In this section we give the full details of the conditioning method for the Gaussian Wigner model.
We assume that the prior is $\mathcal{X} = \IID(\pi/\sqrt{n})$ where $\mathcal{\pi}$ is a finitely-supported distribution on $\RR$ with mean zero and variance one.

The argument that we will use is based on \citet{bmnn}, in particular their Proposition~5. Suppose $\omega_n$ is a set of `good' $x$ values so that $x \in \omega_n$ with probability $1-o(1)$. Let $Q_n = \GWig_n(\lambda,\cX)$ and let $P_n = \GWig_n(0)$. Let $\tilde\cX_n$ be the conditional distribution of $\cX_n$ given $\omega_n$. Let $\tilde Q_n = \GWig_n(\lambda,\tilde \cX)$. Our goal is to show $\tilde Q_n \contig P_n$,  from which it follows that $Q_n \contig P_n$ (see Lemma~\ref{lem:cond}). 
%In our case, the bad events are when the empirical distribution of $x$ differs significantly from $\pi$, i.e.\ $x$ has atypical proportions of entries.  %%% this is vague
If we let $\Omega_n$ be the event that $x$ and $x'$ are both in $\omega_n$, our second moment becomes
\begin{align*}
\Ex_{P_n}\left(\dd[\tilde Q_n]{P_n}\right)^2 &= \Ex_{\tilde x,\tilde x' \sim \tilde \cX}\left[\exp\left(\frac{n \lambda^2}{2} \langle \tilde x,\tilde x' \rangle^2\right)\right] \\
&= (1+o(1))\Ex_{x,x' \sim \cX}\left[\one_{\Omega_n}\exp\left(\frac{n \lambda^2}{2} \langle x,x' \rangle^2\right)\right].
\end{align*}

Let $\Sigma \subseteq \mathbb{R}$ (a finite set) be the support of $\pi$, and let $s = |\Sigma|$. We will index $\Sigma$ by $[s] = \{1,2,\ldots,s\}$ and identify $\pi$ with the vector of probabilities $\pi \in \mathbb{R}^s$. For $a,b \in \Sigma$, let $N_{ab}$ denote the number of indices $i$ for which $x_i = a/\sqrt{n}$ and $x'_i = b/\sqrt{n}$ (recall $x_i$ is drawn from $\pi/\sqrt{n}$). Note that $N$ follows a multinomial distribution with $n$ trials, $s^2$ outcomes, and with probabilities given by $\bar\alpha = \pi \pi^\top \in \mathbb{R}^{s \times s}$. We have
$$\frac{n\lambda^2}{2}\langle x,x' \rangle^2 = \frac{\lambda^2}{2n}\left(\sum_{a,b \in \Sigma} a b N_{ab}\right)^2 = \frac{\lambda^2}{2n}\sum_{a,b,a',b'} aba'b' N_{ab} N_{a'b'} = \frac{1}{n}N^\top A N$$
where $A$ is the $s^2 \times s^2$ matrix $A_{ab,a'b'} = \frac{\lambda^2}{2} aba'b'$, and the quadratic form $N^\top A N$ is computed by treating $N$ as a vector of length $s^2$.

We are now in a position to apply Proposition~5 from \citet{bmnn}. Define $Y = (N - n \bar\alpha)/\sqrt n$. Let $\Omega_n$ be the event defined in Appendix~A of \citet{bmnn}, which enforces that the empirical distributions of $\sqrt{n} x$ and $\sqrt{n} x'$ are close to $\pi$; namely,
$$\max_j \left| \sum_i N_{ij} - n \pi_j \right| \le \eta_n \quad \text{and} \quad \max_i \left| \sum_j N_{ij} - n \pi_i \right| \le \eta_n$$
where (for concreteness) $\eta_n = \sqrt{n} \log n$.

Note that $\bar\alpha$ (treated as a vector of length $s^2$) is in the kernel of $A$ because $\pi$ is mean-zero: the inner product between $\bar\alpha$ and the $(a,b)$ row of $A$ is
$$\sum_{a',b'} A_{ab,a'b'} \bar\alpha_{a'b'} = \frac{\lambda^2}{2} \sum_{a',b'} aba'b'\pi_{a'}\pi_{b'} = \frac{\lambda^2}{2} \,ab\left(\sum_{a'} a' \pi_{a'}\right)\left(\sum_{b'} b' \pi_{b'}\right) = 0.$$
Therefore we have $\frac{1}{n}N^\top A N = Y^\top A Y$ and so we can write our second moment as $(1+o(1))\EE[\one_{\Omega_n} \exp(Y^\top A Y)]$.

Let $\Delta_{s^2}(\pi)$ denote the set of nonnegative vectors $\alpha \in \mathbb{R}^{s^2}$ with row- and column-sums prescribed by $\pi$, i.e.\ treating $\alpha$ as an $s \times s$ matrix, we have (for all $i$) that row $i$ and column $i$ of $\alpha$ each sum to $\pi_i$. Let $D(u,v)$ denote the KL divergence between two vectors: $D(u,v) = \sum_i u_i \log(u_i/v_i)$. For convenience, we restate Proposition~5 in \citet{bmnn}.

\begin{proposition}[\citet{bmnn}, Proposition~5]
\label{prop:nn}
Let $\pi \in \mathbb{R}^s$ be any vector of probabilities. Let $A$ be any $s^2 \times s^2$ matrix. Define $N$, $Y$, $\bar\alpha$, and $\Omega_n$ as above (depending on $\pi$). Let
$$m = \sup_{\alpha \in \Delta_{s^2}(\pi)} \frac{(\alpha-\bar\alpha)^\top A (\alpha-\bar\alpha)}{D(\alpha,\bar\alpha)}.$$
If $m < 1$ then
$\lim_{n \to \infty} \EE[\one_{\Omega_n} \exp(Y^\top A Y)] = \EE[\exp(Z^\top A Z)] < \infty$,
where $Z \sim \cN(0,\diag(\bar\alpha) - \bar\alpha \bar\alpha^\top)$. If $m > 1$ then
$\lim_{n \to \infty} \EE[\one_{\Omega_n} \exp(Y^\top A Y)] = \infty$.
\end{proposition}

We apply Proposition~\ref{prop:nn} to our specific choice of $\pi$ and $A$:
\begin{theorem}[conditioning method]
\label{thm:cond-method-app}
Let $\mathcal{X} = \mathrm{iid}(\pi)$ where $\pi$ has mean zero, unit variance, and finite support $\Sigma \subseteq \mathbb{R}$ with $|\Sigma| = s$. Let $Q_n = \GWig_n(\lambda,\cX)$, $\tilde Q_n = \GWig_n(\lambda,\tilde \cX)$, and $P_n = \GWig_n(0)$. Define the $s \times s$ matrix $\beta_{ab} = ab$ for $a,b \in \Sigma$. Let
$$\overline{\lambda}_\cX = \left[\sup_{\alpha \in \Delta_{s^2}(\pi)} \frac{\langle \alpha,\beta \rangle^2}{2D(\alpha,\bar\alpha)}\right]^{-1/2}. $$
If $\lambda < \overline{\lambda}_\cX$ then
$\lim_{n \to \infty} \EE_{P_n}(\ddflat[\tilde Q_n]{P_n})^2 = (1-\lambda^2)^{-1/2} < \infty$
and so $Q_n \contig P_n$. Conversely, if $\lambda > \overline{\lambda}_\cX$ then
$\lim_{n \to \infty} \EE_{P_n}\left(\ddflat[\tilde Q_n]{P_n}\right)^2 = \infty.$
\end{theorem}

\noindent Note that this is a tight characterization of when the conditional second moment is bounded, but not necessarily of when contiguity holds.

% intuition
The intuition behind this matrix optimization problem is the following. The matrix $\alpha$ represents the `type' of a pair of spikes $(x,x')$ in the sense that for any $a,b \in \Sigma$, $\alpha_{ab}$ is the fraction of entries $i$ for which $x_i = a$ and $x'_i = b$. A pair $(x,x')$ of type $\alpha$ contributes the value $\exp(\frac{n \lambda^2}{2} \langle \alpha,\beta \rangle^2)$ to the second moment $\EE_{x,x'} \exp\left(\frac{n\lambda^2}{2} \langle x,x' \rangle^2\right)$. The probability (when $x,x' \sim \IID(\pi/\sqrt{n})$) that a particular type $\alpha$ occurs is asymptotically $\exp(-n D(\alpha,\bar\alpha))$. Due to the exponential scaling, the second moment is dominated by the worst $\alpha$ value: the second moment is unbounded if there is some $\alpha$ such that $\frac{\lambda^2}{2} \langle \alpha,\beta \rangle^2 > D(\alpha,\bar\alpha)$. (This idea is often referred to as Laplace's method or the saddle point method.) Rearranging this yields the optimization problem in the theorem. The fact that we are conditioning on `good' values of $x$ (that have close-to-typical proportions of entries) allows us to add the constraint $\alpha \in \Delta_{s^2}(\pi)$. If we were not conditioning, we would have the same optimization problem over $\alpha \in \Delta_{s^2}$ (the simplex of dimension $s^2$), which in some cases gives a worse threshold.

Unfortunately we do not have a good general technique to understand the value of the matrix optimization problem. However, in certain special cases we do. Namely, in Appendix~\ref{app:sparse-rad} we show, for the sparse Rademacher prior, how to use symmetry to reduce the problem to only two variables so that it can be easily solved numerically. In other applications, closed form solutions to related optimization problems have been found \citep{an-chrom,bmnn}.

% Compute value of moment
Above we have computed the limit value of the second moment in the case $\lambda < \overline{\lambda}_\cX$ as follows. Defining $Z$ as in Proposition~\ref{prop:nn} we have $\langle Z, \beta \rangle \sim \cN(0,\sigma^2)$ where
$$
\sigma^2 
= \beta^\top (\diag(\bar\alpha) - \bar\alpha \bar\alpha^\top)\beta 
= \sum_{ab} \beta_{ab}^2 \bar\alpha_{ab} + \left(\sum_{ab} \beta_{ab} \bar\alpha_{ab} \right)^2 $$
$$ = \left(\sum_a a^2 \pi_a\right)\left(\sum_b b^2 \pi_b\right) + \left(\sum_a a \pi_a \sum_b b \pi_b \right)^2 = 1, $$
since $\pi$ is mean-zero and unit-variance, and so
$$\EE[\exp(Z^\top A Z)] = \EE\left[\exp\left(\frac{\lambda^2}{2}\langle Z, \beta \rangle^2\right)\right] = \EE\left[\exp\left(\frac{\lambda^2}{2}\chi_1^2\right)\right] = (1-\lambda^2)^{-1/2}.$$

\section{Sparse Rademacher prior}
\label{app:sparse-rad}

In this section we give details for our results on the spiked Gaussian Wigner model with the \iid sparse Rademacher prior: $\IID(\pi/\sqrt{n})$ where $\pi = \sqrt{1/\rho}\,\mathcal{R}(\rho)$ where $\mathcal{R}(\rho)$ is the sparse Rademacher distribution with sparsity $\rho \in (0,1]$:
$$\mathcal{R}(\rho) = \left\{\begin{array}{ccc} 0 & \text{w.p.} & 1-\rho \\ +1 & \text{w.p.} & \rho/2 \\ -1 & \text{w.p.} & \rho/2 \end{array}\right..$$
First we apply the subgaussian method (Theorem~\ref{thm:subg-iid}). The subgaussian constant $\sigma^2$ for $\pi$ needs to satisfy
\begin{equation} \exp\left(\frac{1}{2} \sigma^2 t^2\right) \ge \EE \exp(t \pi) = 1-\rho + \rho \cosh(t/\sqrt{\rho}) \label{eq:sprad-mgf-compare}\end{equation}
for all $t \in \mathbb{R}$ so the best (smallest) choice for $\sigma^2$ is
$$(\sigma^*)^2 \defeq \sup_{t \in \mathbb{R}} \frac{2}{t^2}\log\left[1-\rho + \rho \cosh(t/\sqrt{\rho}) \right].$$

\noindent Recall that Theorem~\ref{thm:subg-iid} (subgaussian method) gives contiguity for all $\lambda < 1/\sigma^*$. We now show that for sufficiently large $\rho$, we have $\sigma^* = 1$, implying that PCA is tight:
\begin{proposition}
When $\rho \geq 1/3$, we have $\sigma^* = 1$, yielding contiguity for all $\lambda < 1$. On the other hand, if $\rho < 1/3$, then $\sigma^* > 1$.
\end{proposition}

\begin{proof}
We equivalently consider the following reformulation of (\ref{eq:sprad-mgf-compare}):
\begin{equation} \frac12 \sigma^2 t^2 \stackrel{?}{\geq} \log\left( 1-\rho + \rho \cosh(t/\sqrt{\rho}) \right) \defeq k_\rho(t). \label{eq:sprad-cgf-compare} \end{equation}
Both sides of the inequality are even functions of $t$, agreeing in value at $t=0$. When $\sigma^2 < 1$, the inequality fails, by comparing their second-order behavior about $t=0$. When $\sigma^2 = 1$ but $\rho < 1/3$, the inequality fails, as the two sides have matching behavior up to third order, but $k^{(4)}_\rho(0) = 3 - 1/\rho < 0$.

It remains to show that the inequality (\ref{eq:sprad-cgf-compare}) does hold for $\rho > 1/3$ and $\sigma^2 = 1$. As the left and right sides agree to first order at $t=0$, and are both even functions, it suffices to show that for all $t \geq 0$,
$$ 1 \stackrel{?}{\geq} k_\rho''(t) = \frac{\rho + (1-\rho) \cosh(t/\sqrt{\rho})}{(1-\rho+\rho \cosh(t/\sqrt{\rho}))^2}. $$
Completing the square for $\cosh$, we have the equivalent inequality:
$$ 0 \stackrel{?}{\leq} 1 - 3\rho + \rho^2 + \Big( \underbrace{\rho \cosh(t/\sqrt{\rho}) + \frac{(2\rho-1)(1-\rho)}{2\rho}}_{(*)} \Big)^2 - \frac{(2\rho - 1)^2(1-\rho)^2}{4\rho^2}. $$
Note that $\cosh$ is bounded below by $1$; thus for $\rho > 1/3$, the underbraced term ($*$) is nonnegative, and hence minimized in absolute value when $t=0$. It then suffices to show the above inequality in the case $t=0$, so that $\cosh(t/\sqrt{\rho}) = 1$; but here the inequality is an equality, by simple algebra.
\end{proof}

Using the conditioning method of Section~\ref{sec:cond-method}, we will now improve the range of $\rho$ for which PCA is optimal, although our argument here relies on numerical optimization.

\begin{example}\label{ex:sparse-rad}
Let $\cX$ be the sparse Rademacher prior $\IID(\sqrt{1/\rho} \,\mathcal{R}(\rho))$. There exists a critical value $\rho^* \approx 0.184$ (numerically computed) such that if $\rho \ge \rho^*$ and $\lambda < 1$ then $\GWig(\lambda,\cX)$ is contiguous to $\GWig(0,\cX)$. When $\rho < \rho^*$ we are only able to show contiguity when $\lambda < \lambda^*_\rho$ for some $\lambda^*_\rho < 1$.
\end{example}

\begin{proof}[Details]
Consider the optimization problem of Theorem~\ref{thm:cond-method} (conditioning method). We will first use symmetry to argue that the optimal $\alpha$ must take a simple form. Abbreviate the support of $\pi$ as $\{0,+,-\}$. For a given $\alpha$ matrix, define its complement by swapping $+$ and $-$, e.g.\ swap $\alpha_{0+}$ with $\alpha_{0-}$ and swap $\alpha_{-+}$ with $\alpha_{+-}$. Note that if we average $\alpha$ with its complement, the numerator $\langle \alpha,\beta \rangle^2$ remains unchanged, the denominator $D(\alpha,\bar\alpha)$ can only decrease, and the row- and column-sum constraints remain satisfied; this means the new solution is at least as good as the original $\alpha$. Therefore we only need to consider $\alpha$ values satisfying $\alpha_{++} = \alpha_{--}$ and $\alpha_{+-} = \alpha_{-+}$. Note that the remaining entries of $\alpha$ are uniquely determined by the row- and column-sum constraints, and so we have reduced the problem to only two variables. It is now easy to solve the optimization problem numerically, say by grid search.
%The result is that we have contiguity for all $\lambda < 1$ provided $\rho$ exceeds a new critical value $\rho^* \approx 0.184$, an improvement over the subgaussian method.
\end{proof}

\section{Proof of non-Gaussian Wigner lower bound}\label{app:nong-lower}

In this section we prove Theorem~\ref{thm:nongauss-lower}, and verify its hypotheses for spherical and \iid priors.

\begin{reptheorem}{thm:nongauss-lower}
Under Assumption~\ref{as:nong-lower}, $\Wig(\lambda,\cP,\cP_d,\cX)$ is contiguous to $\Wig(0,\cP,\cP_d)$ for all $\lambda < \lambda^*_\cX/\sqrt{F_\cP}$.
\end{reptheorem}

\begin{proof}
We begin by conditioning the prior $\cX$ on the high-probability events that $\|x\|_q^q \leq \alpha_q n^{\frac1q - \frac12}$ for $q=2,4,6,8$, and on the event that no entry of $x$ exceeds $5 \sqrt{\log n / n}$, which is true with high probability by the subgaussian hypothesis; let $\tilde\cX$ be this conditioned prior. Hence if $\Wig(\lambda,\cP,\tilde\cX)$ is contiguous to $\Wig(0,\cP)$ then so is $\Wig(\lambda,\cP,\cP_d,\cX)$. Let $Q_n = \Wig_n(\lambda,\cP,\cP_d,\cX)$, $\tilde Q_n = \Wig_n(\lambda,\cP,\cP_d,\tilde\cX)$, and $P_n = \Wig_n(0,\cP,\cP_d)$.

For convenience, let $p_{ij}$ denote $p$ if $i\neq j$ and $p_d$ if $i = j$, the density of the noise on the $ij$ entry. Likewise let $\tau_{ij}$ denote $\tau$ or $\tau_d$ as appropriate. We proceed from the second moment:
\begin{align*}
\Ex_{P_n} & \left(\dd[\tilde Q_n]{P_n}\right)^2
= \Ex_{Y \sim P_n}\left[ \Ex_{x,x' \sim \tilde\cX} \prod_{i \leq j} \frac{p_{ij}(\sqrt{n} Y_{ij} - \lambda \sqrt{n} x_i x_j)}{p_{ij}(\sqrt{n} Y_{ij})} \frac{p_{ij}(\sqrt{n} Y_{ij} - \lambda \sqrt{n}  x_i' x_j')}{p_{ij}(\sqrt{n} Y_{ij})} \right] \\
&= \Ex_{x,x' \sim \tilde\cX}\left[ \prod_{i \leq j} \Ex_{\sqrt{n} Y_{ij} \sim \cP} \frac{p_{ij}(\sqrt{n} Y_{ij} - \lambda\sqrt{n} x_i x_j)}{p_{ij}(\sqrt{n} Y_{ij})} \frac{p_{ij}(\sqrt{n} Y_{ij} - \lambda \sqrt{n} x_i' x_j')}{p_{ij}(\sqrt{n} Y_{ij})} \right] \\
&= \Ex_{x,x' \sim \tilde\cX}\left[ \exp\left( \sum_{i \leq j} \tau_{ij}(\lambda \sqrt{n} x_i x_j, \lambda \sqrt{n} x_i' x_j') \right) \right].
\end{align*}

We will expand $\tau$ and $\tau_d$ using Taylor's theorem, using the $C^4$ assumption:
\begin{align*}
\tau(a,b) =& \sum_{0 \leq k+\ell \leq 3} \frac{1}{(k+\ell)!} \frac{\partial^{k+\ell} \tau}{\partial a^k \partial b^\ell}(0,0)\; a^k b^\ell \\
&\quad + \sum_{k+\ell=4} \frac{1}{4!} \left( \frac{\partial^4 \tau}{\partial a^k \partial b^\ell}(0,0) + h_{k,\ell}(a,b) \right) a^k b^\ell
\end{align*}
\begin{align*}
\tau_d(a,b) =& \sum_{0 \leq k+\ell \leq 1} \frac{1}{(k+\ell)!} \frac{\partial^{k+\ell} \tau}{\partial a^k \partial b^\ell}(0,0)\; a^k b^\ell \\
&\quad + \sum_{k+\ell=2} \frac{1}{2!} \left( \frac{\partial^2 \tau}{\partial a^k \partial b^\ell}(0,0) + h_{d;k,\ell}(a,b) \right) a^k b^\ell
\end{align*}
for some remainder function $h_{k,\ell}(a,b)$ tending to $0$ as $(a,b) \to (0,0)$.
As $x$ and $x'$ are entrywise $O(\sqrt{\log n / n})$, these remainder terms $h_{k,\ell}(\lambda\sqrt{n} x_i x_j, \lambda\sqrt{n} x_i' x_j')$ are $o(1)$ as $n \to \infty$.
Note that $\tau(a,0) = 0 = \tau(0,b)$, so that the non-mixed partials of $\tau$ vanish, and likewise for $\tau_d$. We note also that $\frac{\partial^2 \tau}{\partial a \partial b}(0,0) = F_\cP$, the Fisher information defined above. Thus,
\begin{align*}
\Ex_{P_n} \left(\dd[\tilde Q_n]{P_n}\right)^2 =& \hspace{-2pt} \Ex_{x,x' \sim \tilde \cX}\bigg[ \exp\bigg( F_\cP \lambda^2 n \sum_{i < j} x_i x_j x_i' x_j' \\
&+ \sum_{\substack{k+\ell=3 \\ k,\ell>0}} \frac{\partial^3 \tau}{\partial a^k \partial b^\ell}(0,0) \frac{\lambda^3 n^{3/2}}{k!\ell!} \sum_{i < j} x_i^k x_j^k (x_i')^\ell (x_j')^\ell \\
&+ \sum_{k+\ell=4} \left( \frac{\partial^4 \tau}{\partial a^k \partial b^\ell}(0,0) + o(1) \right) \frac{\lambda^4 n^2}{k!\ell!} \sum_{i < j} x_i^k x_j^k (x_i')^\ell (x_j')^\ell \\
&+ \left( \frac{\partial^2 \tau_d}{\partial a \partial b}(0,0) + o(1) \right) \lambda^2 n \sum_{i} x_i^2 (x_i')^2
\quad \bigg) \bigg].
%&\leq \hspace{-2pt} \Ex_{x,x' \sim \tilde\cX}\left[ \exp\left( \frac{F_\cP \lambda^2 n}{2} \langle x,x' \rangle^2 \right) \prod_{k+\ell=4} \exp\left( \left( \frac{\partial^4 \tau}{\partial a^k \partial b^\ell}(0,0) + o(1) \right) \frac{\lambda^4 n^2}{2 k! \ell!} \langle x^k, (x')^\ell \rangle^2 \right) \right],
\end{align*}
We can separate these four terms using a weighted AM--GM inequality. For all $\eps > 0$:
\begin{align*}
\Ex_{P_n} \left(\dd[\tilde Q_n]{P_n}\right)^2 &\leq \Ex_{x,x' \sim \tilde \cX} \exp\left( (1-\eps)^{-1} F_\cP \lambda^2 n \sum_{i < j} x_i x_j x_i' x_j' \right) \numberthis\label{eq:term1}\\
&+ \sum_{\substack{k+\ell=3 \\ k,\ell>0}} \Ex_{x,x' \sim \tilde \cX} \exp\left( \frac{8}{\eps} \frac{\partial^3 \tau}{\partial a^k \partial b^\ell}(0,0) \frac{\lambda^3 n^{3/2}}{k!\ell!} \sum_{i < j} x_i^k x_j^k (x_i')^\ell (x_j')^\ell \right) \numberthis\label{eq:term2}\\
&+ \sum_{k+\ell=4} \Ex_{x,x' \sim \tilde \cX} \exp\left( \frac{8}{\eps} \left( \frac{\partial^4 \tau}{\partial a^k \partial b^\ell}(0,0) + o(1) \right) \frac{\lambda^4 n^2}{k!\ell!} \sum_{i < j} x_i^k x_j^k (x_i')^\ell (x_j')^\ell \right) \numberthis\label{eq:term3}\\
&+ \Ex_{x,x' \sim \tilde \cX} \exp\left( \frac{8}{\eps} \left( \frac{\partial^2 \tau_d}{\partial a \partial b}(0,0) + o(1) \right) \lambda^2 n \sum_{i} x_i^2 (x_i')^2 \right) \numberthis\label{eq:term4}
\end{align*}
so it suffices to control terms (\ref{eq:term1}--\ref{eq:term4}) individually.

By hypothesis, $\lambda < \lambda^*_\cX / \sqrt{F_\cP}$, implying that we can choose $\eps > 0$ such that $(1-\eps)^{-1} F_\cP \lambda^2 < (\lambda^*_\cX)^2$. But $\tilde\cX$ is dominated as a measure by $(1+o(1)) \cX$; it follows that $\lambda_\cX \leq \lambda_{\tilde \cX}$, and the first term (\ref{eq:term1}) is bounded.

% second term; subgaussian
We bound the second term (\ref{eq:term2}) using the subgaussian assumption:
\begin{align*}
(\ref{eq:term2}) &\leq 2 \Ex_{x,x' \sim \tilde \cX} \exp\left( \frac{2\lambda^3 n^{3/2}}{\eps} \frac{\partial^3 \tau}{\partial a^2 \partial b}(0,0) \langle x^2, x' \rangle^2 \right) \\
&= 2 \EE_{x \sim \tilde\cX} \EE_{x' \sim \tilde\cX} \exp(\langle v,x' \rangle^2) \\
&= 2 \EE_{x \sim \tilde\cX} (1+o(1)) \EE_{x' \sim \cX} \exp(\langle v,x' \rangle^2)
\end{align*}
where $v = \sqrt{2/\eps} \lambda^{3/2} n^{3/4} \sqrt{\frac{\partial^3 \tau}{\partial a^2 \partial b}(0,0)}\, x^2$. We thus have
$$ \|v\|_2^2 = \frac{2 \lambda^3 n^{3/2}}{\eps} \frac{\partial^3 \tau}{\partial a^2 \partial b}(0,0) \|x\|_4^4 = O(n^{1/2}). $$
By subgaussian hypothesis on $\cX$, the inner expectation over $x'$ is $O(1)$, so that the overall term (\ref{eq:term2}) is bounded.

We bound the third term (\ref{eq:term3}) using Cauchy--Schwarz:
\begin{align*}
(\ref{eq:term3}) &\leq \sum_{k+\ell=4} \Ex_{x,x' \sim \tilde\cX} \exp\left( \left( \frac{\partial^4 \tau}{\partial a^k \partial b^\ell}(0,0) + o(1) \right) \frac{8 \lambda^4 n^2}{2\eps k! \ell!} \langle x^k, (x')^\ell \rangle^2 \right) \\
&\leq \sum_{k+\ell=4} \Ex_{x,x' \sim \tilde\cX} \exp\left( \left( \frac{\partial^4 \tau}{\partial a^k \partial b^\ell}(0,0) + o(1) \right) \frac{8 \lambda^4 n^2}{2\eps k! \ell!} \|x^k\|_2^2\; \|(x')^\ell\|_2^2 \right) \\
&= \sum_{k+\ell=4} \Ex_{x,x' \sim \tilde\cX} \exp\left( \left( \frac{\partial^4 \tau}{\partial a^k \partial b^\ell}(0,0) + o(1) \right) \frac{8\lambda^4 n^2}{2\eps k! \ell!} \|x\|_{2k}^{2k}\; \|x'\|_{2\ell}^{2\ell} \right) \\
&\leq \sum_{k+\ell=4} \Ex_{x,x' \sim \tilde\cX} \exp\left( \left( \frac{\partial^4 \tau}{\partial a^k \partial b^\ell}(0,0) + o(1) \right) \frac{8\lambda^4 n^2}{2\eps k! \ell!} \alpha_{2k}^{2k} n^{1-k} \alpha_{2\ell}^{2\ell} n^{1-\ell} \right) \\
&= \sum_{k+\ell=4} \exp\left( \left( \frac{\partial^4 \tau}{\partial a^k \partial b^\ell}(0,0) + o(1) \right) \frac{8\lambda^4}{2\eps k! \ell!} \alpha_{2k}^{2k} \alpha_{2\ell}^{2\ell} \right),
\end{align*}
due to the norm restrictions on prior $\tilde\cX$. This evidently remains bounded as $n \to \infty$.

The fourth term proceeds similarly:
\begin{align*}
(\ref{eq:term4}) &\leq \Ex_{x,x' \sim \tilde\cX} \exp\left(  \frac{8 \lambda^2 n}{\eps} \left( \frac{\partial^2 \tau_d}{\partial a \partial b}(0,0) + o(1) \right) \langle x^2, (x')^2 \rangle \right) \\
&\leq \Ex_{x,x' \sim \tilde\cX} \exp\left(  \frac{8 \lambda^2 n}{\eps} \left( \frac{\partial^2 \tau_d}{\partial a \partial b}(0,0) + o(1) \right) \|x\|_4^2 \|x'\|_4^2 \rangle \right) \\
&\leq \Ex_{x,x' \sim \tilde\cX} \exp\left(  \frac{8 \lambda^2 n}{\eps} \left( \frac{\partial^2 \tau_d}{\partial a \partial b}(0,0) + o(1) \right) \alpha_4^4 n^{-1} \right)
\end{align*}
which likewise remains bounded.

With the overall second moment $\Ex_{P_n} \left(\dd[\tilde Q_n]{P_n}\right)^2$ bounded as $n \to \infty$, the result follows from Lemma~\ref{lem:cond}.
\end{proof}

\begin{repproposition}{prop:nong-lower-spherical}
Conditions (i) and (ii) in Assumption~\ref{as:nong-lower} are satisfied for the spherical prior $\cXs$.
\end{repproposition}
\begin{proof}
Note that one can sample $x \sim \cXs$ by first sampling $y \sim \N(0,I_n)$ and then taking $x = y/\|y\|_2$. By Chebyshev, $\left| \|y\|_2^2 - n\right| < n^{3/4}$ with probability $1-o(1)$.
For $q \in \{4,6,8\}$, $\|y\|_q^q$ has expectation $n(q-1)!!$ and variance $$n[(2q-1)!!-((q-1)!!)^2].$$ Supposing that $\|y\|_2^2 > n - n^{3/4} > n/2$, which occurs with probability $1-o(1)$, we have for any $\alpha_q$ that
\begin{align*}
\Pr[\|x_q\| > \alpha_q n^{\frac1q-\frac12}] &= \Pr[\|x\|_q^q > \alpha_q^q n^{1-\frac{q}2}] \\
&= \Pr[\|y\|_q^q > \alpha_q^q n^{1-\frac{q}2} \|y\|_2^q ] \\
&\leq \Pr[ \|y\|_q^q > \alpha_q^q 2^{-q/2} n] \\
&\leq \frac{n((2q-1)!! - ((q-1)!!)^2)}{n^2(2^{-q}\alpha_q^{2q} - (q-1)!!)^2},
\end{align*}
by Chebyshev. This probability is $o(1)$ so long as we take $\alpha_q^{2q} > 2^q (q-1)!!$.

The spherical prior is appropriately subgaussian: the inner product $\langle x,v \rangle$ is distributed as $2z-1$ with $z \sim \mathrm{Beta}(n/2,n/2)$, which is known to be $O(1/n)$-subgaussian (see e.g.\ \citet{sam}).
\end{proof}

\begin{repproposition}{prop:nong-lower-iid}
Consider an \iid prior $\cX = \IID(\pi/\sqrt{n})$ where $\pi$ is zero-mean, unit-variance,
%, has $\EE[\pi^{16}] < \infty$,
and subgaussian with some constant $\sigma^2$. Then conditions (i) and (ii) in Assumption~\ref{as:nong-lower} are satisfied.
\end{repproposition}
%\noindent An immediate implication of this is that the norm conditions are also satisfied for a `conditioned' prior which draws $x$ from $\IID(\pi)$ but then outputs zero if a 'bad' event occured.
\begin{proof}
We have $x_i = \frac{1}{\sqrt n} \pi_i$ where $\pi_i$ are independent copies of $\pi$. For $q \in \{2,4,6,8\}$,
\begin{align*}
\prob{\|x\|_q > \alpha_q n^{\frac{1}{q}-\frac{1}{2}}}
&= \prob{\|x\|_q^q > \alpha_q^q n^{1-\frac{q}{2}}} \\
&= \problr{\sum_i x_i^q > \alpha_q^q n^{1-\frac{q}{2}}} \\
&= \problr{\sum_i \pi_i^q > \alpha_q^q n} \\
&= \problr{\sum_i \pi_i^q - n \EE[\pi^q] > (\alpha_q^q - \EE[\pi^q]) n}. \\
\intertext{Choose $\alpha_q$ so that $C \equiv \alpha_q^q - \EE[\pi^q] > 0$, and apply Chebyshev's inequality:}
&\le \frac{\mathrm{Var}[\sum_i \pi_i^q]}{C^2 n^2} = \frac{n \mathrm{Var}[\pi^q]}{C^2 n^2} = \mathcal{O}(1/n).
\end{align*}
Here we needed $\EE[\pi^{2q}] < \infty$ (which follows from subgaussianity) so that $\mathrm{Var}[\pi^q] < \infty$.
\end{proof}

\section{Non-Gaussian Wigner with discrete noise}
\label{app:nong-discrete}

In this section we show that in the non-Gaussian Wigner model, if the noise distribution has a point mass then the detection problem becomes easy for any $\lambda > 0$.

\begin{theorem}
\label{thm:point-mass}
Let $\mathcal{P}$ be a (mean-zero, unit-variance) distribution on $\RR$ with a point mass: $\mathrm{Pr}_{w \sim \mathcal{P}}[w = c] = m$ for some $c$ and some $m > 0$. Let $\cP_d$ be any distribution on $\RR$. Let $\cX$ be a spike prior such that for some $\delta > 0$ and $\alpha > 0$, with probability $1-o(1)$, $x \sim \cX_n$ satisfies both (i) $\|x\|_0 \ge \delta n$ and (ii) $|x_i| \le n^{-1/4-\alpha}\; \forall i$. Then for any $\lambda > 0$, there exists a test that consistently distinguishes $\Wig(\lambda,\cP,\cP_d,\cX)$ from $\Wig(0,\cP,\cP_d)$.
\end{theorem}
\noindent Here, $\|x\|_0$ denotes the $\ell_0$ norm, i.e.\ the number of nonzero entries.
\begin{proof}
Let the test statistic $T(Y)$ be the fraction of entries of $Y$ that are exactly equal to $c/\sqrt{n}$. Under the unspiked model $Y \sim \Wig(0,\cP,\cP_d)$, we have $T(Y) \to m$ in probability. Let $\eps > 0$. Under the spiked model $Y \sim \GWig(\lambda,\cP,\cP_d,\cX)$ we have with probability $1-o(1)$ that at least $(\delta^2-\eps)n^2$ entries of $xx^\top$ lie in the set $[-n^{-1/2-2\alpha},n^{-1/2-2\alpha}] \setminus \{0\}$. With probability $1-o(1)$, at most $\eps n^2$ of the corresponding entries of $Y$ take the value (exactly) $c/\sqrt{n}$ because by continuity of measure,
$$\lim_{d \to 0^+} \mathrm{Pr}_{w \sim \cP}[w \in [c-d,c+d] \setminus \{c\}] = 0.$$
Therefore, taking $\eps$ sufficiently small, we have $T(Y) \le m - \eps$ with probability $1-o(1)$ and thus $T$ consistently distinguishes the spiked and unspiked models.
\end{proof}

\section{Proof of pre-transformed PCA}
\label{app:nong-upper}

In this section we prove our upper bound for the non-Gaussian Wigner model via pre-transformed PCA. We make the following assumptions on the spike prior $\cX$ and the entrywise noise distribution $\cP$.

% assumptions
\begin{repassumption}{as:nong-upper}
Of the prior $\cX$ we require (as usual) $\|x\| \to 1$ in probability, and we also assume that with probability $1-o(1)$, all entries of $x$ are small: $|x_i| \le n^{-1/2 + \alpha}$ for some fixed $\alpha < 1/8$. Of the noise $\cP$, we assume the following:
\begin{enumerate}[label=(\roman*),resume]
\item $\cP$ has a non-vanishing $C^3$ density function $p(w) > 0$,
\item Letting $f(w) = -p'(w)/p(w)$, we have that $f$ and its first two derivatives are polynomially-bounded: there exists $C > 0$ and an even integer $m \ge 2$ such that $|f^{(\ell)}(w)| \le C + w^m$ for all $0 \le \ell \le 2$.
\item With $m$ as in (ii), $\cP$ has finite moments up to $5m$:
$\EE|\cP|^k < \infty$ for all $1 \le k \le 5m$. 
\end{enumerate}
\end{repassumption}

\noindent An important consequence of assumptions (ii) and (iii) is the following.
\begin{lemma}
\label{lem:mom}
$\EE|f^{(\ell)}(\cP)|^q < \infty$ for all $0 \le \ell \le 2$ and $1 \le q \le 5$. Likewise $\EE|f^{(\ell)}(\cP_d)|^q < \infty$ for all $0 \le \ell \le 2$ and $1 \le q \le 3$.
\end{lemma}
\begin{proof}
We demonstrate $\cP$; then $\cP_d$ follows identically. Using $|a+b|^q \le |2a|^q + |2b|^q = 2^q(|a|^q+|b|^q)$ we have
\begin{equation*} \EE|f^{(\ell)}(\cP)|^q \le \EE|C + \cP^m|^q \le 2^q(C^q + \EE|\cP|^{mq}) < \infty. \qedhere\end{equation*}
\end{proof}

% main theorem
The main theorem of this section is the following.
\begin{reptheorem}{thm:nong-upper}
Let $\lambda \ge 0$ and let $\cX,\cP$ satisfy Assumption~\ref{as:nong-upper}. Let $\hat Y = \sqrt{n}\, Y$ where $Y$ is drawn from $\Wig(\lambda,\cP,\cP_d,\cX)$. Let $f(\hat Y)$ denote entrywise application of the function $f(w) = -p'(w)/p(w)$ to $\hat Y$, except we define the diagonal entries of $f(\hat Y)$ to be zero.
\begin{itemize}[leftmargin=*]
\item If $\lambda \le 1/\sqrt{F_\cP}$ then $\frac{1}{\sqrt n}\lambda_{\max}(f(\hat Y)) \to 2\sqrt{F_\cP}$ as $n \to \infty$.
\item If $\lambda > 1/\sqrt{F_\cP}$ then $\frac{1}{\sqrt n}\lambda_{\max}(f(\hat Y)) \to \lambda F_\cP + \frac{1}{\lambda} > 2 \sqrt{F_\cP}$ as $n \to \infty$ and furthermore the top (unit-norm) eigenvector $v$ of $f(\hat Y)$ correlates with the spike:
$\langle v,x \rangle^2 \ge (\lambda - 1/\sqrt{F_\cP})^2/\lambda^2 - o(1)$ with probability $1-o(1)$.
\end{itemize}
Convergence is in probability. Here $\lambda_{\max}(\cdot)$ denotes the maximum eigenvalue.
\end{reptheorem}
\noindent Note that Lemma~\ref{lem:mom} implies that the expectation defining $F_\cP$ is finite.

% proof
\begin{proof}%[Proof of Theorem~\ref{thm:nong-upper}]
First we justify a local linear approximation of $f(\hat Y_{ij})$. For $i \ne j$, define the error term $\mathcal{E}_{ij}$ by $$f(\hat Y_{ij}) = f(W_{ij}) + \lambda \sqrt{n} x_i x_j f'(W_{ij}) + \mathcal{E}_{ij}.$$
(Define $\mathcal{E}_{ii} = 0$.) We will show that the operator norm of $\mathcal{E}$ is small: $\|\mathcal{E}\| = o(\sqrt n)$ with probability $1-o(1)$.
Apply the mean-value form of the Taylor approximation remainder: $\mathcal{E}_{ij} = \frac{1}{2} f''(W_{ij} + e_{ij}) \lambda^2 n x_i^2 x_j^2$ for some $|e_{ij}| \le |\lambda \sqrt n x_i x_j|$. Bound the operator norm by the Frobenius norm:
$$\|\mathcal{E}\|^2 \le \|\mathcal{E}\|_F^2 = \frac{\lambda^4 n^2}{4} \sum_{i \ne j} x_i^4x_j^4 f''(W_{ij} + e_{ij})^2 \le \frac{\lambda^4}{4} n^{8\alpha - 2} \sum_{i \ne j} f''(W_{ij} + e_{ij})^2.$$
Using the polynomial bound on $f''$ and the fact $|a+b|^{k} \le 2^{k}(|a|^{k} + |b|^{k})$, we have
\begin{align*}
f''(W_{ij} + e_{ij})^2 &\le (C + (W_{ij} + e_{ij})^{m})^2 
\le 4C^2 + 4(W_{ij} + e_{ij})^{2m} \\
&\le 4C^2 + 4 \cdot 2^{2m} (W_{ij}^{2m} + e_{ij}^{2m}) \\
&\le 4C^2 + 2^{2m+2} (W_{ij}^{2m} + \lambda^{2m}n^{(4\alpha-1)m}) \\
&= 4C^2 + 2^{2m+2} W_{ij}^{2m} + o(1).
\end{align*}
Using finite moments of $W_{ij} \sim \cP$, it follows that
$\EE\left[\sum_{i \ne j} f''(W_{ij} + e_{ij})^2\right] = \mathcal{O}(n^2)$,
and so $\EE \|\mathcal{E}\|^2 = \mathcal{O}(n^{8\alpha}).$
Since $\alpha < 1/8$, Markov's inequality now gives the desired result: with probability $1-o(1)$, $\|\mathcal{E}\|^2 = o(n)$ and so $\|\mathcal{E}\| = o(\sqrt n)$.

% Delta_ij
Our goal will be to show that $f(\hat Y)$ is, up to small error terms, another spiked Wigner matrix. Toward this goal we define another error term: for $i \ne j$, let
$\Delta_{ij} = \lambda \sqrt{n} x_i x_j \left(f'(W_{ij}) - \mathbb{E}[f'(W_{ij})]\right)$,
so that
\begin{equation}
\label{eq:off-diag}
f(\hat Y_{ij}) = f(W_{ij}) + \lambda \sqrt{n} x_i x_j \mathbb{E}[f'(W_{ij})] + \mathcal{E}_{ij} + \Delta_{ij}.
\end{equation}
(Define $\Delta_{ii} = 0$.) We will show that the operator norm of $\Delta$ is small: $\|\Delta\| = o(\sqrt n)$ with probability $1-o(1)$. Let $A_{ij} = f'(W_{ij}) - \mathbb{E}[f'(W_{ij})]$ so that $\Delta_{ij} = \lambda \sqrt{n} x_i x_j A_{ij}$. (Define $A_{ii} = 0$.) We have $\|\Delta\| \le \lambda n^{-1/2+2\alpha} \|A\|$ because for any unit vector $y$,
\begin{align*}
y^\top \Delta y &= \sum_{i,j} \lambda \sqrt{n} x_i x_j A_{ij} y_i y_j \le \sum_{i,j} \lambda \sqrt{n} z_i A_{ij} z_j \qquad\text{where } z_i = x_i y_i \\
&\le \lambda \sqrt{n}\, \|A\| \cdot \|z\|^2 \le \lambda n^{-1/2+2\alpha} \|A\| \cdot \|y\| = \lambda n^{-1/2+2\alpha} \|A\|.
\end{align*}
Note that $A$ is a Wigner matrix (i.e.\ a symmetric matrix with off-diagonal entries i.i.d.) and so $\|A\| = \mathcal{O}(\sqrt n)$ with probability $1-o(1)$. This follows from \citet{wig-spk} Theorem~1.1, provided we can check that each entry of $A$ has finite fifth moment. But this follows from Lemma~\ref{lem:mom}:
$$\EE|A_{ij}|^5 \le 2^5 \left(\EE|f'(W_{ij})|^5 + |\EE[f'(W_{ij})]|^5\right) < \infty.$$
Now we have $\|\Delta\| = \mathcal{O}(n^{2\alpha}) = o(\sqrt n)$ with probability $1-o(1)$ as desired.

% diag
From (\ref{eq:off-diag}) we now have that, up to small error terms, $f(\hat Y)$ is another spiked Wigner matrix:
$$f(\hat Y) = f(W) + \lambda \sqrt{n}\, \EE[f'(\cP)]\, xx^\top + \mathcal{E} + \Delta - \delta$$
where (to take care of the diagonal) we define $f(W)_{ii} = 0$, $\delta_{ij} = 0$, and $\delta_{ii} = \lambda \sqrt{n}\, \EE[f'(\cP)] x_i^2$. Note that the final error term $\delta$ is also small: $\|\delta\| \le \|\delta\|_F = \mathcal{O}(n^{2\alpha}) = o(\sqrt n)$. We now have
$$\frac{1}{\sqrt n} \lambda_{\max}(f(\hat Y)) = \lambda_{\max}\left(\frac{1}{\sqrt n} f(W) + \lambda\, \mathbb{E}[f'(\cP)] \,xx^\top \right) + o(1)$$
and so the theorem follows from known results on the spectrum of spiked Wigner matrices, namely Theorem~1.1 from \citet{wig-spk}. We need to check the following details. First note that the Wigner matrix $f(W)$ has off-diagonal \iid entries that are centered:
$$\mathbb{E}[f(W_{ij})] 
%= \int_{-\infty}^\infty f(w)p(w)dw 
= \int_{-\infty}^\infty \frac{-p'(w)}{p(w)}p(w)dw %= -\int_{-\infty}^\infty p'(w)dw 
= p(-\infty) - p(\infty) = 0.$$
Each off-diagonal entry of $f(W)$ has variance
$\mathbb{E}[f(W_{ij})^2] = F_\cP.$
The rank-1 deformation $\lambda\, \mathbb{E}[f'(\cP)] \,xx^\top$ has top eigenvalue $\lambda\, \mathbb{E}[f'(\cP)] \cdot \|x\|^2$. Recall that $\|x\|^2 \to 1$ in probability. Also,
$$f'(w) = \frac{d}{dw} \frac{-p'(w)}{p(w)} = -\frac{p''(w)p(w) - p'(w)^2}{p(w)^2}$$
and so
$$\mathbb{E}[f'(\cP)] 
%= \int_{-\infty}^\infty f'(w)p(w)dw 
= \int_{-\infty}^\infty \left[-p''(w) + \frac{p'(w)^2}{p(w)}\right]dw = \int_{-\infty}^\infty \frac{p'(w)^2}{p(w)}dw = F_\cP.$$
Therefore the top eigenvalue of the rank-1 deformation converges in probability to $\lambda F_\cP$. By Lemma~\ref{lem:mom}, the entries of $f(W)$ have finite fifth moment.

The desired convergence of the top eigenvalue now follows. It remains to show that when $\lambda > 1/\sqrt{F_\cP}$, the top eigenvalue of $f(\hat Y)$ correlates with the planted vector $x$. Let $v$ be the top eigenvector of $f(\hat Y)$ with $\|v\| = 1$. From above we have
$$v^\top \left(\frac{1}{\sqrt n} f(\hat Y)\right)v = v^\top \left(\frac{1}{\sqrt n} f(W)\right)v + \lambda F_\cP \langle v,x \rangle^2 + o(1).$$
We know $\frac{1}{\sqrt n} f(\hat Y)$ has top eigenvalue $\lambda F_\cP + 1/\lambda + o(1)$ and $\frac{1}{\sqrt n} f(W)$ has top eigenvalue $2\sqrt{F_\cP} + o(1)$, which yields
\begin{align*}
\langle v,x \rangle^2 &\ge \frac{1}{\lambda F_\cP}(\lambda F_\cP + 1/\lambda - 2\sqrt{F_\cP})-o(1)
= \frac{(\lambda - 1/\sqrt{F_\cP})^2}{\lambda^2} - o(1).\qedhere
\end{align*}
\end{proof}

%%% WISHART

\section{Proof of Theorem~5.3: MLE for Wishart with finite prior}\label{app:wishart-mle}

Note the following well-known Chernoff bound for the $\chi^2_k$ distribution:
\begin{lemma}\label{lemma:chi2rate}
For all $0 < z < 1$,
$$ \frac1k \log \Pr\left[ \chi_k^2 < zk \right] \le \frac12(1-z+\log z). $$
Similarly, for all $z > 1$,
$$ \frac1k \log \Pr\left[ \chi_k^2 > zk \right] \le \frac12(1-z+\log z). $$

\end{lemma}

We now prove the following theorem:
\begin{reptheorem}{thm:wishart-mle}
Let $\beta \in (-1,\infty)$. Let $\cX_n$ be a spike prior supported on at most $c^n$ points, for some fixed $c > 0$. If
$$ 2\gamma \log c < \beta - \log(1+\beta)$$
then there is a (computationally inefficient) procedure that distinguishes between the spiked Wishart model $\Wish(\gamma,\beta,\cX)$ and the unspiked model $\Wish(\gamma)$, with $o(1)$ probability of error.
\end{reptheorem}

\begin{proof}
First consider the case $\beta < 0$. Given a matrix $Y$, consider the test statistic
$$T = \min_{v \in \supp(\cX_n)} \frac{v^\top Y v}{\|v\|^2}$$
where $\supp(\cX_n)$ denotes the support of $\cX_n$. Under $Y \sim \Wish(\gamma,\beta,\cX)$ with true spike $x$, we have that $x^\top Y x/\|x\|^2 \sim \frac{1}{N}(1+\beta\|x\|^2) \chi_{N}^2$, which converges in probability to $1+\beta$ (since $\|x\| \to 1$ in probability). Hence, for any $\eps > 0$, we have that $T < 1+\beta+\eps$ with probability $1-o(1)$ under the spiked model $\Wish(\gamma,\beta,\cX)$.

Let $\hat\gamma = n/N$ so that $\hat\gamma \to \gamma$. Under the unspiked model, we have
\begin{align*}
\Pr[T \leq 1+\beta+\eps] &\leq \sum_{v \in \supp(\cX)} \Pr[v^\top Y v/\|v\|^2 \le 1+\beta+\eps] \\
&\leq c^n \Pr\left[\chi_{N}^2 \leq (1+\beta+\eps)N \right] \\
&= \exp\left[N\left( \hat\gamma \log c + \frac1N \log \Pr\left[\chi_{N}^2 \leq (1+\beta+\eps)N\right] \right)\right] \\
&\le \exp\left[N\left( \hat\gamma \log c + \frac{1}{2}(1-(1+\beta+\eps) + \log(1+\beta+\eps)) \right)\right] \\
\end{align*}
by Lemma~\ref{lemma:chi2rate}. This is $o(1)$ so long as
$$2 \gamma \log c - \beta - \eps + \log(1+\beta+\eps) < 0.$$
We can choose such $\eps > 0$ precisely under the hypothesis of this theorem.

Hence, by thresholding the statistic $T$ at $1+\beta+\eps$, we obtain a hypothesis test that distinguishes $Y \sim \Wish(\gamma,\beta,\cX)$ from $Y \sim \Wish(\gamma)$, with probability $o(1)$ of error of either type.

The proof for the case $\beta > 0$ is similar, using instead the test statistic $T = \max_{v \in \supp(\cX_n)} v^\top Y v/\|v\|^2$ along with the upper tail bound for $\chi_k^2$.
\end{proof}

\section{Basic properties of Wishart lower bound}
\label{app:prop-F}

In this section we give basic properties of the condition on $\gamma,\beta$ required by Theorem~\ref{thm:wish-nc}. Recall that this condition is $\gamma > \gamma^*$ where
\begin{equation}
\label{eq:nc-result-app}
\gamma^* f_\cX(t) \ge F(\beta,t) \quad \forall t \in (0,1)
\end{equation}
where
$$F(\beta,t) \defeq (1+\beta) \frac{t(w-t)}{1-t^2} + \frac{1}{2}\log\left(\frac{1-w^2}{1-t^2}\right)$$
and
$$w = \sqrt{A^2+1}-A \;\text{ with }\; A = \frac{1-t^2}{2t(\beta+1)}.$$

We have the following properties of $F(\beta,t)$, which can be shown using basic calculus.

\begin{itemize}

\item The $t \to 0^+$ and $t\to 1^-$ limits of $F(\beta,t)$ exist and so $F(\beta,t)$ is defined and continuous in both variables on the domain $\beta \in (-1,\infty)$, $t \in [0,1]$. The boundary values are $F(\beta,0) = 0$ and $F(\beta,1) = \frac{1}{2}(\beta - \log(1+\beta))$.

\item For any $\beta \in (-1,\infty)\setminus\{0\}$, $F(\beta,t)$ is a strictly increasing function of $t$. In particular, $F(\beta,t) \ge 0$ with equality only at $t = 0$.

\item For any $\beta \in (-1,\infty)$, $\lim_{t \to 0^+} \frac{\partial}{\partial t} F(\beta,t) = 0$ and $\lim_{t \to 0^+} \frac{\partial^2}{\partial t^2} F(\beta,t) = \beta^2$.
\

\end{itemize}

We now give some lemmas that allow for a tradeoff between certain variables while keeping (\ref{eq:nc-result-app}) true. The first allows the rate function to be weakened slightly at the expense of increasing $\gamma^*$ slightly.

\begin{lemma}
\label{lem:change-f}
Let $\gamma^* > 0$, $\beta \in (-1,\infty)\setminus\{0\}$, and $\eps > 0$. Let $f(t)$ be a function on $(0,1)$. If $\gamma^* f(t) \ge F(\beta,t)\;\forall t \in (0,1)$ then there exists $\delta > 0$ such that $(\gamma^* + \eps) f(t(1-\delta)^2) \ge F(\beta,t)\;\forall t \in (0,1)$.
\end{lemma}
\begin{proof}
We have
$$(\gamma^* + \eps)f(t(1-\delta)^2)) \ge \frac{\gamma^*+\eps}{\gamma^*} F(\beta,t(1-\delta)^2)$$
so it is sufficient to show
\begin{equation}
\label{eq:F-frac}
\frac{F(\beta,t)}{F(\beta,t(1-\delta)^2)} \le \frac{\gamma^*+\eps}{\gamma^*} \quad \forall t \in (0,1].
\end{equation}
For each $t \in (0,1]$ there exists a maximal $\delta = \delta(t) > 0$ such that (\ref{eq:F-frac}) holds, and $\delta(t)$ is a continuous function of $t$. We want to show that $\delta(t)$ is bounded above 0, so we only need to check the limit $t \to 0$.

Since $\lim_{t \to 0} F(\beta,t) = \lim_{t \to 0} \frac{\partial}{\partial t} F(\beta,t) = 0$ and $\lim_{t \to 0} \frac{\partial^2}{\partial t^2}F(\beta,t) = \beta^2 > 0$ we have, using L'H\^opital's rule,
$$\lim_{t \to 0} \frac{F(\beta,t)}{F(\beta,t(1-\delta)^2)} = \frac{1}{(1-\delta)^4}$$
which can be made smaller than $(\gamma^*+\eps)/\gamma^*$ by taking $\delta > 0$ small enough.
\end{proof}

The next lemma allows $\beta$ to be increased slightly at the expense of increasing $\gamma^*$ slightly.

% lemma: beta

\begin{lemma}
\label{lem:change-beta}
Let $\gamma^* > 0$, $\beta \in (-1,\infty)\setminus\{0\}$, and $\eps > 0$. Let $f(t)$ be a function on $(0,1)$. If $\gamma^* f(t) \ge F(\beta,t)\;\forall t \in (0,1)$ then there exists $\delta > 0$ such that $(\gamma^* + \eps) f(t) \ge F(\beta(1+\delta)^2,t)\;\forall t \in (0,1)$.
\end{lemma}
\begin{proof}
We have
$$(\gamma^* + \eps)f(t) \ge \frac{\gamma^*+\eps}{\gamma^*} F(\beta,t)$$
so it is sufficient to show
$$\frac{F(\beta(1+\delta)^2,t)}{F(\beta,t)} \le \frac{\gamma^*+\eps}{\gamma^*} \quad \forall t \in (0,1],$$
or equivalently,
$$\log F(\beta(1+\delta)^2,t) - \log F(\beta,t) \le \log\left(\frac{\gamma^*+\eps}{\gamma^*}\right).$$
It is sufficient to have, for any fixed compact interval $\mathcal{I} \subseteq (-1,\infty)$ not containing zero, that $|\frac{\partial}{\partial\beta} \log F(\beta,t)|$ is bounded by a constant, uniformly over all $t \in (0,1]$ and $\beta \in \mathcal{I}$. Since $\frac{\partial}{\partial \beta} \log F(\beta,t)$ is defined and continuous in both variables (on the domain $t \in (0,1]$ and $\beta > -1$), we only need to check the limit $t \to 0$. We have $\lim_{t \to 0} \frac{\partial}{\partial \beta} \log F(\beta,t) = 2/\beta$.

\end{proof}

\section{Proof of Lemma~5.13}
\label{app:wish-small-dev-chernoff}

Here we show how to use the local Chernoff bound to bound the small deviations of the Wishart second moment. Letting $\hat\gamma = n/N$ so that $\hat\gamma \to \gamma$ we have
\begin{align*}
S(\eps) &= \Ex_{x,x' \sim \cX} \exp\left(\frac{-n}{2\hat\gamma} \log(1-\beta^2\langle x,x' \rangle^2) \right) \;\one_{\langle x,x' \rangle^2 \leq \eps} \\
&\leq \Ex_{x,x' \sim \cX} \exp\left( \frac{-n}{2\hat\gamma \eps^2} \log(1-\eps^2\beta^2) \langle x,x' \rangle^2 \right) \;\one_{\langle x,x' \rangle^2 \leq \eps} \\
\intertext{where we used the convexity of $t \mapsto -\log(1-\beta^2 t)$}
&= \int_0^\infty \problr{\exp\left( \frac{-n}{2\hat\gamma \eps^2} \log(1-\eps^2\beta^2) \langle x,x' \rangle^2 \right) \one_{\langle x,x' \rangle^2 \leq \eps} \ge u} \, \dee u \\
&= \int_0^\infty \problr{\langle x,x' \rangle^2 \le \eps \quad\text{and}\quad \exp\left( \frac{-n}{2\hat\gamma \eps^2} \log(1-\eps^2\beta^2) \langle x,x' \rangle^2 \right) \ge u} \, \dee u \\
&= \int_0^\infty \problr{\langle x,x' \rangle^2 \le \eps \text{ and } \langle x,x' \rangle^2 \ge t} \frac{-n}{2\hat\gamma\eps^2} \log(1-\eps^2\beta^2) \exp\left(-\frac{n}{2\hat\gamma\eps^2} \log(1-\eps^2\beta^2) t\right)\,\dee t \\
&= \int_0^\eps \problr{\langle x,x' \rangle^2 \ge t} \frac{-n}{2\hat\gamma\eps^2} \log(1-\eps^2\beta^2) \exp\left(-\frac{n}{2\hat\gamma\eps^2} \log(1-\eps^2\beta^2) t\right)\,\dee t \\
\intertext{where $t$ is defined by $\exp(-\frac{n}{2\hat\gamma\eps^2}\log(1-\eps^2\beta^2)t) = u$. If $\eps$ is sufficiently small we can apply the local Chernoff bound:}
&\le \int_0^\eps \frac{-Cn}{2\hat\gamma\eps^2} \log(1-\eps^2\beta^2) \exp\left(-n f_\cX(\sqrt{t}) -\frac{n}{2\hat\gamma\eps^2} \log(1-\eps^2\beta^2) t\right)\,\dee t.
\end{align*}
Using the identity $\int_0^\infty n \exp(-n \alpha t) \,\dee t = 1/\alpha$ (for $\alpha > 0$), the above is bounded provided we have $\eps > 0$ and $\alpha > 0$ such that
$$f_\cX(\sqrt{t}) \ge -\frac{1}{2\gamma\eps^2}\log(1-\eps^2\beta^2)t + \alpha t \quad\forall t \in [0,\eps).$$
Using the bound $\log t \ge 1 - 1/t$ we have $-\frac{1}{\eps^2} \log(1-\eps^2\beta^2) \le \frac{\beta^2}{1-\eps^2\beta^2}$ and so it is sufficient to show $f_\cX(\sqrt{t}) \ge \left(\frac{\beta^2}{2\gamma} + \eta \right)t$ for all $t \le \eps$, for some $\eta > 0$. But this can be derived from (\ref{eq:nc-result-app}) as follows. With $\gamma > \gamma^*$ we have $\gamma^* f_\cX(t) \ge F(\beta,t)$ for all $t \in (1,0)$. Rewrite this as $f_\cX(\sqrt{t}) \ge F(\beta,\sqrt{t})/\gamma^*$ and compute $\lim_{t \to 0} \frac{\partial}{\partial t} F(\beta,\sqrt{t})/\gamma^* = \beta^2/2\gamma^* > \beta^2/2\gamma$.

% tilde too small?

\section{Comparison of priors and general case of Wishart lower bound}
\label{app:unit-wlog}

In the main text we have proven Theorem~\ref{thm:wish-nc} in the special case that $\cX$ is supported on unit vectors. Here we extend the proof to the general case where $\|x\| \to 1$ in probability. The same argument also yields a result for comparison of priors (Proposition~\ref{prop:wish-compare}), similar to Proposition~\ref{prop:compare} for the Gaussian Wigner model.

Suppose that $\cX, \beta, \gamma^*$ satisfy the assumptions of Theorem~\ref{thm:wish-nc} and let $\gamma > \gamma^*$. Our goal is to show $\Wish(\gamma,\beta,\cX) \contig \Wish(\gamma)$. Let $\delta > 0$ and let $\tilde \cX$ be the conditional distribution of $x \sim \cX$ given $1-\delta \le \|x\| \le 1+\delta$. Let $M(\gamma,\beta,\tilde\cX)$ denote the conditional Wishart second moment defined in Section~\ref{sec:wish-nc}. We will show that for $\delta$ small enough, $M(\gamma,\beta,\tilde\cX)$ is bounded, implying the desired result (via Lemma~\ref{lem:cond}).

% define \bar\cX, rate function
Let $\bar{\cX}$ be the distribution of $\bar x \defeq \tilde x/\|\tilde x\|$ with $\tilde x \sim \tilde\cX$. The idea of the proof is to show that the assumptions of Theorem~\ref{thm:wish-nc} are satisfied for $\bar{\cX}$ so that we can apply the basic ($\|x\|=1$) version of the theorem (which we have already proven). Note that $f_{\bar\cX}(t) \defeq f_\cX(t(1-\delta)^2)$ is a valid rate function for $\bar\cX$. This follows from
$$\prob{|\langle \bar x,\bar x' \rangle| \ge t} \le \prob{|\langle \tilde x,\tilde x' \rangle| \ge t(1-\delta)^2} \le c \cdot \prob{|\langle x,x' \rangle| \ge t(1-\delta)^2}$$
where $c = 1+o(1)$. We can take the lower bound (in Definition~\ref{def:rate-function}) to be $b_{n,\bar\cX} = -\frac{1}{n} \log c + b_{n,\cX}(t(1-\delta)^2)$. If $f_\cX$ admits a local Chernoff bound (condition (ii) of Theorem~\ref{thm:wish-nc}) then so does $f_{\bar\cX}$.

% small and large devs
As in the proof for the $\|x\|=1$ case, we treat the small and large deviations separately. The parameter $\alpha$ that separates the small ($|\alpha| \in [0,\eps]$) and large ($|\alpha| \in (\eps,1]$) deviations is now defined with normalization: $\alpha \defeq \langle x,x' \rangle/(\|x\|\cdot\|x'\|)$.

\subsection{Small deviations}

We have
$$\Ex_{\tilde x,\tilde x' \sim \tilde\cX}(1-\beta^2 \langle \tilde x,\tilde x' \rangle^2)^{-N/2} \one_{\langle \tilde x,\tilde x' \rangle^2/(\|\tilde x\| \cdot \|\tilde x'\|)^2 \le \eps}
\le \Ex_{\bar x,\bar x' \sim \bar\cX}(1-\beta^2 (1+\delta)^4 \langle \bar x,\bar x' \rangle^2)^{-N/2} \one_{\langle \bar x,\bar x' \rangle^2 \le \eps},$$
i.e.\ the small deviations of $M(\gamma,\beta,\tilde\cX)$ are bounded by the small deviations of $M(\gamma,\beta(1+\delta)^2,\bar\cX)$. Therefore it is sufficient to verify the conditions of Theorem~\ref{thm:wish-nc} for $\gamma,\beta(1+\delta)^2,\bar\cX$.

% wigner threshold -- condition (i)
First we show that if condition (i) ($\beta^2/\gamma^* \le (\lambda^*_\cX)^2$) in Theorem~\ref{thm:wish-nc} was satisfied for $\gamma,\beta,\cX$ then it is still satisfied for $\gamma,\beta(1+\delta)^2,\bar\cX$ (provided we allow an arbitrarily-small increase in $\gamma^*$). Since conditioning on a $(1-o(1))$-probability event can only increase the Wigner second moment by a $(1+o(1))$ factor, we have $\lambda^*_{\tilde \cX} \ge \lambda^*_{\cX}$. We also have
$$\Ex_{\bar x,\bar x' \sim \bar\cX} \exp\left(\frac{n\lambda^2}{2}\langle \bar x,\bar x' \rangle^2\right) \le \Ex_{\tilde x,\tilde x' \sim \tilde\cX} \exp\left(\frac{n\lambda^2}{2(1-\delta)^2} \langle \tilde x,\tilde x' \rangle^2\right)$$
and so $\lambda^*_{\bar\cX} \ge (1-\delta)\lambda^*_{\tilde\cX} \ge (1-\delta) \lambda^*_\cX$. Therefore by choosing $\delta$ small enough we can find $\bar\gamma^*$ with $\gamma^* < \bar\gamma^* < \gamma$ such that $\beta^2(1+\delta)^4/\bar\gamma^* \le (\lambda^*_{\bar\cX})^2$ as desired.

% condition (ii), i.e. condition on F
Now we check that (\ref{eq:nc-result-app}) is satisfied for $\gamma,\beta(1+\delta)^2,\bar\cX$. We are guaranteed $\gamma > \gamma^*$ with
\begin{equation}
\label{eq:F-1}
\gamma^* f_\cX(t) \ge F(\beta,t) \quad\forall t \in (0,1).
\end{equation}
Our goal is to show (for sufficiently small $\delta$) $\gamma > \bar\gamma^*$ with
\begin{equation}
\label{eq:F-2}
\bar\gamma^* f_{\bar\cX}(t) \ge F(\beta(1+\delta)^2,t) \quad\forall t \in (0,1).
\end{equation}
The proof of (\ref{eq:F-2}) follows from (\ref{eq:F-1}) by Lemmas~\ref{lem:change-f} and~\ref{lem:change-beta}. The first allows us to replace $f_\cX$ by $f_{\bar\cX}$ and the second allows us to increase $\beta$ to $\beta(1+\delta)^2$. Each of these changes comes at the expensive of increasing $\gamma^*$ (to $\bar\gamma^*$) by an arbitrarily-small amount (which can be done such that $\gamma > \bar\gamma^*$).

\subsection{Large deviations}

We now consider the contribution to  $M(\gamma,\beta,\tilde\cX)$ from $|\alpha| \in [\eps,1-\eps]$. The contribution from $|\alpha| \in (1-\eps,1]$ can be handled similarly. We have
$$\Ex_{\tilde x,\tilde x' \sim \tilde \cX}[ \one_{|\alpha| \in [\eps,1-\eps]}\, \tilde m(\tilde x,\tilde x')]$$
where
\begin{align*}
\tilde m(\tilde x,\tilde x')
\defeq& \Ex_{Y \sim P_n} (1+\beta \|\tilde x\|^2)^{-N/2} (1+\beta \|\tilde x'\|^2)^{-N/2} \\
& \exp\left(\frac{N}{2}\left(\frac{\beta}{1+\beta\|\tilde x\|^2} \tilde x^\top Y \tilde x + \frac{\beta}{1+\beta\|\tilde x'\|^2} \tilde x'^\top Y \tilde x'\right)\right) \one_{\Omega(\tilde x,Y)} \one_{\Omega(\tilde x',Y)} \\
\le& \Ex_{Y \sim P_n} (1+\beta(1-\delta)^2)^{-N} \\
& \exp\left(\frac{N}{2}\left(\frac{\beta \|\tilde x\|^2}{1+\beta\|\tilde x\|^2} \bar x^\top Y \bar x + \frac{\beta\|\tilde x'\|^2}{1+\beta\|\tilde x'\|^2} \bar x'^\top Y \bar x'\right)\right) \one_{\Omega(\tilde x,Y)} \one_{\Omega(\tilde x',Y)}
\end{align*}
where $\bar x = \tilde x/\|\tilde x\|$ and $\bar x' = \tilde x'/\|\tilde x'\|$. Note that $\Omega(\tilde x,Y)$ can be written as $\bar x^\top Y \bar x \in [(1+\beta\|\tilde x\|^2)(1-\delta),(1+\beta\|\tilde x\|^2)(1+\delta)]$. We can upper bound the resulting expression by replacing each instance of $\|\tilde x\|^2$ by either $1+\delta$ or $1-\delta$. Since only $\bar x,\bar x'$ (and not $\tilde x,\tilde x'$) now appear, we have reduced to the original case of the proof (since $\|\bar x\|=\|\bar x'\|=1$) but with the $\beta$'s replaced by slightly different constants; carrying through the proof as before yields the sufficient condition $\gamma > \bar\gamma^*$ with $\bar\gamma^* f_{\bar \cX}(t) \ge F_\delta(\beta,t) \; \forall t \in [\eps,1-\eps]$ where for each $t$, $F_\delta(\beta,t) \to F(\beta,t)$ as $\delta \to 0^+$. Since $F,F_\delta$ are continuous and $[\eps,1-\eps]$ is compact, the convergence $F_\delta(\beta,t) \to F(\beta,t)$ is uniform over $t \in [\eps,1-\eps]$. Let $\gamma^* < \hat\gamma^* < \bar \gamma^* < \gamma$. $F(\beta,t)$ is positive and increasing in $t$ for $t \in [\eps,1-\eps]$ (see Appendix~\ref{app:prop-F}), so provided $\delta$ is small enough, it is sufficient to show $\hat\gamma^* f_{\bar\cX}(t) \ge F(\beta,t) \;\forall t \in [\eps,1-\eps]$. This follows from the assumption $\gamma^* f_{\cX}(t) \ge F(\beta,t)$ along with Lemma~\ref{lem:change-f}. The proof of Theorem~\ref{thm:wish-nc} in full generality is now complete.

\subsection{Comparison of similar priors}

The same argument used above implies the following which may be of independent interest.

\begin{proposition}
\label{prop:wish-compare}
Let $\cX$ and $\mathcal{Y}$ be spike priors. Suppose that $x \sim \cX_n$ and $y \sim \mathcal{Y}_n$ can be coupled such that $y = \alpha x$ where $\alpha = \alpha_n$ is a random variable with $\alpha_n \to 1$ in probability as $n\to\infty$. Suppose that the conditions of Theorem~\ref{thm:wish-nc} are satisfied for $\cX,\beta,\gamma^*$. Then for any $\gamma > \gamma^*$, $\Wish(\beta,\gamma,\mathcal{Y}) \contig \Wish(0)$.
\end{proposition}

\begin{proof}
The proof is similar to the arguments above so we only give a sketch. We define modified priors $\tilde\cX, \tilde{\mathcal{Y}}, \bar\cX, \bar{\mathcal{Y}}$ as above and note that $\bar\cX$ and $\bar{\mathcal{Y}}$ are the same. Since the conditions of Theorem~\ref{thm:wish-nc} are satisfied for $\cX$, they are also satisfied for $\bar\cX$ at the expense of an arbitrarily-small increase in $\gamma^*$. We can then control the Wishart conditional second moment $M(\gamma,\beta,\tilde{\mathcal{Y}})$ by comparison to $\bar\cX$ (i.e.\ $\bar{\mathcal{Y}}$).
\end{proof}

\section{Monotonicity of Wishart lower bound}
\label{app:monotonicity}

In this section we prove various properties of the condition (\ref{eq:nc-result-app}), implying certain monotonicity properties of the Wishart lower bound (Theorem~\ref{thm:wish-nc}). Informally speaking, we will show the following.
\begin{itemize}
\item If the PCA threshold is optimal for some $\bar\beta \in (-1,\infty)\setminus \{0\}$, it is also optimal for all $\beta > \bar\beta$ (Proposition~\ref{prop:eq-monotone}).
\item If the PCA threshold is optimal for the Wigner model, it is also optimal for the positively-spiked ($\beta > 0$) Wishart model (Corollary~\ref{cor:wig-pos-wish}). Conversely, if PCA is optimal for Wishart for all $\beta > 0$ then it is optimal for Wigner (Proposition~\ref{prop:pos-wish-wig}).
\item For any reasonable \iid prior, if $\beta$ is sufficiently large then the PCA threshold is optimal (Proposition~\ref{prop:wish-large-beta}).
\end{itemize}
\noindent The statements above are informal; the true results we prove are of the form e.g.\ ``if \emph{our methods} show a Wigner lower bound then they also show a Wishart lower bound.''

\begin{repproposition}{prop:eq-monotone}
Let $\cX$ be a spike prior. Fix $\lambda > 0$ and $\bar\beta \in (-1,\infty)\setminus\{0\}$. If (\ref{eq:nc-result-app}) holds for $\bar\beta$ and $\gamma^* = \bar\beta^2/\lambda^2$ then it also holds for any $\beta > \bar\beta$ and $\gamma^* = \beta^2/\lambda^2$.
\end{repproposition}

\begin{proof}
The condition (\ref{eq:nc-result-app}) takes the form $\gamma^* f_\cX(t) \ge F(\beta,t)\;\; \forall t \in (0,1)$. With $\lambda$ fixed and $\gamma^* = \beta^2/\lambda^2$, this is equivalent to $f_\cX(t) \ge \lambda^2 F(\beta,t)/\beta^2\;\; \forall t \in (0,1)$. It is therefore sufficient to show the following lemma.
\phantom\qedhere % remove qed box
\end{proof}

\begin{lemma}
\label{lem:Fb2-dec}
For any fixed $t \in (0,1)$, $F(\beta,t)/\beta^2$ is a decreasing function of $\beta$ on the domain $\beta \in (-1,\infty)$. (When $\beta = 0$ we define $F(\beta,t)/\beta^2$ by its limit value.)
\end{lemma}

\begin{proof}
We wish to show that $\dd{\beta} \frac{F(\beta,t)}{\beta^2} < 0$. One computes that $\lim_{t \to 0} \dd{\beta} \frac{F(\beta,t)}{\beta^2} = 0$, so it suffices to show that $\frac{\partial^2}{\partial \beta\, \partial t} \frac{F(\beta,t)}{\beta^2} \leq 0$ for $0 < t < 1$. We compute:
$$ \frac{\partial^2}{\partial \beta\, \partial t} \frac{F(\beta,t)}{\beta^2} = \frac{U_1 - U_2}{\beta^3\, S\, t (1-t^2)^2},  \quad \text{where}$$
$$ S = \sqrt{1 + t^2 (2 + 8\beta + 4\beta^2) + t^4}, \quad U_1 = S (1 + 2 (1+\beta) t^2 + t^4) \geq 0,$$
$$U_2 = (1+t^2) ( 1 + t^2 (2 + 6\beta + 2\beta^2) + t^4) \geq 0. $$
The denominator evidently has sign matching $\beta$, so it suffices to see that the numerator has sign matching $-\beta$. As $U_1 \geq 0$ and $U_2 \geq 0$, the sign of $U_1-U_2$ will match that of $U_1^2-U_2^2$, and we compute:
$$ U_1^2 - U_2^2 = -4 \beta^3 (2+\beta) t^4 (1-t^2)^2 $$
which has sign matching that of $-\beta$, as desired.
\end{proof}

% Wigner to positive Wishart

\begin{repcorollary}{cor:wig-pos-wish}
Suppose $\langle x,x' \rangle$ is $(\sigma^2/n)$-subgaussian, where $x$ and $x'$ are drawn independently from $\cX_n$. Then for any $\beta > 0$ and any $\gamma > \beta^2 \sigma^2$ we have $\Wish(\gamma,\beta,\cX) \contig \Wish(\gamma)$.
\end{repcorollary}

\noindent This connects the Wigner model to the Wishart model because the subgaussian condition above implies a \emph{Wigner} lower bound for all $\lambda < 1/\sigma$ (Proposition~\ref{prop:subg}).

\begin{proof}

The subgaussian tail bound $\prob{|\langle x,x' \rangle| \ge t} \le 2 \exp(-nt^2/2\sigma^2)$ implies that we have the rate function $f_\cX(t) = t^2/2\sigma^2$.

Let $F(\beta,t)$ be defined as in (\ref{eq:nc-result-app}). For any fixed $t \in (0,1)$, we have
\begin{equation}
\label{eq:limit-wig-nc}
\lim_{\beta \to 0} \frac{F(\beta,t)}{\beta^2} = \frac{t^2}{2(1+t^2)}.
\end{equation}
This can be shown by computing the Taylor series of $F(\beta,t)$ at $\beta = 0$. From Lemma~\ref{lem:Fb2-dec} above, we know that $F(\beta,t)/\beta^2$ is a decreasing function of $\beta$. Therefore, for any $t \in (0,1)$ and any $\beta > 0$ we have
$$\frac{F(\beta,t)}{\beta^2} \le \frac{t^2}{2(1+t^2)} \le \frac{t^2}{2}.$$

By combining the above results it follows that (\ref{eq:nc-result-app}) holds with $\gamma^* = \beta^2\sigma^2$. Proposition~\ref{prop:subg} implies that the Wigner threshold is $\lambda^*_\cX \ge 1/\sigma$ and so condition (i) of Theorem~\ref{thm:wish-nc} is satisfied. The result now follows from Theorem~\ref{thm:wish-nc}.

% remark about noise-conditioned Wigner
We remark that (\ref{eq:nc-result-app}) would follow from the weaker condition $f_\cX(t) \ge \frac{t^2}{2\sigma^2(1+t^2)}$ (instead of $f_\cX(t) \ge \frac{t^2}{2\sigma^2}$). This is exactly the condition for the Wigner lower bound of \citet{noise-cond} with $\lambda^* = 1/\sigma$.
\end{proof}

% converse: Wishart to Wigner
\begin{proposition}
\label{prop:pos-wish-wig}
Fix $\lambda^* > 0$. Suppose that for each $\beta > 0$, the assumptions of Theorem~\ref{thm:wish-nc} are satisfied for $\cX$ and $\gamma^* = \beta^2/(\lambda^*)^2$. Then $\GWig(\lambda,\cX) \contig \GWig(0)$ for any $\lambda < \lambda^*$.
\end{proposition}

\begin{proof}
If condition (i) of Theorem~\ref{thm:wish-nc} is satisfied for some $\beta$ then we are done immediately, so assume condition (ii) holds for all $\beta > 0$. For all $\beta > 0$ and all $t \in (0,1)$ we have $\gamma^* f_\cX(t) \ge F(\beta,t)$, i.e.\ $f_\cX(t) \ge (\lambda^*)^2 F(\beta,t)/\beta^2$. Using (\ref{eq:limit-wig-nc}) this implies $f_\cX(t) \ge \frac{(\lambda^*)^2}{2} \frac{t^2}{1+t^2}$ and so the result follows from \citet{noise-cond}. (Although \citet{noise-cond} assume that the spike has exactly unit norm, arguments similar to Appendix~\ref{app:unit-wlog} can be used.)
\end{proof}

% PCA tight for large b

\begin{repproposition}{prop:wish-large-beta}
Suppose $\cX = \IID(\pi/\sqrt{n})$ where $\pi$ is a mean-zero unit-variance distribution for which $\pi \pi'$ (product of two independent copies of $\pi$) has a moment-generating function $M(\theta) \defeq \EE\exp(\theta \pi \pi')$ which is finite on an open interval containing zero. Then there exists $\bar \beta$ such that for any $\beta \ge \bar\beta$ and any $\gamma > \beta^2$ we have $\Wish(\gamma,\beta,\cX) \contig \Wish(\gamma)$.
\end{repproposition}

\noindent In other words, for sufficiently large $\beta$, detection is impossible below the spectral threshold.

\begin{proof}

We will show that for sufficiently large $\beta$, the assumptions of Theorem~\ref{thm:wish-nc} hold with $\gamma^* = \beta^2$ so that the result follows. The usual Chernoff bound yields
$$ \prob{\langle x,x' \rangle\ge t} \le \exp[-n(t\theta - \log M(\theta))] \quad \forall \theta \in \RR$$
and so, letting $\theta = t$,
$$\prob{\langle x,x' \rangle\ge t} \le \exp[-n(t^2 - \log M(t))].$$
Similarly,
$$ \prob{\langle x,x' \rangle\le -t} \le \exp[-n(t^2 - \log M(-t))].$$
Therefore, $f_\cX(t) \defeq \min \{f_1(t),f_2(t)\}$ is a valid rate function for $\cX$ with a local Chernoff bound, where $f_1(t) \defeq t^2 - \log M(t)$ and $f_2(t) \defeq t^2 - \log M(-t)$. It remains to show that for sufficiently large $\beta$, (\ref{eq:nc-result-app}) holds with $\gamma^* = \beta^2$, i.e. $f_\cX(t) \ge F(\beta,t)/\beta^2\;\forall t \in (0,1)$. We will show
\begin{equation}
\label{eq:f1}
f_1(t) \ge F(\beta,t)/\beta^2 \quad\forall t \in (0,1)
\end{equation}
but the proof for $f_2$ is similar.

First we show that (\ref{eq:f1}) holds for $t \in (0,\eps]$ for some $\eps > 0$. Using the well-known identity $\frac{\dee^k}{\dee \theta^k} M(\theta)|_{\theta = 0} = \EE[(\pi\pi')^k]$ we have derivatives $M(0) = 0$, $M'(0) = 0$, $M''(0) = 1$, $|M'''(0)| < \infty$, $|M''''(0)| < \infty$.
We can use these to compute $f_1(0) = 0$, $\lim_{t \to 0^+} f_1'(t) = 0$, $\lim_{t \to 0^+} f_1''(t) = 1$, $\lim_{t \to 0^+} f_1'''(t) = 0$, $|\lim_{t \to 0^+} f_1''''(t)| < \infty$. We can also compute $F(\beta,0)/\beta^2 = 0$, $\lim_{t \to 0^+} \frac{\partial}{\partial t} F(\beta,t)/\beta^2 = 0$, $\lim_{t \to 0^+} \frac{\partial^2}{\partial t^2} F(\beta,t)/\beta^2 = 1$, $\lim_{t \to 0^+} \frac{\partial^3}{\partial t^3} F(\beta,t)/\beta^2 = 0$, $\lim_{t \to 0^+} \frac{\partial^4}{\partial t^4} F(\beta,t)/\beta^2 = -6(b^2+4b+2)$. Note that $f_1(t)$ and $F(\beta,t)/\beta^2$ have matching derivatives (at $t = 0$) up to third order and that the fourth derivative of $F(\beta,t)/\beta^2$ goes to $-\infty$ as $\beta \to \infty$. Therefore we can find $\beta$ and $\eps > 0$ such that $f_1(t) \ge F(\beta,t)/\beta^2$ for all $t \in (0,\eps]$. Since $F(\beta,t)/\beta^2$ is a decreasing function of $\beta$ (Lemma~\ref{lem:Fb2-dec}), this remains true for any larger $\beta$.

Now we show that (\ref{eq:f1}) holds for $t \in (\eps,1)$. For any $t \in (\eps,1)$ we have $f_1(t) \ge f_1(\eps) > 0$ (using the derivatives above and the fact that rate functions are increasing). Also, for any $t \in [0,1]$ we have $\lim_{\beta \to \infty} F(\beta,t)/\beta^2 = 0$ and by compactness this convergence is uniform over $t$. Therefore if $\beta$ is sufficiently large, (\ref{eq:f1}) holds for $t \in (\eps,1)$, completing the proof.

\end{proof}

\section{Wishart results for specific priors}
\label{app:wish-priors}

\subsection{Spherical prior}

Our lower bound for the spherical prior is obtained by combining Theorem~\ref{thm:wish-nc} with the rate function of Proposition~\ref{prop:rate-functions}, along with the fact that $\lambda_\cX^* = 1$ for the spherical prior (Corollary~\ref{cor:sphere-prior}). The result is that the PCA threshold is optimal (i.e.\ we have contiguity for all $\gamma > \beta^2$) for all $\beta \in (-1,\infty)$. To show this, we need to check (\ref{eq:nc-result-app}) with $\gamma^* = \beta^2$. This follows from $\lim_{\beta \to -1^+} F(\beta,t)/\beta^2 = -\frac{1}{2} \log(1-t^2)$ (which is precisely the spherical rate function of Proposition~\ref{prop:rate-functions}) along with the fact that $F(\beta,t)/\beta^2$ is a decreasing function of $\beta$ (Lemma~\ref{lem:Fb2-dec}).

% these 2 facts above imply that F/b^2 is at most as big as the spherical rate function for any b > -1

\subsection{Rademacher prior}

Our lower bound for the Rademacher prior is obtained by combining Theorem~\ref{thm:wish-nc} with the rate function of Proposition~\ref{prop:rate-functions}, along with the fact that $\lambda_\cX^* = 1$ for the Rademacher prior (Corollary~\ref{cor:pmone-wigner}). We also obtain an upper bound from Theorem~\ref{thm:wishart-mle}, taking $c = 2$.

\subsection{Sparse Rademacher prior}
\label{app:wish-sparse-rad}

First consider the variant of the sparse Rademacher prior where the spike has exactly $\rho n$ nonzero entries, which are \iid $\pm 1/\sqrt{\rho n}$ (and we restrict to $n$ for which $\rho n$ is an integer). In this case the rate function $f_\rho$ stated in Proposition~\ref{prop:rate-functions} has been proven to be valid, and furthermore to admit a local Chernoff bound \citep{noise-cond}. This yields a lower bound via Theorem~\ref{thm:wish-nc}. We show that the same lower bound holds for the \iid sparse Rademacher prior:

\begin{proposition}
Let $\cX_\rho$ be the \iid sparse Rademacher prior with sparsity $\rho$ (as defined in Section~\ref{sec:sparse-rad}). Let $f_\rho$ be the rate function defined in Proposition~\ref{prop:rate-functions}. Let $F(\beta,t)$ be defined as in (\ref{eq:nc-result-app}). If
\begin{equation}
\label{eq:sparse-rad-F}
\gamma^* f_\rho(t) \ge F(\beta,t) \quad \forall t \in (0,1)
\end{equation}
then $\Wish(\gamma,\beta,\cX_\rho) \contig \Wish(\gamma)$ for all $\gamma > \gamma^*$
\end{proposition}

\noindent In proving this we will not quite show that $f_\rho$ is a rate function for $\cX_\rho$ (but we will show that something arbitrarily-close is).

\begin{proof}

First we prove the result for the variant $\bar\cX_\rho$ of the sparse Rademacher prior where the number $K$ of nonzeros satisfies $K/n \to \rho$ in probability, and the nonzero entries are \iid $\pm 1/\sqrt{K}$. Suppose we have some $\gamma^*,\beta,\rho$ for which (\ref{eq:sparse-rad-F}) holds. At the expense of increasing $\gamma^*$ by an arbitrarily-small constant, we can find a small compact interval $[\rho_1,\rho_2]$ with $\rho$ in its interior such that (\ref{eq:sparse-rad-F}) holds on the entire interval. Condition on the $(1-o(1))$-probability event that $K/n$ lies in this interval. The function $f(t) = \min_{\hat\rho \in [\rho_1,\rho_2]} f_{\hat\rho}(t)$ is a valid rate function for $\bar\cX_\rho$ and furthermore has a local Chernoff bound. This follows from the sparse Rademacher tail bounds of \citet{noise-cond} (Propositions~4.8 and 4.9 of \citet{noise-cond}). The proof for $\bar\cX$ now follows from Theorem~\ref{thm:wish-nc} because $f$ satisfies (\ref{eq:nc-result-app}). The same lower bound holds for $\cX_\rho$ by Proposition~\ref{prop:wish-compare} (comparison of priors).

\end{proof}

% upper bound (using conditioning)
For the upper bound, we will apply Theorem~\ref{thm:wishart-mle}. However, instead of $\cX_\rho$ we consider the conditional distribution $\tilde\cX_\rho$ of $\cX_\rho$ given the $(1-o(1))$-probability event that the number $K$ of nonzero entries of $x$ satisfies $\rho n - \sqrt{n} \log n < K < \rho n + \sqrt{n} \log n$. The support size of $\tilde \cX_\rho$ is at most
$$2 \sqrt{n} \log n \cdot 2^{(\rho + o(1)) n} \binom{n}{(\rho \pm o(1)) n}.$$
By Stirling's approximation, $\log \binom{n}{\rho n} = n H(\rho) + o(n)$ where $H(\rho) = -\rho \log \rho - (1-\rho) \log (1-\rho)$ is the binary entropy. We can therefore apply Theorem~\ref{thm:wishart-mle} with any $c > 2^\rho \exp(H(\rho))$.

\bibliography{main}

\end{document}